\documentclass[reqno]{amsart}
\usepackage{a4wide}
\usepackage{graphics,graphicx}
\usepackage[english]{babel}
\usepackage{color,enumerate}
\usepackage{amsfonts,amsthm,amssymb,amsmath,mathrsfs}

\usepackage{xargs}                      
\usepackage[pdftex,dvipsnames]{xcolor}
\usepackage[colorinlistoftodos,prependcaption,textsize=tiny]{todonotes}
\newcommandx{\Amin}[2][1=]{\todo[inline,linecolor=red,backgroundcolor=red!10,bordercolor=red,#1]{\large #2}}

\newcommandx{\change}[2][1=]{\todo[linecolor=blue,backgroundcolor=blue!25,bordercolor=blue,#1]{#2}}

\usepackage{mathrsfs}
\usepackage[dvipsnames]{xcolor}
\usepackage[colorlinks, linkcolor=OliveGreen, citecolor=OliveGreen, urlcolor=OliveGreen, pagebackref, hypertexnames=false]{hyperref}

\newcommand{\mq}{\mathcal{Q}}

\newcommand{\abs}[1]{\left\vert#1\right\vert}

\newcommand{\R}{\mathbb{R}}

\newcommand{\N}{\mathbb{N}}
\newcommand{\ltr}{{L^3(\mathbb{R}^n)}}

\newcommand{\vr}{\varrho}

\numberwithin{equation}{section}

\newtheorem{theorem}{\quad Theorem}[section]
\newtheorem{lemma}{\quad Lemma}[section]

\newtheorem{remark}{\quad Remark}[section]

\newtheorem{proposition}{\quad Proposition}[section]

\newtheorem{definition}{Definition}[section]

\newcommand{\ee}{{\rm e}}
\newcommand{\ii}{{\rm i}}

\newcommand{\nh}{\mathscr{H}}

\newcommand{\lt}{{L^2(\mathbb{R}^n)}}

\newcommand{\sett}[1]{\left\{   #1   \right\}}
\newcommand{\norm}[1]{\left\|   #1   \right\|}

\newcommand{\paar}[1]{\left(   #1   \right)}

\newcommand{\blue}[1]{\textcolor{black}{#1}}

\def\rn{{\mathbb{R}^n}}
\def\N{\mathbb{N}}

\allowdisplaybreaks

\newcommand{\lam}{\lambda}

\newcommand{\rr}{\mathbb{R}}

\begin{document}
	
	\title[Nonlinear Schrödinger system with potential]{Studies on a system of nonlinear Schrödinger equations with potential and  quadratic interaction
		\footnotetext{2020 Mathematical subject classification: 35Q55, 82C10, 35J60, 35J20}
		\footnotetext{Keywords:  NLS system with potential, Normalized solution, Mountain-pass solutions, Blow-up.	}
	}   
	
	\maketitle
	
	\begin{center}
		
		{\bf Vicente  Alvarez and  Amin Esfahani} \\
		\vskip.2cm
		Department of Mathematics, Nazarbayev University, Astana 010000, Kazakhstan \\
		E-mail: {\tt vicente.alvarez@nu.edu.kz, amin.esfahani@nu.edu.kz.}
	\end{center}
	
	\vskip.3cm
	
	\noindent{{\bf Abstract.}
		In this work, we study the existence of various classes of standing waves for a nonlinear Schrödinger system with quadratic interaction, along with a harmonic or partially harmonic potential. We establish the existence of ground-state normalized solutions for this system, which serve as local minimizers of the associated functionals. To address the difficulties raised by the potential term, we employ profile decomposition and concentration-compactness principles. The absence of global energy minimizers in critical and supercritical cases leads us to focus on local energy minimizers. Positive results arise in scenarios of partial confinement, attributed to the spectral properties of the associated linear operators. Furthermore, we demonstrate the existence of a second normalized solution using Mountain-pass geometry, effectively navigating the difficulties posed by the nonlinear terms. We also explore the asymptotic behavior of local minimizers, revealing connections with unique eigenvectors of the linear operators. Additionally, we identify global and blow-up solutions over time under specific conditions, contributing new insights into the dynamics of the system.

	}
	
	

	\section{Introduction}
	
	In this paper, we study the prescribed solutions of
\begin{equation}\label{elliptic}
    \left\{\begin{array}{l}
        -\Delta \phi+\lambda_{1} \phi+V(x)\phi=\phi \psi, \\
        -\kappa \Delta \psi+\lambda_{2}\psi+V(x)\psi=\frac{1}{2}\phi^{2},\qquad 
        x=(x',x_n)\in\mathbb{R}^{n-1}\times\mathbb{R},
    \end{array}\right.    
\end{equation}
where $\kappa>0$, \( \lambda_1, \lambda_2 \in \mathbb{R} \), and $V(x)$ is a real-valued function. Such a system appears in the study of the following nonlinear Schrödinger equations with a trapping potential:
\begin{equation}\label{system1}
    \begin{cases}
        \ii \partial_{t}u_{1}+\Delta u_{1}-V(x)u_{1}+u_{2}\overline{u}_{1}=0,\\
        \ii \partial_{t}u_{2}+\kappa\Delta u_{2}-V(x)u_{2}+\frac{1}{2} u_{1}^2=0,\\
        (u_{1},u_{2})(x,0)=(u_{1}^{0},u_{2}^{0})(x),\qquad x=(x',x_n)\in\mathbb{R}^{n-1}\times\mathbb{R},
    \end{cases}
\end{equation}
where \( u_1(x,t) \) and \( u_2(x,t) \) are complex-valued functions on \( \mathbb{R}^n \times \mathbb{R} \), \( \Delta \) is the Laplace operator on \( \mathbb{R}^n \), \( V(x) \) is a real-valued potential, \( \kappa>0 \), and \( 1\leq n\leq5 \).

System \eqref{system1} in the two-dimensional spatial case was derived to model second-harmonic (SH) generation based on the scaled equations for the fundamental frequency (FF) and SH field amplitudes \( u_1 \) and \( u_2 \), along with an isotropic harmonic oscillator trapping potential \( V(x) = |x|^2 \). In the three-dimensional spatial setting with an axisymmetric potential \( V(x) = x_1^2 + x_2^2 \), the evolution of the FF and SH field amplitudes is described by \eqref{system1} (see \cite{Hang-Vladimir-Malomed,Sakaguchi-Malomed,Sakaguchi-Malomed-2013}). For a detailed derivation of this system from a physical perspective, readers are referred to \cite{sht}.

Disregarding the trapping potential \( V(x) \), where \eqref{system1} is introduced as a non-relativistic version of certain Klein-Gordon systems (see \cite{hot} and references therein), it was demonstrated in \cite{hot} through a contraction argument combined with Strichartz estimates that \eqref{system1} is well-posed in \( L^2 \times L^2 \) for \( 1 \leq n \leq 4 \) and in \( H^1(\mathbb{R}^n) \times H^1(\mathbb{R}^n) \) for \( 1 \leq n \leq 6 \). Furthermore, global well-posedness was established in \( L^2(\mathbb{R}^n) \times L^2(\mathbb{R}^n) \) for the case \( n = 4 \). Similarly, global existence was proven in \( H^1(\mathbb{R}^n) \times H^1(\mathbb{R}^n) \) for \( 1 \leq n \leq 3 \), attributed to the conservation of mass and energy. Additionally, when the initial data belongs to \( H^1(\mathbb{R}^5) \times H^1(\mathbb{R}^5) \), a condition for global existence was identified, relying on the comparison of the size of the initial data with respect to the associated ground states (see \cite{hot}). The exploration of global existence and energy scattering has been undertaken in previous works such as \cite{MX,WY}, while investigations into finite-time blow-up phenomena have been demonstrated in \cite{HIN,inu}. The existence of ground states (energy minimizers) for \( 2\leq n\leq 5 \) was established in \cite{ZhaoZhaoShi}. This paper also addresses the uniqueness of positive solutions, continuous dependence on parameters, investigates the asymptotic behavior of ground states, and proves the existence of multi-pulse solutions with specific symmetry.

For further results on similar systems of nonlinear Schrödinger equations with quadratic interaction, interested readers may refer to \cite{AD}.

In the presence of a potential \( V(x) \not\equiv 0 \), some results are known for a single nonlinear Schrödinger equation given by
\begin{equation}\label{single-pot-nls}
    \ii \partial_{t}u + \Delta u - V(x)u + |u|^p u = 0,\quad p > 1.
\end{equation}
Fujiwara \cite{fu2} investigated \eqref{single-pot-nls}, incorporating a general real-valued potential function \( V(x) \), a study later explored in \cite{chen-zhang}. In situations where \( |D^\alpha V(x)| \) is bounded for all \( \alpha \geq 2 \), Fujiwara established the smoothness of the Schrödinger kernel for potentials exhibiting quadratic growth. Furthermore, Yajima \cite{Yajima} demonstrated that, for super-quadratic potentials, the Schrödinger kernel lacks continuity everywhere and is not of class \( C^1 \). Additionally, Oh \cite{oh} emphasized that quadratic potentials represent the highest order of potentials for ensuring the local well-posedness of the equation. Consequently, the harmonic potential \( V_1(x) = |x|^2 \) emerges as the critical potential for the local existence of the Cauchy problem. See also the results of \cite{zhang2000, zhang2005, zhang2019} and references therein on sharp conditions for global existence and finite-time blow-up of solutions, as well as the instability of standing waves.
We should note that the above results extend the local well-posedness in the energy space in \cite{caz}.

When the partial harmonic oscillator \( V_2 \) is considered in \eqref{single-pot-nls}, the existence of orbitally stable ground states to \eqref{single-pot-nls} in the case \( n=3 \) with \( L^2 \)-supercritical nonlinearity \( 7/3 \leq p < 5 \) was proven in \cite{Bell}. Indeed, since the energy functional is not bounded from below, these ground states are local energy minimizers on the sphere
\[
\{u \in H^1(\mathbb{R}^3),\;\|u\|_{L^2}=r\},  
\]
where \( H(\mathbb{R}^3) = \{u \in H^1(\mathbb{R}^3),\; u \in L^2(V_2 )\} \).
In the scenario where \( V = V_2 \), a key observation arises: the last variable in \( V_2 \) remains unconstrained. Consequently, the one-line solitons cannot be straightforwardly extended across the entire space while preserving their finite energy property.

	\medskip
	
	In this paper, first, our focus is on the study of \eqref{system1} with the harmonic potential \( V(x) = V_1(x) = |x|^2 \) and the partial harmonic potential (confinement) \( V(x) = V_2(x) = |x'|^2 \). Initially, we analyze the existence of different classes of standing waves for \eqref{system1} and subsequently investigate the associated initial value problem.

	It is worth noting that the \( L^2 \)-norm of the solution often represents the number of particles in each component for Bose-Einstein condensates or the power supply in nonlinear optics. Thus, given this physical relevance, it is of great interest to seek (standing wave) solutions to \eqref{system1} (see \eqref{elliptic}) with a prescribed \( L^2 \)-type norm. Indeed, we aim to find the $L^2$-type prescribed (normalized) solutions of \eqref{elliptic}.

	
	To describe our results, we define the natural energy space for \eqref{system1}:
\begin{equation}\label{space}
    H = \{(\phi, \psi) \in \dot{H}^1(\mathbb{R}^n) \times \dot{H}^1(\mathbb{R}^n) : V(x)\phi, V(x)\psi \in L^2\},
\end{equation}
equipped with the norm
\[
\norm{(\phi, \psi)}_H^2 = \norm{\nabla \phi}_{L^2}^2 + \norm{\nabla \psi}_{L^2}^2 + \norm{\phi}_{L^2}^2 + \norm{\psi}_{L^2}^2 + \int_{\mathbb{R}^n} V(x)(|\phi|^2 + |\psi|^2)\,d x.
\]

To find the \( L^2(\mathbb{R}^n) \times L^2(\mathbb{R}^n) \)-normalized  solutions of \eqref{elliptic}, we employ variational methods. Indeed, to find the normalized solutions of \eqref{elliptic}, we define the functional
\begin{equation*}\label{I}
    I(\phi, \psi) := \frac{1}{2}\left(\|\nabla \phi\|_{L^{2}}^{2} + \kappa\|\nabla \psi\|_{L^{2}}^{2}\right) + \frac{1}{2}\int_{\mathbb{R}^n} V(x)(|\phi|^2 + |\psi|^2)  - \frac{1}{2}K(\phi, \psi),
\end{equation*}
with
\begin{equation}\label{K}
    K(u_{1}, u_{2}) := \Re\int_{\mathbb{R}^n} u_{1}^2 \bar{u}_{2} \,d x,
\end{equation}
and consider the minimization problem
\begin{equation}\label{min1}
d_{\mu_{1}, \mu_{2}} := \inf_{S(\mu_{1}, \mu_{2})} I(\phi, \psi),
\end{equation}
where
\begin{equation}\label{product-ball}
S(\mu_{1}, \mu_{2}) = \left\{(\phi, \psi) \in H : \int_{\mathbb{R}^n} |\phi|^2 \,d x  = \mu_{1}, \int_{\mathbb{R}^n} |\psi|^2 \,d x  = \mu_{2}\right\}
\end{equation}
for given \( \mu_{1}, \mu_{2} > 0 \). Referring to the case \( V(x) \equiv 0 \), we can study \eqref{system1} in the \( L^2 \)-subcritical case \( n \leq 3 \), \( L^2 \)-critical \( n = 4 \), and \( L^2 \)-supercritical \( n = 5 \).

	In the case of the harmonic potential \( V = V_1 \), when \( n \leq 3 \), it is standard by the compact embedding \( H \hookrightarrow L^q(\mathbb{R}^n) \times L^q(\mathbb{R}^n) \) with \( 2 \leq q < q^* \) (see Remark \ref{remarkemb}) that the minimizers of \eqref{min1} exist; however, when \( n \geq 4 \), \( I \) is not bounded from below, so we focus on finding local minimizers. See Theorem \ref{0theorem} below. Indeed, for any given \( \chi > 0 \), they are minimizers of
\begin{equation}\label{min2}
d_{\mu_{1}, \mu_{2}}^{\chi} := \inf_{S(\mu_{1},\mu_{2}) \cap B(\chi)} I(\phi, \psi),
\end{equation}
where
\[
B(\chi) = \left\{(\phi, \psi) \in H : \norm{(\phi, \psi)}_{\dot{H}}^{2} \leq \chi/\epsilon_{0}\right\}, \quad \epsilon_{0} = \min\{1, \kappa\},
\]
and
\[
\norm{(\phi, \psi)}_{\dot{H}}^2 := \norm{\nabla \phi}_{L^2}^2 + \norm{\nabla \psi}_{L^2}^2 + \int_{\mathbb{R}^n} V(x)(|\phi|^2 + |\psi|^2) .
\]

When \( V_1 \) is replaced with the partial harmonic oscillator \( V_2 \), the analysis becomes more intricate and the aforementioned embedding is not compact due to the presence of the partial potential. To address the difficulties raised by the potential term, we focus on the ellipse
\begin{equation}\label{ellipse}
S_\vr(\mu) = \{(\phi, \psi) \in H, \; \|\phi\|_{L^2}^2 + 2\vr\|\psi\|_{L^2}^2 = \mu\},
\end{equation}
and employ a profile decomposition (see Proposition \ref{profile}) combined with the spectral properties of the associated linear operator of \eqref{elliptic} to show the minimizers of
\begin{equation}\label{min}
d_{\mu} := \inf_{S_\vr(\mu)} I(\blue{\phi},\blue{\psi}).
\end{equation}
We generalize this result to \( S(\mu_1, \mu_2) \) by using the concentration-compactness principle and show the existence of a minimizer of
\begin{equation}\label{min5}
d_{\mu_1, \mu_2} = \inf_{S(\mu_1, \mu_2)} I(\blue{\phi},\blue{\psi}).
\end{equation}

\blue{\begin{theorem}\label{theorem2-sub}
Let \( 1 \leq n \leq 3 \).
\begin{enumerate}[(i)] 
	\item If \( \mu, \vr > 0 \), then there exists \( (u,v) \in S_\vr(\mu) \) such that \( d_{\mu} = I(u,v) \).
	\item If \( \mu_1, \mu_2 > 0 \), then
	$$
	\mathcal{D}_{\mu_{1}, \mu_{2}} := \{(u, v) \in S(\mu_1, \mu_2) : I(u, v) = d_{\mu_1, \mu_2}\} \neq \emptyset.
	$$ 
\end{enumerate}
\end{theorem}}
	
	In critical and supercritical cases, the absence of global energy minimizers necessitates focusing on local energy minimizers. The characteristics of \eqref{system1} diverge from those in \cite{alji,deng,chen-wei, Guo, Wang-Zhou}, even when \( V = V_1 \). However, following \cite{Bell} in the case of partial confinement, positive results stem from the spectral properties of the linear operator associated with the elliptic-type system \eqref{elliptic}. 

Indeed, for \( \chi > 0 \) we define  
\begin{equation}\label{min3}
d_{\mu_{1}, \mu_{2}}^{\chi} := \inf_{S(\mu_{1},\mu_{2}) \cap B(\chi)} I(\blue{\phi},\blue{\psi}),
\end{equation}
and we will show the following main result.

\blue{\begin{theorem}\label{theorem1} Let \( n \geq 4 \). Given \( \mu_1, \mu_2 > 0 \), for any \( \chi > 0 \), there exists a constant \( \mu^{\star} := \mu^{\star}(\chi,\kappa) > 0 \) such that for any \( \mu_1, \mu_2 > 0 \) satisfying \( \mu_1 + \mu_2 < \mu^{\star} \),
\[
\mathcal{D}_{\mu_1, \mu_2}^\chi
:= \{(u, v) \in S(\mu_1, \mu_2) \cap B(\chi) : I(u, v) = d_{\mu_1, \mu_2}^\chi \} \neq \emptyset.
\]
\end{theorem}}

\blue{Moreover, we will study the asymptotic behavior of the solutions found in \( \mathcal{D}_{\mu_1, \mu_2}^\chi \).
\begin{theorem}\label{2theorem}
Every minimizer obtained in Theorem \ref{theorem1} (that is, in principle, \( \mathbb{C} \)-valued) is of the form \( (e^{i\theta_1} f_1(x^{\prime}, x_{n}), e^{i\theta_2} f_2(x^{\prime}, x_{n})) \), where \( f_1, f_2 \) are positive real-valued minimizers and \( \theta_1, \theta_2 \in \mathbb{R} \). Let \( (u_{\mu_1}, v_{\mu_2}) \in \mathcal{D}_{\mu_1, \mu_2}^\chi \), then there exists \( (\lambda_{\mu_1}, \lambda_{\mu_2}) \in \mathbb{R}^2 \) such that \( (u_{\mu_1}, v_{\mu_2}, \lambda_{\mu_1}, \lambda_{\mu_2}) \) is a pair of weak solutions to the problem \eqref{elliptic} with the estimates:
$$
\begin{gathered}
\lambda_{\mu_1} + \lambda_{\mu_2} < l_0 \max\{1, \sqrt{\kappa}\} \left( 1 + \max \left\{ \frac{\mu_2}{\mu_1}, \frac{\mu_1}{\mu_2} \right\} \right), \\
\lambda_{\mu_1} + \lambda_{\mu_2} \geq l_0 \min\{1, \sqrt{\kappa}\} \left( 1 + \min \left\{ \frac{\mu_2}{\mu_1}, \frac{\mu_1}{\mu_2} \right\} \right) \quad \text{as} \quad (\mu_1, \mu_2) \rightarrow (0, 0)
\end{gathered}
$$
and
\begin{equation}\label{estimat1}
\sup_{(u, v) \in \mathcal{D}_{\mu_1, \mu_2}^\chi} \norm{(u,v) - (\varphi_0 \Psi_0, \psi_0 \Phi_0)}_H^2 = O(\mu_1 + \mu_2),
\end{equation}
where \( \Psi_0(x') \) is the unique normalized positive eigenvector of the quantum harmonic oscillator \( -\Delta_{x'} + V_2 \) and \( \Phi_0(x') \) is the unique normalized positive eigenvector of the quantum harmonic oscillator \( -\kappa \Delta_{x'} + V_2 \),
$$
\varphi_0(x_n) = \int_{\mathbb{R}^{n-1}} u_{\mu_1}(x) \Psi_0(x') \,d x', \quad \psi_0(x_n) = \int_{\mathbb{R}^{n-1}} v_{\mu_2}(x) \Phi_0(x') \,d x'.
$$
Furthermore, for some \( (s, l) \in \mathbb{R}^2 \), \( u_{\mu_1}(x', x_n - s) \) and \( v_{\mu_2}(x', x_n - l) \) are radially symmetric and nonincreasing almost everywhere in \( x' \) for any \( x_n \).
In addition, for sufficiently small \( \mu_1, \mu_2 > 0 \), the minimizers in \( \mathcal{D}_{\mu_1, \mu_2}^\chi \) are least energy normalized solutions to the problem \eqref{elliptic}.
\end{theorem}}

	\medskip
	Moreover, we establish a direct connection between the previously obtained minimizer and a minimizer for a reduced one-dimensional system \eqref{4elliptic} below.
\blue{
\begin{theorem}\label{theorem2.27}
Let \( (\mathcal{D}_1, \mathcal{D}_2) \) be a minimizer for the problem \( \mathcal{J}_{\mu_1, \mu_2} \), where
\begin{equation}\label{NNew-min-as}
	    \mathcal{J}_{\mu_1,\mu_2}=\inf_{(u,v)\in H} \left\{J(u,v): \|u\|_{\lt}^2=\mu_1,\, \, \|v\|_{\lt}^2=\mu_2 \, \, \text { and }\|(u,v)\|_{\dot{H}_{x^{\prime}}}^2 \leq \chi\right\},
	\end{equation}
	 		$$
		J(u,v)=\|(u,v)\|_{\dot{H}_{x^{\prime}}}^2+\frac{1}{2}\left\|\partial_{x_n} u\right\|_{\lt}^2+\frac{\kappa}{2}\left\|\partial_{x_n} v\right\|_{\lt}^2-\frac{1}{2}K(u,v),
		$$
		and 
		$$
		\|(u,v)\|_{\dot{H}_{x^{\prime}}}^2=\int_{\rn}\left(\left|\nabla_{x^{\prime}} u\right|^2+V_2|u|^2- L_0 |u|^2 +\kappa\left|\nabla_{x^{\prime}} v\right|^2+V_2|v|^2- M_0 |v|^2\right)\,d x.
		$$Then,  $$
\norm{(\mathcal{D}_1, \mathcal{D}_2) - (\mathcal{D}_{\infty}^{1}(x_n) \Psi_0(x'), \mathcal{D}_{\infty}^{2}(x_n) \Phi_0(x'))}_{H} \leq O(\mu_1 + \mu_2),
$$    
where \( (\mathcal{D}_\infty^1, \mathcal{D}_\infty^2) \) is the unique positive solution of
\begin{equation}\label{4elliptic}
\left\{
\begin{array}{l}
-\partial_{x_n}^2 \phi - s_n^1 \phi \psi = -\lambda_{\infty}^1 \phi, \\
-\kappa \partial_{x_n}^2 \psi - \frac{s_n^2}{2} \phi^2 = -\lambda_{\infty}^2 \psi,
\end{array}
\right.    
\end{equation}
where \( s_n^1 = \pi^{-\frac{2n+1}{2}} \left( \frac{2\kappa}{2\kappa+1} \right)^{\frac{n-1}{2}} \), \( s_n^2 = \pi^{-\frac{2n+1}{2}} \left( \frac{2\kappa}{3} \right)^{\frac{n-1}{2}} \), and \( \lambda_\infty^2 = \kappa \lambda_\infty^1 \).
\end{theorem}
 }
\blue {To demonstrate this result, initially, we   introduce a minimization problem associated with the one-dimensional system \eqref{4elliptic}, by means of which we guarantee the existence of a solution to a positive symmetric decreasing ground state, which in turn is the only minimizer for such a problem. Subsequently, we consider an \( n \)-dimensional system deduced from \eqref{elliptic} by means of simple translations of the Lagrangian parameters, through which we can establish a minimization problem determined by   \( \mathcal{J}_{\mu_1, \mu_2} \), from where, using arguments similar to those applied in Theorem \ref{2theorem} except for some variations, we can obtain a minimizing sequence for \( \mathcal{J}_{\mu_1, \mu_2} \). We then use a spectral representation for the solutions of the minimization problem in terms of the eigenvectors associated with the quantum harmonic oscillators \( -\Delta_{x'} + V_2 \) and \( -\kappa \Delta_{x'} + V_2 \). We also indicate the \( (\Psi_0, \Phi_0) \)-directional component for each solution, where \( (\Psi_0, \Phi_0) \) represent the lowest eigenvector of the aforementioned operators. This approach enables us to obtain smooth estimates for the minimizer solutions and their projection onto the lowest eigenspace, through which we demonstrate a direct relationship between the \( n \)-dimensional and the one-dimensional minimization problems, as well as a convergence relation between the Lagrangian multipliers of both systems. Thus, by finally following \cite{hong}, and combining the aforementioned estimates and applying some direct calculations, we can obtain a convergence relationship between the minimizing sequence of the \( n \)-dimensional problem and the minimizer of the one-dimensional problem, providing a clearer characterization that allows us to show how the one-dimensional ground state is derived from the \( n \)-dimensional energy minimizer with a precise rate of convergence. }
\medskip

	\blue{ Furthermore, we explore the asymptotic behavior of local minimizers as the radius of existence varies, particularly in terms of eigenvectors of \( -\Delta + V_2(x) \), following the approach in \cite{Bell}.
	\begin{theorem}\label{4theorem}
    Let \(V = V_2\), \(n = 4, 5\) and \(\lambda_N := \lambda_1 =  \lambda_2/\kappa \). Then for \(N > 0\) sufficiently small, \eqref{elliptic} has a second positive normalized solution \((\tilde{u}_{N,2}, \tilde{v}_{N,2})\) on \(\blue{S_{\kappa/2}(N^2)}\), which is also a mountain-pass solution, with a Lagrange multiplier
    $$
    \lambda_{N, 2} = \left( 1 + o_{N}(1) \right) \left[ \frac{(n+6) K(w_{\infty}^1, w_{\infty}^2)}{4 N^2} \right]^{\frac{2}{n-4}} \to \infty \quad \text{as} \quad N \to 0,
    $$
    where \((w_{\infty}^1, w_{\infty}^2)\) is the unique (up to translations) positive solution of the following equation:
    \begin{equation}\label{systemq}
        \left\{
        \begin{array}{l}
            -\Delta u + u = uv, \\
            - \Delta v + v = \frac{1}{2} u^2.
        \end{array}
        \right.
    \end{equation}
\end{theorem}}

To establish a new normalized solution for the system \eqref{elliptic} that differs from those obtained in previous theorems, and motivated by the conjecture presented in \cite{Bell}, we employ a mountain-pass argument to derive this solution. Indeed, we consider the system \eqref{elliptic} subject to the condition
$$
\int_{\mathbb{R}^n} |u|^2\,d x + \kappa \int_{\mathbb{R}^n} |v|^2\,d x = N^2,
$$
and we consider the following scaled system of the system \eqref{elliptic} given by
\begin{equation}\label{scaled}
    \left\{
    \begin{array}{l}
        -\Delta w_1 + w_1 + t^{-2} V(x) w_1 = w_1 w_2, \\
        -\Delta w_2 + w_2 + t^{-2} V(x) w_2 = \frac{1}{2} w_1^2.
    \end{array}
    \right.
\end{equation}
This system has a ground-state solution \((w_1^t, w_2^t)\) for all \( t > 0 \) that converges to \((w_{\infty}^1, w_{\infty}^2)\), a unique solution of the system \eqref{systemq}, for sufficiently large values of \( t \), satisfying certain control estimates for such solutions. Subsequently, we introduce the following scaling map:
\begin{equation}\label{function}
    (N, t) \mapsto N^2 - t^{\frac{4-n}{2}} \left( \frac{n+6}{4} K(w_1^t, w_2^t) - 2 t^{-2} \int_{\mathbb{R}^n} V(x) \left( |w_1^t|^2 + |w_2^t|^2 \right) \,d x \right).
\end{equation}
Next, by using the uniqueness, non-degeneracy of the limit function \((w_\infty^1, w_\infty^2)\) (when \(\lambda_2 = \lambda_1 \kappa\)) and the fact that the curve \((w_1^t, w_2^t)\) remains continuous for sufficiently large \(t > 0\) within an appropriate function space, we conclude that \eqref{function} has a zero \(N = N_t\). This helps us to find a positive normalized solution for \eqref{scaled} on a certain (\(t\)-dependent) ellipse (see \eqref{systemloc}) with a positive Lagrange multiplier \(t > 0\). Finally, we proceed by applying the mountain-pass argument to obtain the result.

\medskip
 
Having studied the solutions of \eqref{elliptic}, which are closely related to system \eqref{system1} in the first part of this paper, it is natural to consider the long-time behavior of the solutions to the Cauchy problem \eqref{system1} in order to explore possible connections with the standing waves.

The local well-posedness results for \eqref{system1} can be obtained similarly to those established for the scalar Schrödinger equation (see \cite{caz}). Additionally, see \cite{tao, tsu} for further references. We will focus on studying the existence of local solutions in the \( H^1(\mathbb{R}^n) \times H^1(\mathbb{R}^n) \) subcritical and critical regimes, specifically for dimensions \( 1 \leq n \leq 6 \). It is important to note that local-in-time Strichartz estimates are sufficient to establish local well-posedness in the energy space.

We also recall some preliminary results obtained by Fujiwara \cite{fu}, which characterize the kernels considering \( \blue{\nh} = -\Delta + V \), where $V$ is bounded below and for each $|\ell|\geq2$, $|D^\ell V|$ is bounded. For a more comprehensive study, one can refer to \cite{ancarsil, oh,Rem}. Let $U(t)$ be the propagator of
$\nh$ and $k(t, x, y)$ its Schwartz kernel (see \cite[Corollary 2.7]{fu2}).

	\begin{proposition}[\cite{fu,fu2}] 
Under the above assumption, for sufficiently small \( \delta > 0 \) and \( 0 < |t| \leqslant \delta \),

(i) the classical action function \( S(t, x, y) \) is globally (uniquely) defined and smooth,

(ii) the kernel \( k(t, x, y) \) has the form
$$
k(t, x, y)=\left(\frac{-\ii}{4 \pi t}\right)^{n / 2} a(t, x, y) e^{\ii S(t, x, y)},
$$
where \( a \) is a bounded continuous function of \( t, x \) and \( y \).
\end{proposition}

The fundamental solution \( U(t)\) for the linear problem \eqref{single-pot-nls} is given explicitly for \( V=V_1 \) by the Mehler formula:
\[
U(t)f(x)=
\left(\frac{1}{2\ii\pi\sin(\sqrt{2}t)}\right)^{\frac n2}\int_\rn
\exp\left(\frac{\ii}{\sin(\sqrt{2}t)}\left(\frac{x^2+y^2}{2}\cos(\sqrt{2}t)-x\cdot y\right)\right)f(y)
\,d y.
\]
However, the singularity may prevent the existence of global in-time Strichartz estimates. The propagator obeys the (local in time) dispersive estimates (\cite{Rem-1, fu, fu2}):
\[
\begin{split}
	&\|U(t)f\|_{L^\infty(\rn)}\lesssim 
	|t|^{-\frac n2}\|f\|_{L^2(\rn)},\quad |t|\leq\delta,\\
	&\|U(t)f\|_{L^q([-T,T]:L^r(\rn))}\leq C\|f\|_{L^2(\rn)},\\
	&\left\|\int_0^tU(t-\tau)F(\tau)\,d\tau\right\|_{L^q([-T,T]:L^r(\rn))}\leq C\|F\|_{L^{p_1'}([-T,T]:L^{r_1'}(\rn))}
\end{split}
\]
for some \( T > 0 \), where \( C = C(q,n,T) \), and \( (q, r) \) and \( (p_1, r_1) \) are \( n \)-admissible pairs, that is, satisfying \( \frac{2}{q} = \frac{n}{2} - \frac{n}{r} \) with \( 2 \leq q, r \leq \infty \) and \( q > 2 \) if \( n = 2 \). The prime in the pair indicates the conjugate. See \cite{kvz} for the energy-critical case.

The above-mentioned results are also valid for the partial harmonic potential \( V = V_2 \); see \cite{ardir, jian-li-luo}, \cite[Theorem 1.8]{Rem-2} and \cite[Theorem 9.2.6]{caz}. However, in \cite{ancarsil}, the authors proved the scattering by obtaining the global Strichartz estimates:
\[
\begin{split}
	&\norm{U(t)f}_{\ell_\gamma^pL^q(I_\gamma:L^r(\rn))}\leq C\|f\|_{L^2(\rn)},\\
	&\norm{\int_\rr U(t)F(t)\,d t}_{L^2(\rn)} \leq \|F\|_{\ell_\gamma^{p'}L^{q'}(I_\gamma:L^{r'}(\rn))}, \\
	&\left\|\int_0^t U(t-\tau)F(\tau)\,d\tau\right\|_{\ell_\gamma^{p}L^q(I_\gamma:L^r(\rn))} \leq C\|F\|_{\ell_\gamma^{p_1'} L^{q_1'}(I_\gamma:L^{r_1'}(\rn))},
\end{split}
\]
where \( I_\gamma = \pi[\gamma-1, \gamma+1) \) with \( \gamma \in \mathbb{Z} \), \( (p,r) \) and \( (p_1, r_1) \) are \( n \)-admissible, while \( (q,r) \) and \( (q_1,r_1) \) are \( (n-1) \)-admissible.

Hence, mimicking the above-mentioned arguments, we can obtain the local existence of solutions of \eqref{system1}.
\begin{theorem}\label{theoremsub}
Let \( 1 \leq n \leq 6 \) and \( V = V_j \), \( j = 1,2 \). Then the Cauchy problem \eqref{system1} is locally well-posed in \( H \) for \( t \in [-T,T] \), where \( T \) depends on the \( H \)-norm of initial data. Furthermore, the quantities \( E \) and \( \mq \) are conserved under the flow of \eqref{system1}, where
\[
\mq(u_1,u_2) = \|u_1\|_{L^2}^2 + 2 \|u_2\|_{L^2}^2
\]
and
\begin{equation}\label{E}
    \begin{split}
    E(u_1(t),u_2(t)) &= P(u_1,u_2) - K(u_1,u_2)
    \end{split} 
\end{equation}
with
\[
\begin{split}
    P(u_1,u_2) &= \|\nabla u_1\|_{L^2}^2 + \kappa\|\nabla u_2\|_{L^2}^2 + \int_\rn V(x) (|u_1|^2 + |u_2|^2) \,d x.
\end{split}
\]
\end{theorem}
	\blue{
Next, we establish that local solutions can be extended globally in time if the initial data is smaller than the radial standing waves of \eqref{system1} with \( V \equiv 0 \).
\begin{theorem}\label{thm410}
Let \( 1 \leq n \leq 5 \), $H$ be defined by \eqref{space}, and \( (Q_1, Q_2) \) be the positive  radially  symmetric solution of
\begin{equation}\label{systemq2}
\left\{
\begin{aligned}
-\Delta Q_1 + Q_1 &= Q_1 Q_2, \\
-\kappa\Delta Q_2 + 2Q_2 &= \frac{1}{2} Q_1^2.
\end{aligned}
\right.
\end{equation} 
If \( (u_1^0, u_2^0) \in H \) satisfies  
\begin{equation}\label{GINQ}
\mq(u_1^0, u_2^0) < \frac{n}{4} \mq(Q_1, Q_2),
\end{equation} 
then the unique solution \( (u_1, u_2)\in C([0,T]:H) \) of the Cauchy problem \eqref{system1} with initial data \( (u_1^0, u_2^0) \) can be extended globally in time.
\end{theorem}
\medskip
In the supercritical and critical cases, blow-up occurs if the initial data is close to the aforementioned standing waves. Additionally, through meticulous analysis aided by localized Morawetz estimates, we ascertain conditions, contingent on the ground states of \eqref{system1}, under which local solutions of \eqref{system1} experience blow-up in both temporal directions (refer to Theorem \ref{thm4.8}).
In the following theorem, we show a blow-up result for the Cauchy problem \eqref{system1}.
\begin{theorem}\label{Theoremblow1}
Let \(n\geq 4  \) and \( \kappa = 1/2 \). Assume that \( (Q_1, Q_2) \) is the positive  radially  symmetric solution of \eqref{systemq2}.
Then for any \( \epsilon > 0 \), there exists \( (u_1^0, u_2^0) \in H \) satisfying
\[
\mq(u_1^0, u_2^0) = \mq(Q_1, Q_2) + \epsilon
\]
such that the unique solution \( (u_1, u_2)\in C([0,T]:H) \) of  \eqref{system1} with the initial data \( (u_1^0, u_2^0) \) blows up in both time directions.
\end{theorem}
}
\vspace{3mm}
In the subsequent sections of the paper, we first demonstrate the existence of standing waves of \eqref{system1} when \( V \) is the Harmonic potential in Section \ref{standing waves}. We then address the case of partial confinement. This section concludes with the proof of the existence of the second (Mountain pass) solution. With the existence results established, we proceed to study the Cauchy Problem \eqref{system1}. Finally, in Section \ref{localw}, we will prove Theorems \ref{thm410} and \ref{Theoremblow1}.

	\section{Existence of standing waves}\label{standing waves}
	\subsection{Preliminaries}
	
	Let us start by defining the concept of a standing wave solution for \eqref{system1}. A standing wave solution is a solution of \eqref{system1} in the form
$$
(u, v)=\left(\ee^{\ii \lambda_{1} t} \phi(x), \ee^{\ii \lambda_{2} t} \psi(x)\right),
$$
where $\omega>0$ is a real parameter and $\phi, \psi$ are real-valued functions, which may depend on $\omega$, with a suitable decay at infinity. By substituting this approach into \eqref{system1}, we obtain the following elliptic system \eqref{elliptic} \blue{provided 	$\lam_2=2\lam_1$. However, since $\kappa$ is arbitrary, we ignore this condition, unless it is necessary, to study the general case of \eqref{elliptic}.}

Our goal is to establish the existence of solutions for the system \eqref{elliptic}. To achieve this, we first need to define what we will consider as a solution of \eqref{elliptic}.

	\begin{definition}
    A pair $(\phi, \psi) \in H^{1}\left(\rn\right) \times H^{1}\left(\rn\right)$ is called a solution (or a weak solution) of \eqref{elliptic} if
    $$
    \begin{aligned}
        & \int_{\rn}(\nabla \phi \cdot \nabla f + \lambda_{1}\phi f + V(x)\phi f) \,d x = \int_{\rn} \phi \psi f \,d x, \\
        & \int_{\rn}(\kappa \nabla \psi \cdot \nabla g + \lambda_{2} \psi g + V(x)\psi g) \,d x = \frac{1}{2} \int_{\rn} \phi^{2} g \,d x,
    \end{aligned}
    $$
    for any $f, g \in C_{0}^{\infty}\left(\rn\right)$.   
\end{definition}

According to the standard elliptic regularity theory (see, for instance, \cite{caz,ZhaoZhaoShi}), weak solutions are indeed smooth and adhere to \eqref{elliptic} via the standard methods. In fact, we can assert the following proposition.

\begin{proposition}\label{regularity}
    Let \blue{$1 \leq n \leq 5$}. Consider $(\Phi_{1},\Phi_{2}) \in H^{1}(\rn) \times H^{1}(\rn)$ as a solution of \eqref{elliptic}. Then, for any $j=1,2$, 
    $\Phi_{j} \in W^{3, p}(\rn)$ for $2 \leq p < \infty$. In particular, $\Phi_j \in C^{2}(\rn)$, and $\sum_{k=1}^{2}\left|D^{\beta} \Phi_{k}(x)\right|\stackrel{|x| \rightarrow \infty}{\longrightarrow} 0$ for all $|\beta| \leq 2$.
\end{proposition}

\begin{proof}
    The proof of this result is an adaptation of Theorem 8.1.1 in \cite{caz}. Indeed, let $\xi$ be a smooth function satisfying $\xi(x) = 1$ for $|x| \leq 1$ and $\xi(x) = 0$ for $|x| \geq 2$. Given $R > 0$, we denote $\xi_{R}(x) = \xi\left(\frac{x}{R}\right)$. Hence, a direct computation gives
    $$
    -\Delta(\xi_{R}\Phi_{1}) + \lambda_{1}\xi_{R}\Phi_{1} = (-\Delta \xi_{R} + \xi_{R}\Phi_{2} + V(x)\xi_{R})\Phi_{1} - 2\nabla \xi_{R} \cdot \nabla \Phi_{1}.
    $$
    Since $\xi_{R}$ is a smooth, bounded function supported in a ball and $\nabla \xi_{R}$ is smooth and compactly supported, we have 
    $$(-\Delta \xi_{R} + \xi_{R}\Phi_{2} + V(x)\xi_{R})\Phi_{1} - 2\nabla \xi_{R} \cdot \nabla \Phi_{1} \in L^{\frac{p}{2}}(\rn).$$
   Therefore, it follows from the standard bootstrap argument (see Theorem 8.1.1 in \cite{caz}) that $\xi_{R}\Phi_{1} \in W^{3,p}(\rn)$ for all $2 \leq p < \infty$. In particular, $\xi_{R}\Phi_{1} \in C^{2,\eta}(\rn)$ for all $0 < \eta < 1.$ Since $\xi_{R}\Phi_{1} = \Phi_{1}$
    on $B(0,R)$ for any $R > 0$, it follows that $\Phi_{1} \in C^{2}(\rn)$ and $|D^{\beta}\Phi_{1}| \rightarrow 0$ as $|x| \rightarrow \infty$ for all $|\beta| \leq 2.$ Similarly, we proceed in the case of $\Phi_{2}$ using the second equation in \eqref{elliptic}.
\end{proof}

	\subsection{Existence of standing waves: Harmonic potential}
	
Here, we look for $L^2$-solutions of \eqref{elliptic}. Our goal is directed towards ensuring the existence of critical points of $I$ on the constraint $S(\mu_{1},\mu_{2})$, specifically by ensuring the minimization of problem \eqref{min1}.

It is standard that the minimizers of $d_{\mu_{1}, \mu_{2}}$ are critical points of $\left.I\right|_{S\left(\mu_{1},\mu_{2}\right)}$ as well as normalized solutions to problem \eqref{elliptic}.

In fact, by using the ideas established in \cite{zhang2000}, we can guarantee this result in the $L^2$-subcritical case.

\begin{lemma}[Heisenberg's inequality]\label{uncert}
    If $\psi \in H^1\left(\rn\right) \cap L^2(|x|^2\,d x)$, then  
    $$
    \int_{\rn}|\psi|^2 \,d x \leq \frac{2}{n} \left(\int_{\rn}|\nabla \psi|^2 \,d x\right)^{\frac{1}{2}}\left(\int_{\rn}|x|^2|\psi|^2 \,d x\right)^{\frac{1}{2}}.
    $$
\end{lemma} 
\medskip
 
 By using Lemma \ref{uncert}, it follows that \blue{from \eqref{K}},

\begin{equation}\label{ineqK}
    \begin{split}
        K(u,v)
        &\leq \frac{1}{2} \norm{u}_\ltr ^2\norm{v}_\ltr 
        \leq \frac{1}{2}C_\circ^3
        \norm{\nabla u}_\lt^{\frac{n}{3}}\norm{u}_\lt^{2-\frac{n}{3}}
        \norm{\nabla v}_\lt^{\frac{n}{6}}\norm{v}_\lt^{1-\frac{n}{6}}\\
        &\leq
        \frac{1}{2}C_\circ^3\left(\frac{2}{n}\right)^{\frac{6-n}{4}} 
        \norm{\nabla u}_\lt^{1+\frac{n}{6}}\norm{xu}_\lt^{1-\frac{n}{6}}
        \norm{\nabla v}_\lt^{\frac{6+n}{12}}\norm{xv}_\lt^{\frac{6-n}{12}}\\
        &\leq C
        \left(\norm{\nabla u}_\lt^2 + \norm{\nabla v}_\lt^2\right)^{\frac{6+n}{8}}
        \left(\norm{xu}_\lt^2 + \norm{xv}_\lt^2\right)^{\frac{6-n}{8}},
    \end{split}
\end{equation}
where $C = C(C_{\circ}, n)$, and $C_\circ$ is the sharp constant of the Sobolev embedding $$\|u\|_\ltr \leq C_\circ\|\nabla u\|_\lt^{\frac{n}{6}}\|u\|_\lt^{1-\frac{n}{6}}.$$ 

From \eqref{ineqK}, for $n \leq 3$, we observe that
\begin{equation}\label{boundI}
    \begin{split}
        I(u,v) &\gtrsim
        \frac{1}{2}\|(u,v)\|_{\dot{H}}^2 
        - C(n, \mu_1, \mu_2)\|(u,v)\|_{\dot{H}}^{\frac{n}{2}} \\
        &
        \gtrsim 
        \frac{1}{4}\|(u,v)\|_{\dot{H}}^2
        - C(n, \mu_1, \mu_2)\\
        &\geq - C(n, \mu_1, \mu_2),
    \end{split}    
\end{equation}
where $$\|(u, v)\|_{\dot{H}}^2 := \norm{\nabla u}_\lt^2 + \norm{\nabla v}_\lt^2
+ \int_\rn V(x)(|u|^2 + |v|^2) \,d x.$$

Thus, when $n \leq 3$, $I$ is coercive and bounded from below. Hence, the existence of a minimizer of $d_{\mu_1,\mu_2}$ follows from the compact embedding in Remark \ref{remarkemb} in the case $V_1(x)$.

	\vspace{5mm}
	
	However, when attempting to apply these same tools to the $L^2$-critical and $L^2$-supercritical cases, certain issues arise. It is important to note that in this scenario, the functional $I$ is not bounded from below. Specifically, for any fixed $(\varphi, \psi) \in S(\mu_{1},\mu_{2})$ with $\varphi, \psi > 0$, and defining $(\varphi_{s}, \psi_{s}) = (s^{\frac{n}{2}} \varphi(s x), s^{\frac{n}{2}} \psi(s x))$, it follows that $(\varphi_{s}, \psi_{s}) \in S(\mu_{1},\mu_{2})$. However, it can be observed that
\begin{align*}
    I\left(\varphi_{s}, \psi_{s}\right)  
    &= \frac{s^{2}}{2} \int_{\rn}|\nabla \varphi|^{2} + \kappa|\nabla \psi|^{2} \,d x + \frac{1}{2 s^{2}} \int_{\rn} V(x)\left(\varphi^{2} + \psi^{2}\right) \,d x - \frac{s^{n/2}}{2} K(\phi,\psi) \rightarrow -\infty  
\end{align*}
as $s \rightarrow \infty$. For this reason, the method proposed in \cite{zhang2000} for obtaining standing waves does not apply to problem \eqref{elliptic}. So, we consider a local minimization problem \eqref{min2} for a given  $\chi > 0$. 

It is worth noting that
\begin{equation}\label{welldef}
    \begin{aligned}
        I(\phi,\psi) 
        & \geqslant -\frac{1}{2} K(\phi,\psi) \geqslant -\frac{1}{2} \|\phi\|_{L^3(\rn)}^{2} \|\psi\|_{L^3(\rn)} \\
        & \geqslant -C \|\nabla \phi\|_{L^2(\rn)}^{n/3} \|\phi\|_{L^2(\rn)}^{2 - n/3} \|\nabla \psi\|_{L^2(\rn)}^{n/6} \|\psi\|_{L^2(\rn)}^{1 - n/6} \\
        & \geqslant -C \mu_{1}^{1 - \frac{n}{6}} \mu_{2}^{\frac{1}{2} - \frac{n}{12}} \chi^{\frac{n}{4}}.
    \end{aligned}    
\end{equation}

	So, for any fixed $\chi>0$, $\mu_{1}, \mu_{2}>0$ with $S(\mu_1,\mu_2)\cap B(\chi) \neq \emptyset$, we have that $d_{\mu_{1}, \mu_{2}}^{\chi}$ is well-defined.

\begin{remark}\label{remarkemb}
    From \cite[Lemma 3.1]{zhang2000} or \cite{bar-pankov, omanawillem}, we have that $\{\phi \in H^1(\mathbb{R}^n): V_1|\phi|^2 \in L^1(\mathbb{R}^n)\}$ is compactly embedded into $L^{q+1}(\mathbb{R}^n)$ for $1 \leq q < \infty$ if $n = 1,2$, and $1 \leq q < 1 + 4/(n-2)$ if $n \geq 3$. Furthermore, by the Sobolev embedding theorem, $H^{1}(\mathbb{R}^n) \hookrightarrow L^{q+1}(\mathbb{R}^n)$ for $1 \leq q < \infty$ if $n = 1,2$, and $1 \leq q < 1 + \frac{4}{n-2}$ if $n \geq 3$. Thus, $\{\phi \in H^1(\mathbb{R}^n): V_2|\phi|^2 \in L^1(\mathbb{R}^n)\}$ is embedded into $H^{1}(\mathbb{R}^n)$, which in turn is embedded into $L^{q+1}(\mathbb{R}^n)$. However, this embedding is not compact if we replace $V_1$ with $V_2$.
\end{remark}

We now proceed to establish the following result that ensures the existence of a minimizer for \eqref{min2}.

\begin{theorem}\label{0theorem}
    There exists a solution to \eqref{elliptic} as a minimizer to \eqref{min2}.
\end{theorem}

\begin{proof}
    Let $(\phi_{m}, \psi_{m}) \in H$ such that
    $$
    \int_{\mathbb{R}^{n}}\left|\phi_{m}\right|^{2} \,d x \rightarrow \mu_{1}, \, \int_{\mathbb{R}^{n}}\left|\psi_{m}\right|^{2} \,d x \rightarrow \mu_{2}, \quad I\left(\phi_{m}, \psi_{m}\right) \rightarrow d_{\mu_1, \mu_2}^{\chi}, \quad \|(u_{m}, v_{m})\|_{\dot{H}}^{2} \leq \chi.
    $$
    Thus, from \eqref{ineqK}, we have 
    $$K(u_{m}, v_{m}) \leq C \chi^{n/4}.$$
    Now, since $I\left(\phi_{m}, \psi_{m}\right) \rightarrow d_{\mu_1, \mu_2}^{\chi}$, then 
    \begin{equation}\label{bound}
        \begin{aligned}
            \frac{1}{2}\left(\|\nabla \phi_{m}\|_{L^{2}}^{2} + \kappa \|\nabla \psi_{m}\|_{L^{2}}^{2}\right) + \frac{1}{2} \int_{\rn} V(x)(|\phi_{m}|^2 + |\psi_{m}|^2) \,d x 
            &\leq d_{\mu_1, \mu_2}^{\chi} + \frac{1}{2} + K(\phi_{n}, \psi_{n}) \\
            & \leq d_{\mu_1, \mu_2}^{\chi} + \frac{1}{2} + C \chi^{n/4}.
        \end{aligned}   
    \end{equation}
    Therefore, $\{(\phi_n, \psi_n): n \in \mathbb{N} \}$ is bounded in $H$. Then there exists $(\phi, \psi) \in H$ such that
    \begin{equation}\label{weakconv}
        (\phi_{m}, \psi_{m}) \rightharpoonup (\phi, \psi) \text{ in } H.    
    \end{equation}
    We observe that $H \hookrightarrow L^2(\mathbb{R}^n) \times L^2(\mathbb{R}^n)$ from the uncertainty inequality.
    Moreover, this embedding is compact. Thus,  
    $$
    (\phi_{m}, \psi_{m}) \rightarrow (\phi, \psi) \text{ in } L^2(\rn) \times L^2(\rn).
    $$
    This implies that
    \begin{equation}\label{iguality12}
        \int_{\mathbb{R}^{n}} |\phi|^{2} \,d x = \mu_1, \, \int_{\mathbb{R}^{n}} |\psi|^{2} \,d x = \mu_2.    
    \end{equation}
    Also, from \eqref{remarkemb}, we have that 
    $$
    (\phi_{m}, \psi_{m}) \rightarrow (\phi, \psi) \text{ in } L^q(\mathbb{R}^n) \times L^q(\mathbb{R}^n), \, \text{ for } 2 \leq q \leq \frac{2n}{n-2}.
    $$
    Thus, since $\{(\phi_n, \psi_n): n \in \mathbb{N} \}$ is bounded in $H$, 
    \begin{equation}\label{estim}
        \begin{aligned}
            \left| \int_{\mathbb{R}^n} \phi_{m}^2 \psi_{m} \,d x - \int_{\mathbb{R}^n} \phi^2 \psi \,d x \right|
            &\leq \left| \int_{\mathbb{R}^n} \phi_{m}^2 \psi_{m} \,d x - \int_{\mathbb{R}^n} \phi^2 \psi_{m} \,d x \right| + \left| \int_{\mathbb{R}^n} \phi^2 \psi_{m} \,d x - \int_{\mathbb{R}^n} \phi^2 \psi \,d x \right| \\
            &\leq C \left( \|\phi_{m} - \phi\|_{L^3(\rn)} + \|\psi_{m} - \psi\|_{L^3(\rn)} \right).
        \end{aligned}
    \end{equation}
    Therefore, by combining \eqref{weakconv}, \eqref{iguality12}, and \eqref{estim}, we get that $I(u) = d_{\mu_1, \mu_2}^{\chi}$. Hence, 
    $$d_{\mu_{1}, \mu_{2}}^{\chi} = \min_{S\left(\mu_{1}, \mu_{2}\right) \cap B(\chi)} I(\phi, \psi),$$
    from where it follows the desired result.
\end{proof}

	\subsection{Existence of standing waves: Partial confinement }
	In this section, we will consider $V(x) = V_2(x)$ in the system \eqref{elliptic}. Our goal is to ensure the existence of standing waves under these new conditions. This section is dedicated to the study of the existence of stationary waves for the system \eqref{elliptic} with partial confinement within the $L^2(\rn)$-subcritical and $L^2(\rn)$-supercritical scenarios. To accomplish this, we will employ spectral theory arguments associated with the aforementioned system, providing valuable tools for addressing the challenges inherent to our work. In particular, the arguments we will use are largely based on the principles of profile decomposition and compactness-concentration.

Consider the system
\begin{equation}\label{ellipticn}
    \left\{
    \begin{array}{l}
        -\Delta u + (x_1^2 + x_2^2 + \cdots + x_{n-1}^2)u = \lambda_{1} u, \\
        -\kappa \Delta v + (x_1^2 + x_2^2 + \cdots + x_{n-1}^2)v = \lambda_{2} v,
    \end{array}
    \right.
\end{equation}


	\begin{lemma}\label{l0}
		Define
		$$
		L_0:=\inf _{\int_{\rn}|u|^2 d x=1} \int_{\rn}|\nabla u|^2+V_{2}(x)|u|^2 d x,
		$$
		$$
		M_0:=\inf _{\int_{\rn}|u|^2 d x=1} \int_{\rn}\kappa|\nabla u|^2+V_{2}(x)|u|^2 d x,
		$$
		$$
		l_0:=\inf _{\int_{\R^{n-1}}|v|^2 d x^{\prime}=1} \int_{\mathbb{R}^
			{n-1}}\left(\sum_{i=1}^{n-1}\left|\partial_{x_i} v\right|^2\right)+V_{2}(x)|v|^2 d x^{\prime} .
		$$
		and
		$$
		m_0:=\inf _{\int_{\R^{n-1}}|v|^2 d x^{\prime}=1} \int_{\mathbb{R}^
			{n-1}}\kappa\left(\sum_{i=1}^{n-1}\left|\partial_{x_i} v\right|^2\right)+V_{2}(x)|v|^2 d x^{\prime} .
		$$
		Then, $L_0=l_0$ and $M_0=m_0=\sqrt{\kappa}l_{0}$.    
	\end{lemma}

	\begin{remark}
		It is known that (see \cite{ancarsil})
		$\gamma n$ is the first simple eigenvalue of $-\Delta +\gamma^2|x|^2$ with the eigenfunction $\mathbb{\pi}^{-\frac n2}\ee^{-\frac{\gamma|x|^2}{2} }$. The pure point spectrum of $-\Delta + |x|^2$ is $\{n+2k,k\in\N\}$, and the associated eigenfunctions are given by the Hermite functions.  
	\end{remark}
	Before delving into showing the existence of a solution for system \eqref{elliptic}, we will lay down some preliminary concepts that will play a crucial role in its resolution.
	
	First, we recall some rearrangement results. In the following, for any $x \in \rn$ we write $x:=\left(x^{\prime}, x_n\right)$ with $x^{\prime}:=\left(x^{\prime},\ldots,x_{n-1}\right) \in \R^{n-1}$ and $x_n \in \mathbb{R}$. Let $u: \rn \rightarrow \mathbb{R}$ be a Lebesgue measurable function and vanish at infinity, i.e. $\lim _{|x| \rightarrow \infty} u(x)=0$. For any $t>0, x^{\prime} \in \R^{n-1}$, setting
	$$
	\left\{\left|u\left(x^{\prime}, y\right)\right|>t\right\}:=\left\{y \in \mathbb{R}:\left|u\left(x^{\prime}, y\right)\right|>t\right\},
	$$
	the   Steiner rearrangement  $u^{\star }$ of $u$ is defined (see \cite{shiba}) by
	\begin{equation}\label{defsteiner}
		u^{\star }(x)=u^{\star }\left(x^{\prime}, x_n\right):=\int_0^{\infty} \chi_{\left\{\left|u\left(x^{\prime}, y\right)\right|>t\right\}^*}\left(x_n\right) d t,   
	\end{equation}
	where $A^{\star } \subset \mathbb{R}$ stands for the Steiner rearrangement of set $A \subset \mathbb{R}$ given by
	$$
	A^{\star }:=\left(-\sigma^1(A) / 2, \sigma^1(A) / 2\right),
	$$
	and $\sigma^{n}(A)$ is $n$-dimensional Lebesgue measure of set $A \subset \rn$. In view of the definition \eqref{defsteiner}, for any $x^{\prime} \in \R^{n-1}$ we see that the function $x_n \rightarrow u^{\star }\left(x^{\prime}, x_n\right)$ is nonincreasing with respect to $\left|x_n\right|$, and $u^{\star }\left(x^{\prime}, \cdot\right)$ is equimeasurable to $\left|u\left(x^{\prime}, \cdot\right)\right|$, namely for any $t>0$,
	\begin{equation}\label{sigma2}
		\sigma^1\left(\left\{u^{\star }\left(x^{\prime}, y\right)>t\right\}\right)=\sigma^1\left(\left\{\left|u\left(x^{\prime}, y\right)\right|>t\right\}\right).   
	\end{equation}
	About the Steiner rearrangement, we summarize well-known facts as follows.
	\begin{lemma}\label{steiner}
		Assume $1 \leq p<\infty$, and $u^{\star }$ be the Steiner rearrangement of $u$. Then 
		\begin{itemize}
			\item[(i)] $u^{\star }$ and $|u|$ is equimeasurable in $\rn$, i.e. for any $t>0$,
			$$
			\sigma^n\left(\left\{x \in \rn: u^{\star }(x)>t\right\}\right)=\sigma^n\left(\left\{x \in \rn:|u(x)|>t\right\}\right).
			$$
			\item[(ii)] Let $\Phi:[0, \infty) \rightarrow[0, \infty)$ be non-decreasing, then $\Phi(|u|)^{\star}=\Phi(u^{\star}).$
			In particular, $$\int_{\rn}\left|u^{\star }\right|^p d x=\int_{\rn}|u|^p d x.$$
			If $u\geq 0$ then $$\int_{\rn}\left (u^{\star }\right)^{2}v^{\star }\,d x=\int_{\rn} \left(u^{2}\right)^{\star}v^{\star }\,d x.$$
			\item[(iii)] If $u \in W^{1, p}\left(\rn\right)$, then $u^{\star } \in W^{1, p}\left(\rn\right)$, and
			$$
			\int_{\rn}\left|\partial_{x_i} u^{\star }\right|^p d x \leq \int_{\rn}\left|\partial_{x_i} u\right|^p d x \text { for } i=1,2,3,\ldots,n.
			$$
		\end{itemize}
	\end{lemma}
	\begin{proof}
		See \cite{lieb}.
	\end{proof}
	\begin{theorem}[\cite{Bell}]
		\label{theoremV}
		Let $V: \rn \rightarrow[0, \infty)$ be a measurable function, radially symmetric satisfying $V(|x|) \leq V(|y|)$ for $|x| \leq|y|$ then we have:
		$$
		\int_{\rn} V(|x|)\left|u^*\right|^2 d x \leq \int_{\rn} V(|x|)|u|^2 d x .
		$$
		If, in addition, $V(|x|)<V(|y|)$ for $|x|<|y|$, then
		$$
		\int_{\rn} V(|x|)\left|u^*\right|^2 d x=\int_{\rn} V(|x|)|u|^2 d x \Rightarrow u(x)=u^*(|x|) .
		$$
		This result holds for any measurable function $u$ which vanishes at infinity.
	\end{theorem}
	
	Now, considering two measurable functions $u, v$ which vanish at infinity, we define for $t>0, A^{\star}(u, v, t):=\left\{x \in \rn:|x|<r\right\}$, where $r>0$ is chosen such that
	$$
	\operatorname{\sigma}\left(\left\{x \in \rn:|x|<r\right\}\right)=\operatorname{\sigma}\left(\left\{x \in \rn:|u(x)|>t\right\}\right)+\operatorname{\sigma}\left(\left\{x \in \rn:|v(x)|>t\right\}\right),
	$$
	and we define  Schwarz rearrangement   $\{u, v\}^{\star}$ by
	$$
	\{u, v\}^{\star}(x):=\int_{0}^{\infty} \chi_{A^{\star}(u, v, t)}(x) \mathrm{d} t,
	$$
	where $\chi_{A}$ is a characteristic function of the set $A \subset \rn$ (see \cite{shiba}).
	\begin{lemma}\label{schwartz}
		(i) The function $\{u, v\}^{\star}$ is radially symmetric, non-increasing, and lower semi-continuous. Moreover, for each $t>0$, there holds $\left\{x \in \rn:\{u, v\}^{\star}>t\right\}=$ $A^{\star}(u, v, t)$.
		
		(ii) Let $\Phi:[0, \infty) \rightarrow[0, \infty)$ be  non-decreasing, lower semi-continuous, continuous at 0 and $\Phi(0)=$ 0. Then $\{\Phi(u), \Phi(v)\}^{\star}=\Phi\left(\{u, v\}^{\star}\right)$.
		
		(iii) $\left\|\{u, v\}^{\star}\right\|_{L^p(\rn)}^{p}=\|u\|_{L^p(\rn)}^{p}+\|v\|_{L^p(\rn)}^{p}$ for $1 \leqslant p<\infty$.
		
		
		(iv) If $u, v \in W^{1,p}\left(\rn\right)$, then $\{u, v\}^{\star} \in W^{1,p}\left(\rn\right)$ and $$\left\|\nabla\{u, v\}^{\star}\right\|_{L^p(\rn)}^{p} \leqslant\|\nabla u\|_{L^p(\rn)}^{p}+\|\nabla v\|_{L^p(\rn)}^{p}$$ for $1 \leqslant p<\infty$. In addition, if $u$, $v \in\left(W^{1,p}\left(\rn\right) \cap C^{1}\left(\rn\right)\right) \backslash\{0\}$ are radially symmetric, positive, and non-increasing, then
		
		$$
		\int_{\rn}\left|\nabla\{u, v\}^{\star}\right|^{p}\,d x<\int_{\rn}|\nabla u|^{p}\,d x+\int_{\rn}|\nabla v|^{p}\,d x, \, \, \text{for}\,\, 1 \leqslant p<\infty.
		$$
		
		(v) Let $u_{1}, u_{2}, v_{1}, v_{2} \geqslant 0$ be Borel measurable functions which vanish at infinity, then
		
		$$
		\int_{\rn}\left(u_{1} u_{2}+v_{1} v_{2}\right)\,d x \leqslant \int_{\rn}\left\{u_{1}, v_{1}\right\}^{\star}\left\{u_{2}, v_{2}\right\}^{\star}\,d x .
		$$
	\end{lemma}
	\begin{proof}
		See  \cite[Lemma  A.1]{ikoma}.
	\end{proof}
	We now present some lemmas that will be of great utility in the development of our work.
	\begin{lemma}\label{igualimit}
		If $\left(u_{m}, v_{m}\right) \rightarrow(u, v)$ in $H$, then up to a subsequence,
		$$
		\begin{gathered}
			\left\|(u_{m},v_{m})\right\|_{\dot{H}}^2=\left\|(u_{m},v_{m})-(u,v)\right\|_{\dot{H}}^2+\|(u,v)\|_{\dot{H}}^2+o_n(1),\\
			\int_{\rn}(u_{m}^2v_{m}-(u_{m}-u)^{2}(v_{m}-v))\,d x= K(u , v )+o_m(1),
		\end{gathered}
		$$
		where $o_m(1) \rightarrow 0$ as $m \rightarrow \infty$.
	\end{lemma}
	\begin{proof}
		See Lema $3.2$ in \cite{zhanp}.
	\end{proof}
	\subsection*{\underline{Subcritical case}}
	In order, to ensure the existence of stationary waves \blue{for the system \eqref{elliptic} in $L^2(\R^{n})$-subcritical}  case, that is,  when $1\leq n\leq 3$, we proceed to apply the profile decomposition principle to address the minimization problem \eqref{min1}. 
	
	Indeed, we present a result of the profile decomposition principle.
	\begin{proposition}\label{profile}
		Let $1\leq n\leq 3$ and $\left(u_{m}, v_{m}\right)_{m \geq 1}$ be a bounded sequence in $H$. Then there exist a subsequence, still denoted by $\left(u_{m}, v_{m}\right)_{m \geq 1}$, a family $\left(x_{m}^{j}\right)_{m \geq 1}$ of sequences in $\R$ and a sequence $\left(U^{j}, V^{j}\right)_{j \geq 1}$ of $H$-functions such that
		
		(1) for every $j \neq k$,
		
		$$
		\left|x_{m}^{j}-x_{m}^{k}\right| \rightarrow \infty \text { as } m \rightarrow \infty
		$$
		
		(2) for every $l \geq 1$ and every $x \in \rn$,
		
		$$
		u_{m}(x)=\sum_{j=1}^{l} \tau_{x_m^j}U^{j}(x)+r_{m}^{l}(x), \quad v_{m}(x)=\sum_{j=1}^{l} \tau_{x_m^j}V^{j}(x)+p_{m}^{l}(x)
		$$
		where $\tau_{x_m^j} U^{j}(x)=U^j(x^{\prime},x_{n}-x_m^j)$ and $\tau_{x_m^j} V^{j}(x)=V^j(x^{\prime},x_{n}-x_m^j)$ and
		
		\begin{equation}\label{limsup}
			\limsup _{m \rightarrow \infty}\left\|\left(r_{m}^{l}, p_{m}^{l}\right)\right\|_{L^{q}(\rn) \times L^{q}(\rn)} \rightarrow 0 \text { as } l \rightarrow \infty    
		\end{equation}
		for every $q \in(2,4)$.
		Moreover, for every $l \geq 1$,
		\begin{equation}\label{M}
			\begin{aligned}
				M\left(u_{m}, v_{m}\right) =\sum_{j=1}^{l} M\left(U^{j}, V^{j}\right)+M\left(r_{m}^{l}, p_{m}^{l}\right)+o_{m}(1), \end{aligned}  
		\end{equation}
		
		\begin{equation}\label{P}
			\begin{aligned}
				P\left(u_{m}, v_{m}\right) =\sum_{j=1}^{l} P\left(U^{j}, V^{j}\right)+P\left(r_{m}^{l}, p_{m}^{l}\right)+o_{m}(1)\end{aligned}  
		\end{equation}
		and 
		\begin{equation}\label{K}
			\begin{aligned}
				K\left(u_{m}, v_{m}\right)  =\sum_{j=1}^{l} K\left(U^{j}, V^{j}\right)+K\left(r_{m}^{l}, p_{m}^{l}\right)+o_{m}(1)
			\end{aligned}    
		\end{equation}
		where $o_{m}(1) \rightarrow 0$ as $m \rightarrow \infty$.
	\end{proposition}
	\begin{proof}
		The proof is based on the argument of \cite{dinh}. Let $(\mathbf{u}, \mathbf{v})=\left(u_{m}, v_{m}\right)_{m \geq 1}$ be a bounded sequence in $H$. Since $H$ is a Hilbert space, we denote $A(\mathbf{u}, \mathbf{v})$ the set of functions obtained as weak limits of sequences of $\left(u_{m}\left(x^{\prime},\cdot+x_{m}\right), v_{m}\left(x^{\prime},\cdot+x_{m}\right)\right)_{m \geq 1}$ with $\left(x_{m}\right)_{m \geq 1}$ a sequence in $\R$. Denote
		
		$$
		T(\mathbf{u}, \mathbf{v}):=\sup \{M(u, v)+P(u, v):(u, v) \in A(\mathbf{u}, \mathbf{v})\}
		$$
		
		If $T(\mathbf{u}, \mathbf{v})=0$, then we can take $\left(U^{j}, V^{j}\right)=(0,0)$ for all $j \geq 1$. Otherwise we choose $\left(U^{1}, V^{1}\right) \in A(\mathbf{u}, \mathbf{v})$ such that
		$$
		M\left(U^{1}, V^{1}\right)+P\left(U^{1}, V^{1}\right) \geq \frac{1}{2} T(\mathbf{u}, \mathbf{v})>0 .
		$$
		
		By definition of $A(\mathbf{u}, \mathbf{v})$, there exists a sequence $\left(x_{m}^{1}\right)_{m\geq 1} \subset \R$ such that up to a subsequence,
		$$
		\left(u_{m}\left(x^{\prime},\cdot+x_{m}^{1}\right), v_{m}\left(x^{\prime},\cdot+x_{m}^{1}\right)\right) \rightharpoonup  \left(U^{1}, V^{1}\right) \text { weakly in } H
		$$
		
		We consider  $r_{m}^{1}(x):=u_{m}(x)-U^{1}\left(x-x_{m}^{1}\right)$ and $p_{m}^{1}(x):=v_{m}(x)-V^{1}\left(x-x_{m}^{1}\right)$. It follows that $\left(u_{m}^{1}\left(\cdot+x_{m}^{1}\right), v_{m}^{1}\left(\cdot+x_{m}^{1}\right)\right) \rightharpoonup  (0,0)$ weakly in $H$. Then, using the Brezis–Lieb  Lemma, we have 
		\begin{equation}\label{brez1}
			\begin{aligned}
				&\|u_m\|_{\lt}^2=\|\tau_{x_{m}^{1}}U^1\|_{\lt}^2+\|r_m^1\|_{\lt}^2+o_m(1)\\
				&\|\nabla u_m\|_{\lt}^2=\|\nabla \tau_{x_{m}^{1}}U^1\|_{\lt}^2+\|\nabla r_m^1\|_{\lt}^2+o_m(1),\\
				&\|v_m\|_{\lt}^2=\|\tau_{x_{m}^{1}}V^1\|_{\lt}^2+\|p_m^1\|_{\lt}^2+o_m(1),\\
				&\|\nabla v_m\|_{\lt}^2=\|\nabla \tau_{x_{m}^{1}}V^1\|_{\lt}^2+\|\nabla p_m^1\|_{\lt}^2+o_m(1)    
			\end{aligned}    
		\end{equation}
		and 
		\begin{equation}\label{brez2}
			\begin{aligned}
				&\int_{\rn}V(x)|u_m|^2\,d x=\int_{\rn}V(x)|\tau_{x_{m}^{1}}U^1|^2\,d x+\int_{\rn}V(x)|\tau_{x_{m}^{1}}r_m^1|^2\,d x,\\
				&\int_{\rn}V(x)|v_m|^2\,d x=\int_{\rn}V(x)|\tau_{x_{m}^{1}}V^1|^2\,d x+\int_{\rn}V(x)|\tau_{x_{m}^{1}}p_m^1|^2\,d x.   
			\end{aligned}    
		\end{equation}
		Therefore, 
		$$
		\begin{aligned}
			M\left(u_{m}, v_{m}\right) & =M\left(U^{1}, V^{1}\right)+M\left(r_{m}^{1}, p_{m}^{1}\right)+o_{m}(1) \\
			P\left(u_{m}, v_{m}\right) & =P\left(U^{1}, V^{1}\right)+P\left(r_{m}^{1}, r_{m}^{1}\right)+o_{m}(1).
		\end{aligned}
		$$
		Thus, we have
		$$
		\begin{aligned}
			K(u_{m},v_m)& =K \paar{\tau_{x_{m}^{1}}V^{1}+p_{m}^{1},\tau_{x_{m}^{1}}U^{1}+r_{m}^{1}}\\
			& =K \paar{\tau_{x_{m}^{1}}V^{1},\tau_{x_{m}^{1}}U^1}+K \paar{p_{m}^{1},r_{m}^{1}}+R_{1, m} \\
			& =K \paar{V^{1},U^{1}}+K \paar{p_{m}^{1},r_{m}^{1}}+R_{1, m},
		\end{aligned}
		$$
		where
		$$
		\begin{aligned}
			R_{1, m}:&=K\paar{\tau_{x_{m}^{1}}V^{1},r_{m}^{1}}  +K \paar{p_{m}^{1},\tau_{x_{m}^{1}}U^{1}} \\
			& \quad +2 \Re\int_{\rn} \overline{p_{m}^{1}} \tau_{x_{m}^{1}}U^{1}r_{m}^{1}\,d x+2 \Re\int_{\rn}\overline{\tau_{x_{m}^{1}} V^{1}} \tau_{x_{m}^{1}}U^{1}r_{m}^{1}\,d x.
		\end{aligned}
		$$
			Since $\left(r_{m}^{1}\left(\cdot+x_{m}^{1}\right), p_{m}^{1}\left(\cdot+x_{m}^{1}\right)\right) \rightharpoonup (0,0)$ weakly in $H$, we see that $R_{1, m}=o_{m}(1)$. Hence,
		$$
		\begin{aligned}
			K\left(u_{m}, v_{m}\right) & =K\left(U^{1}, V^{1}\right)+K\left(r_{m}^{1}, p_{m}^{1}\right)+o_{m}(1).
		\end{aligned}
		$$
		We now replace $(\mathbf{u}, \mathbf{v})=\left(u_{M}, v_{M}\right)_{M \geq 1}$ by $\left(\mathbf{u}^{1}, \mathbf{v}^{1}\right)=\left(r_{m}^{1}, p_{m}^{1}\right)_{m \geq 1}$ and repeat the same process. If $T\left(\mathbf{u}^{1}, \mathbf{v}^{1}\right)=0$, then we choose $\left(U^{j}, V^{j}\right)=(0,0)$ for all $j \geq 2$. Otherwise, there exist $\left(U^{2}, V^{2}\right) \in A\left(\mathbf{u}^{1}, \mathbf{v}^{1}\right)$ and a sequence $\left(x_{m}^{2}\right)_{m \geq 1} \subset \R$ such that
		
		$$
		M\left(U^{2}, V^{2}\right)+K\left(U^{2}, V^{2}\right) \geq \frac{1}{2} T\left(\mathbf{u}^{1}, \mathbf{v}^{1}\right)>0
		$$
		and $\left(r_{m}^{1}\left(x^{\prime},\cdot+x_{m}^{2}\right), p_{m}^{1}\left(x^{\prime},\cdot+x_{m}^{2}\right)\right) \rightharpoonup \left(U^{2}, V^{2}\right)$ weakly in $H$. Set $r_{m}^{2}(x):=r_{m}^{1}(x)-U^{2}\left(x^{\prime},\cdot-x_{m}^{2}\right)$ and $p_{m}^{2}(x):=p_{m}^{1}(x)-V^{2}\left(x^{\prime},\cdot-x_{m}^{2}\right)$. It follows that $\left(r_{m}^{2}\left(x^{\prime},\cdot+x_{m}^{2}\right), p_{m}^{2}\left(x^{\prime},\cdot+x_{m}^{2}\right)\right) \rightharpoonup (0,0)$ weakly in $H$ and 
		$$
		\begin{aligned}
			M\left(r_{m}^{1}, p_{m}^{1}\right) & =M\left(U^{2}, V^{2}\right)+M\left(r_{n}^{2}, p_{m}^{2}\right)+o_{m}(1) \\
			K\left(r_{m}^{1}, p_{m}^{1}\right) & =K\left(U^{2}, V^{2}\right)+K\left(r_{n}^{2}, p_{n}^{2}\right)+o_{m}(1) \\
			P\left(r_{m}^{1}, p_{m}^{1}\right) & =P\left(U^{2}, V^{2}\right)+P\left(r_{m}^{2}, p_{m}^{2}\right)+o_{m}(1).
		\end{aligned}
		$$
		
		We next claim that $\left|x_{m}^{1}-x_{m}^{2}\right| \rightarrow \infty$ as $m \rightarrow \infty$. Indeed, if it is not true, then up to a subsequence, $x_{m}^{1}-x_{m}^{2} \rightarrow x_{0}$ as $m \rightarrow \infty$ for some $x_{0} \in \R$. Since
		$$
		\begin{aligned}
			&r_{m}^{1}\left(x^{\prime},x_n+x_{m}^{2}\right)=r_{m}^{1}\left(x^{\prime},x_n+\left(x_{m}^{2}-x_{m}^{1}\right)+x_{m}^{1}\right), \\ & p_{m}^{1}\left(x^{\prime},x_n+x_{m}^{2}\right)=p_{m}^{1}\left(x^{\prime},x_n+\left(x_{m}^{2}-x_{m}^{1}\right)+x_{m}^{1}\right)   
		\end{aligned}
		$$
		and $$\left(r_{m}^{1}\left(x^{\prime},\cdot+x_{m}^{1}\right), p_{m}^{1}\left(x^{\prime},\cdot+x_{m}^{1}\right)\right) \rightharpoonup (0,0),$$ it yields that $\left(U^{2}, V^{2}\right)=(0,0)$, which is a contradiction to $T\left(\mathbf{u}^{1}, \mathbf{v}^{1}\right)>0$.
		
		A bootstrapping argument and orthogonal extraction allow us to construct the family $\left(x_{n}^{j}\right)_{j \geq 1}$ of sequences in $\mathbb{R}$ and a sequence $\left(U^{j}, V^{j}\right)_{j \geq 1}$ of $H$-functions satisfying the desired conclusion. To show \eqref{limsup}, we can refer to Proposition $3.5$ in \cite{dinh}, where we can find all the details, and whose arguments remain valid in our conditions.
	\end{proof}
	To establish the existence of stationary waves in the subcritical case, we consider the ellipse \eqref{ellipse}. Our objective is to prove Theorem \ref{theorem2-sub}, which ensures the existence of critical points of \( I \) on \( S \), corresponding to \eqref{min}.
	
	Before presenting a proof of Theorem \ref{theorem2-sub}, we need to introduce some preliminary lemmas. Here, we present the first lemma, which establishes a correlation estimate between the minimizer \( d_{\mu} \) and the eigenvalues of the spectral problem in \eqref{l0}.

	\begin{lemma}\label{2propd}
		Let $\mu \geq 0$, then we have
		$$d_{\mu_{1}, \mu_{2}}<\frac{l_0 \mu}{4}\left(1+\frac{\sqrt{\kappa}}{2}\right)$$
	\end{lemma}
	\begin{proof}
		Since $\tilde{H}\hookrightarrow L^2(\R^{n-1})$, it is standard to show that $(\Psi_0,\Phi_0)$ and $(l_0,m_0)$ satisfy  
$$
\begin{gathered}
\begin{array}{ll}
-\Delta_{x^{\prime}} \Psi_0+V_2(x) \Psi_0=l_0 \Psi_0,\, \,  \int_{\R^{n-1}}\left|\Psi_0\right|^2 d x^{\prime}=1,    \\
-\kappa\Delta_{x^{\prime}} \Phi_0+V_2(x) \Phi_0=m_0\Phi_0, \, \, \int_{\R^{n-1}}\left|\Phi_0\right|^2 d x^{\prime}=1.
\end{array}
\end{gathered}
$$
From Lemma \ref{l0}, we know that $L_0=l_0$ and $M_0=m_0$. Now, we set $\psi, \phi \in H^1(\R)$ such that 
$$
u(x)=\Psi_0(x^{\prime}) \psi\left(x_{n}\right),\quad \int_{\mathbb{R}}|\psi|^2 d x_{n}=\frac{\mu}{2}
$$
$$
v(x)=\Phi_0(x^{\prime}) \phi\left(x_{n}\right),\quad \int_{\mathbb{R}}|\phi|^2 d x_{n}=\frac{\mu}{4}
$$
with $\psi \left(x_{n}\right)$ and $\phi\left(x_{n}\right)$ to be chosen later. After some basic calculations, it is evident that $(u, v) \in S\left(\mu\right)$.
Moreover, from Lemma \ref{l0}, we have $m_0=\sqrt{\kappa}l_0$, then
\begin{equation}\label{Ibas}
\begin{aligned}
I(u, v)&=  \frac{1}{2} \int_{\rn}|\nabla u|^2+\kappa|\nabla v|^2+V_2(x)\left(|u|^2+|v|^2\right) d x  -\frac{1}{2}  K(u , v ) \\
&=\frac{1}{2} \int_{\rn}\left(-\Delta_{x^{\prime}} u +V_2(x) u\right)\overline{u}\,d x+\frac{1}{2} \int_{\rn}\left(-\kappa\Delta_{x^{\prime}} v +V_2(x) v\right)\overline{v}\,\,d x\\
&\quad + \frac{1}{2} \int_{\mathbb{R}}\left|\partial_{x_n} \psi\right|^2 d x_{n}+\frac{1}{2} \int_{\mathbb{R}}\left|\partial_{x_n} \phi \right|^2 \,d x_{n}-\frac{1}{2} \left(\int_{\R^{n-1}}(\Psi_0)^{2}\Phi_{0} d x^{\prime}\right)\left(\int_{\mathbb{R}}(\psi)^{2}\phi d x_{n}\right)\\
& =\frac{1}{2} \int_{\mathbb{R}}\left|\partial_{x_n} \psi\right|^2 d x_{n}+\frac{1}{2} \int_{\mathbb{R}}\left|\partial_{x_n} \phi \right|^2 \,d x_{n}+\frac{l_0}{2} \int_{\mathbb{R}}|\psi|^2 \,d x_{n}+\frac{\sqrt{\kappa}l_0}{2} \int_{\mathbb{R}}|\phi|^2 \,d x_{n} \\
& \quad-\frac{1}{2}\Re\left(\int_{\R^{n-1}}(\Psi_0)^{2}\Phi_{0}\,d x^{\prime}\right)\left(\int_{\mathbb{R}}(\psi)^{2}\phi \,d x_{n}\right).
\end{aligned}    
\end{equation}
To obtain the desired result, it is sufficient to choose $(\psi,\phi)$ such that
\begin{equation}\label{1firtscon}
\frac{1}{2} \int_{\mathbb{R}}\left|\partial_{x_n} \psi\right|^2 d x_{n}+\frac{1}{2} \int_{\mathbb{R}}\left|\partial_{x_n} \phi \right|^2 \,d x_{n}-\frac{1}{2} \left(\int_{\R^{n-1}}(\Psi_0)^{2}\Phi_{0}\,d x^{\prime}\right)\left(\int_{\mathbb{R}}(\psi)^{2}\phi \,d x_{n}\right)<0.   
\end{equation}
In fact, we define  $$\psi\left(x_n\right)=\sqrt{\gamma} f\left(\gamma x_n\right),\quad \text{where} \int_{\mathbb{R}}\left|f\left(x_{n}\right)\right|^2 d x_n=\mu_1$$
and $$\phi\left(x_{n}\right)=\sqrt{\gamma} g\left(\gamma x_{n}\right), \quad\text{where} \int_{\mathbb{R}}\left|g\left(x_{n}\right)\right|^2 d x_n=\mu_2.$$
We claim that there exists a $\gamma_0>0$ such that $
(\psi,\phi)$ satisfies all the conditions above for every $\gamma<\gamma_0$. Then, notice that
$$
\begin{aligned}
& \frac{1}{2} \int_{\mathbb{R}}\left|\partial_{x_n} \psi\right|^2 d x_{n}+\frac{1}{2} \int_{\mathbb{R}}\left|\partial_{x_n} \phi \right|^2 \,d x_{n}-\frac{1}{2} \left(\int_{\R^{n-1}}(\Psi_0)^{2}\Phi_{0}\,d x^{\prime}\right)\left(\int_{\mathbb{R}}(\psi)^{2}\phi \,d x_{n}\right) \\
= & \frac{\gamma^2}{2} \int_{\mathbb{R}}\left|\partial_{x_n} f\right|^2 d x_{n}+\frac{\gamma^2}{2} \int_{\mathbb{R}}\left|\partial_{x_n} g\right|^2 d x_n-\frac{\sqrt{\gamma}}{2}  \left(\int_{\R^{n-1}}(\Psi_0)^{2}\Phi_{0}\,d x^{\prime}\right)\left( \int_{\mathbb{R}}f^{2} g\,d x_{n}\right).
\end{aligned}
$$
Hence,   $\eqref{1firtscon}$ is deduced for $\gamma\ll1$. And the result is now complete.
	\end{proof}
	The following lemma shows the non-vanishing of the minimizing sequences of \eqref{min1}.
	\begin{lemma}\label{vani}
		Let $1 \leq n\leq 3$. Assume that $\left\{(u_{m},v_{m})\right\}$ is a minimizing sequence to \eqref{min1}. Then there exist $\delta>0$ such that
		
		$$
		\liminf _{m\rightarrow \infty} \int_{\rn}\left|u_{m}\right|^{3} \,d x>\delta
		$$ and $$
		\liminf _{m\rightarrow \infty} \int_{\rn}\left|v_{m}\right|^{3}\,d x>\delta.
		$$
	\end{lemma}
	\begin{proof}
		Suppose by contradiction that $$
		\liminf _{m\rightarrow \infty} \int_{\rn}\left|u_{m}\right|^{3} \,d x=o_m(1).
		$$
		Since  $\left\{(u_{m},v_{m})\right\}$ is a minimizing sequence to \eqref{min1},  
		it follows from \eqref{boundI} for $n\leq 3$ and $(u_m,v_m) \in S(\mu_1,\mu_2)$  that  ${(u_m,v_m)}$ is bounded in $H$. 
		So the fact
		$$
		K(u_{m}, v_{m}) \leq \|u_m\|_{L^3(\rn)}^2\|v_m\|_{L^3(\rn)} 
		$$
		shows that $  K(u_{m}, v_{m})=o_m(1).$ 
		By using a similar argument to that employed in \eqref{Ibas},
		$$\begin{aligned}
			d_{mu_1,\mu_2}&=\liminf_{m \rightarrow \infty}I(u_m,v_m)\\
			&\geq \frac{1}{2} \liminf_{m \rightarrow \infty} \int_{\rn}|\nabla u_m|^2+\kappa|\nabla v_m|^2+V_2(x)\left(|u_m|^2+|v_m|^2\right) d x \\
			&\geq \frac{\mu L_0}{4}+\frac{\mu M_0}{8}\\
			&=\frac{\mu l_0}{4}\left(1+\frac{\sqrt{\kappa}} {2}\right).
		\end{aligned}$$
		This is contradictory to Lemma \eqref{2propd}. Similarly, we can proceed if
		$$
		\liminf _{m\rightarrow \infty} \int_{\rn}\left|v_{m}\right|^{3} \,d x=o_m(1).
		$$
	\end{proof}

Now, we have all the conditions to prove Theorem \ref{theorem2-sub}.
	\begin{proof}[Proof of Theorem \ref{theorem2-sub}(i)]
		For the sake of simplicity we assume that $\vr=1$ and $S(\mu)=S_1(\mu)$. 
		
		Let $\sett{(u_m,v_m)}$ be a minimizing sequence for \eqref{min1}. From \eqref{boundI}, $\{(u_m,v_m)\}$ is bounded in $H$. Then, by applying the profile decomposition of bounded sequences in $H$ as outlined in Proposition \ref{profile}, there exists a weakly convergent subsequence of $\left\{(u_{m},v_{m})\right\}$ such that 
		
		\begin{equation}\label{defpro}
			u_{m}(x)=\sum_{j=1}^{l}\tau_{x_m^j} U^{j}\left(x\right)+r_{m}^{l}(x), \quad v_{m}(x)=\sum_{j=1}^{l} \tau_{x_m^j}V^{j}\left(x\right)+p_{m}^{l}(x)    
		\end{equation}
		with
		\begin{equation}\label{lim}
			\limsup _{m \rightarrow \infty}\left\|\left(r_{m}^{l}, p_{m}^{l}\right)\right\|_{L^{q}(\rn) \times L^{q}(\rn)} \rightarrow 0   
		\end{equation}
		for every $q \in(2,4)$. Hence, from Proposition \ref{profile},
		\begin{equation}\label{1I}
			I\left(u_{m},v_m\right)=\sum_{j=1}^{l} I\left(\tau_{x_m^j} U^{j},\tau_{x_m^j} V^{j}\right)+I\left(r_{m}^{l},p_{m}^{l}\right)+o_m(1).    
		\end{equation}
		Now, for all $j =1,\ldots,l$, we consider the  transformations $ U_{s_{j}}^{j}(x)=s_{j} \tau_{x_m^j} U^{j}(x)$ and $ V_{s_{j}}^{j}(x)=s_{j} \tau_{x_m^j} V^{j}(x)$ with $$s_{j}=\frac{\sqrt{\mu}}{\sqrt{\left\|U^{j}(x)\right\|_{\lt}^2+2\left\|V^{j}(x)\right\|_{\lt}^2}}.$$
		Then, it is clear for all $j =1,\cdots,l$ that 
		\begin{equation}\label{satimu}
			\left\|U_{s_{j}}^{j}(x)\right\|_{\lt}^{2}+2 \left\|V_{s_{j}}^{j}(x)\right\|_{\lt}^{2}=\mu.    
		\end{equation}
		Thus, $(U_{s_{j}}^{j},V_{s_{j}}^{j}) \in S(\mu)$ for all $j =1,\cdots,l$.		On the other hand, 
		$$
		\begin{aligned}			&I\left(U_{s_{j}}^{j},V_{s_{j}}^{j}\right)\\ & =\frac{1}{2}\left(\left\|\nabla U_{s_{j}}^{j}\right\|_{\lt}^{2}  + \int_{\rn} V(x)\left(\left| U_{s_{j}}^{j}\right|^{2} +|V_{s_{j}}^{j}|^2\right)\,d x+\kappa\left\|\nabla V_{s_{j}}^{j}\right\|_{\lt}^{2}  \right)-\frac{1}{2 } K\paar{U_{s_{j}}^{j},V_{s_{j}}^{j}} \\
			& =\frac{1}{2}\left(s_j^2\left\|\nabla \tau_{x_m^j}U^{j}\right\|_{\lt}^{2}  + \int_{\rn} V(x)\left(s_j^2\left|\tau_{x_m^j} U^{j}\right|^{2} +s_{j}^2\left|\tau_{x_m^j}V^{j}\right|^2\right)\,d x+\kappa s_{j}^2\left\|\nabla \tau_{x_m^j} V^{j}\right\|_{\lt}^{2}  \right)\\&\quad-\frac{s_j^3}{2 } K\paar{\tau_{x_m^j}U^{j},\tau_{x_m^j}V^{j}}\\
			& =s_j^2 I\left(\tau_{x_m^j} U^{j},\tau_{x_m^j} V^{j}\right)-\frac{ s_j^2\left( s_j-1\right)}{2 } K\paar{\tau_{x_m^j}U^{j},\tau_{x_m^j}V^{j}}.
		\end{aligned}
		$$
		This shows that
		\begin{equation}\label{2I}
			I\left(\tau_{x_m^j} U^{j},\tau_{x_m^j} V^{j}\right)\geq \frac{I\left( U_{s_{j}}^{j},V_{s_{j}}^{j}\right)}{s_{j}^{2}}+\frac{ s_j-1}{2  } K\paar{\tau_{x_m^j}U^{j},\tau_{x_m^j}V^{j}}.   
		\end{equation}
		Similarly, for the term $I\left((r_{m}^{l},p_{m}^{l})\right)$, we can get from \eqref{lim} that
		\begin{equation}\label{3I}
			\begin{aligned}
				I\left((r_{m}^{l},p_{m}^{l})\right) & =\frac{t_m^2}{\mu} I\left(\frac{\sqrt{\mu}}{t_m}r_m^1,\frac{\sqrt{\mu}}{t_m}r_m^2\right)+\frac{\left(\frac{\sqrt{\mu}}{t_m}\right)^{3}-1}{2 }K(r_{m}^{l},p_m^l)+o_m(1) \\
				& \geq \frac{t_m^2}{\mu} I\left(\frac{\sqrt{\mu}}{t_m}r_m^1,\frac{\sqrt{\mu}}{t_m}r_m^2\right)+o_m(1),
			\end{aligned}    
		\end{equation}
		where $t_m=\sqrt{\left\|r_{m}^{l}\right\|_{\lt}^{2}+2\left\|p_{m}^{l}\right\|_{\lt}^{2}}$.
		Due to $$\left\|\frac{\sqrt{\mu}}{t_m}r_{m}^{l}\right\|_{\lt}^{2}+2\left\|\frac{\sqrt{\mu}}{t_m}p_{m}^{l}\right\|_{\lt}^{2}=\mu,$$    we can obtain from the definition of $d_{\mu}$ and \eqref{satimu} that
		\begin{equation}\label{4I}
			I\left( U_{s_{j}}^{j},V_{s_{j}}^{j}\right) \geq d_{\mu}   
		\end{equation}
		and 
		\begin{equation}\label{5I}
			I\left(\frac{\sqrt{\mu}}{t_m}r_m^1,\frac{\sqrt{\mu}}{t_m}r_m^2\right)\geq d_{\mu}.    
		\end{equation}
		Moreover, from Proposition \ref{profile}, we know that   $$\sum_{j=1}^{\infty}\left(\left\|U^{j}(x)\right\|_{\lt}^{2}+2\left\|V^{j}(x)\right\|_{\lt}^{2}\right)$$ is convergent, so that $\left\|U^{j}(x)\right\|_{\lt}^{2}+2\left\|V^{j}(x)\right\|_{\lt}^{2} \rightarrow 0$  as $j \rightarrow \infty$. That is, $$\left\{\frac{1}{\left\|U^{j}(x)\right\|_{\lt}^{2}+2\left\|V^{j}(x)\right\|_{\lt}^{2}}\right\}_{j\geq1}$$ is bounded from  below. Hence,  there exists $j_{0} \geq 1$ such that
		\begin{equation}\label{6I}
			\inf _{j \geq 1} \frac{s_{j}-1}{2 }=\frac{1}{2}\left(\frac{\sqrt{\mu}}{\sqrt{\left\|U^{j_{0}}\right\|_{\lt}^2+2\left\|V^{j_{0}}\right\|_{\lt}^2}}-1\right).    
		\end{equation}
		Therefore, by combining \eqref{1I}-\eqref{6I}, we have
		$$
		\begin{aligned}
			& I\left(u_{m},v_{m}\right)\\ &\geq \sum_{j=1}^{l}\left(\frac{I\left( U_{s_{j}}^{j},V_{s_{j}}^{j}\right)}{s_{j}^{2}}+\frac{s_{j}-1}{2} K\paar{\tau_{x_m^j} U^{j},\tau_{x_m^j} V^{j}}\right)+o_m(1)   \\
			&\qquad +\frac{\left\|r_{m}^{l}\right\|_{\lt}^{2}+2\left\|p_{m}^{l}\right\|_{\lt}^{2}}{\mu} I\left(\frac{\sqrt{\mu}}{t_m}r_m^1,\frac{\sqrt{\mu}}{t_m}r_m^2\right)\\
			&\geq d_{\mu}\sum_{j=1}^{l} \left(\frac{\left\|U^{j}(x)\right\|_{\lt}^{2}+2\left\|V^{j}(x)\right\|_{\lt}^{2}}{\mu} \right)+\inf _{j \geq 1} \frac{s_{j}-1}{2 }\left(\sum_{j=1}^{l} K\paar{U^{j},V^{j}}\right) \\
			& \qquad+d_{\mu}\frac{\left\|r_{m}^{l}\right\|_{\lt}^{2}+2\left\|p_{m}^{l}\right\|_{\lt}^{2}}{\mu} +o_m(1).
		\end{aligned}
		$$
		So,  it follows from \eqref{M} and Lemma \ref{vani} that 
		$$I(u_m,v_m)\geq d_{\mu}+\frac{\delta}{2}\left(\frac{\sqrt{\mu}}{\sqrt{\left\|U^{j_{0}}\right\|_{\lt}^2+2\left\|V^{j_{0}}\right\|_{\lt}^2}}-1\right)+o_m(1)   $$
		as $m,l \rightarrow \infty$.
		Then, since $I(u_m,v_m) \rightarrow d_{\mu}$ as $m \rightarrow \infty$, we obtain 
		
		$$
		\frac{\delta}{2}\left(\frac{\sqrt{\mu}}{\sqrt{\left\|U^{j_{0}}\right\|_{\lt}^2+2\left\|V^{j_{0}}\right\|_{\lt}^2}}-1\right)\leq 0 ;
		$$
		from which 
		$\left\|U^{j_{0}}\right\|_{\lt}^{2} +2\left\|V^{j_{0}}\right\|_{\lt}^{2}\geq \mu$. Hence, $$\left\|U^{j_{0}}\right\|_{\lt}^{2} +2\left\|V^{j_{0}}\right\|_{\lt}^{2}=\mu$$ and from \eqref{defpro}, there exists only   terms $U^{j_{0}},V^{j_{0}} \not\equiv 0$ in \eqref{defpro}. Thus, from \eqref{M}  and $\left\|U^{j_{0}}\right\|_{\lt}^{2} +2\left\|V^{j_{0}}\right\|_{\lt}^{2}=\mu$, it follows that 
		$$
		\left\|u_{m}\right\|_{\lt}^2+2\left\|v_{m}\right\|_{\lt}^2\rightarrow\left\|U^{j_{0}}\right\|_{\lt}^{2} +2\left\|V^{j_{0}}\right\|_{\lt}^{2},
		$$and$$\left\|r_{m}^1\right\|_{\lt}^2+2\left\|r_{m}^2\right\|_{\lt}^2 \rightarrow 0,$$
		as $ m \rightarrow \infty.$ Hence,  it follows from \eqref{brez1} that $$
		\left\|u_{m}\right\|_{\lt}^2+\left\|v_{m}\right\|_{\lt}^2\rightarrow\left\|U^{j_{0}}\right\|_{\lt}^{2} +\left\|V^{j_{0}}\right\|_{\lt}^{2} \quad \text{as} \quad m \rightarrow \infty.
		$$
		Also, notice that 
		$$
		\int_{\rn} V(x)\left(\left|\tau_{x_{n}^{j_{0}}} U^{j_{0}}\right|^{2}+\left|\tau_{x_{n}^{j_{0}}} V^{j_{0}}\right|^{2} \right)\,d x=\int_{\rn} V(x)\left(\left|U^{j_{0}}\right|^{2}+\left|U^{j_{0}}\right|^{2}\right)\,d x.
		$$
		Finally,  we observe from \eqref{1I} that
		$$
		\begin{aligned}
			I\left(u_{m},v_{m}\right)&=
			I\left(U^{j_{0}},V^{j_{0}}\right)+\frac{1}{2} \int_{\rn}\left|\nabla r_{m}^{j_{0}}\right|^{2} \,d x+\frac{\kappa}{2} \int_{\rn}\left|\nabla p_{m}^{j_{0}}\right|^{2} \,d x\\
			&\quad+\frac{1}{2} \int_{\rn} V(x)\left(\left|r_{n}^{j_{0}}\right|^{2}+\left|p_{n}^{j_{0}}\right|^{2}\right) \,d x+o_m(1)\\
			&\geq I\left(U^{j_{0}},V^{j_{0}}\right).
		\end{aligned}
		$$
		Now, since $I\left(u_{m},v_{m}\right) \rightarrow d_{\mu}$ as $m \rightarrow\infty$ and $I\left(U^{j_{0}},V^{j_{0}}\right)\geq d_{\mu}$, then $$\int_{\rn}\left|\nabla r_{m}^{j_{0}}\right|^{2} \,d x+\frac{\kappa}{2} \int_{\rn}\left|\nabla p_{m}^{j_{0}}\right|^{2} \,d x+\frac{1}{2} \int_{\rn} V(x)\left(\left|r_{n}^{j_{0}}\right|^{2}+\left|p_{n}^{j_{0}}\right|^{2}\right) \,d x \rightarrow 0\quad \text{as}\quad m \rightarrow \infty$$
		and as a result, $$I\left(U^{j_{0}},V^{j_{0}}\right)=d_{\mu}$$
		Therefore, we get 
		$$
		\begin{gathered}
			\lim _{m \rightarrow \infty}\left\|\nabla r_{m}^{j_{0}}\right\|_{\lt}=\lim _{m \rightarrow \infty}\left\|\nabla p_{m}^{j_{0}}\right\|_{\lt}=\lim _{n \rightarrow \infty} \int_{\rn} V(x)\left|r_{m}^{j_{0}}\right|^{2} \,d x=\lim _{n \rightarrow \infty} \int_{\rn} V(x)\left|p_{m}^{j_{0}}\right|^{2} \,d x=0.
		\end{gathered}
		$$
		From the above results and \eqref{P}, we conclude that 
		$$\begin{aligned}
			&\|\nabla u_m\|_{\lt}^2+\kappa \|\nabla v_m\|_{\lt}^2 +\int_{\rn}V(x)\left(|u_m|^2+|v_m|^2\right)\,d x \\ & \rightarrow \|\nabla U^{j_0}\|_{\lt}^2+\kappa \|\nabla V^{j_0}\|_{\lt}^2 +\int_{\rn}V(x)\left(|U^{j_0}|^2+|V^{j_0}|^2\right)\,d x\quad \text{as}\quad m \rightarrow \infty.   
		\end{aligned}$$
		Consequently,  we deduce from \eqref{brez1} and \eqref{brez2} that $$\|(u_m,v_m)\|_{H}^2  \rightarrow \|(U^{j_0},V^{j_0})\|_{H}^2 \quad \text{as}\quad m \rightarrow \infty.$$ By combining this with the fact that $(\tau_{-x_m^{j_0}}u_m,\tau_{-x_m^{j_0}}v_m) \rightharpoonup (U^{j_0},V^{j_0})$ in $H$ (from the construction of the sequence in Proposition \ref{profile}), we can conclude that $(\tau_{-x_m^{j_0}}u_m,\tau_{-x_m^{j_0}}v_m)  \rightarrow (U^{j_0},V^{j_0})$ in $H$ and $I(U^{j_0},V^{j_0})=d_{\mu}$. 
		This yields that $d_{\mu}$ is achieved by $(U^{j_0},V^{j_0})$.
	\end{proof}
	\vspace{5mm}
	Next,   consider the manifold $S(\mu_{1},\mu_{2})$ defined in \eqref{product-ball}. 	Our objective now is to solve the   minimization problem \eqref{min5}. 
	This entails ensuring the existence of stationary waves for the system \eqref{elliptic} under the constraints imposed by $S(\mu_1,\mu_2)$. To achieve this, we will employ rearrangement arguments (refer to Lemmas \ref{schwartz} and \ref{steiner}), \blue{combined with}  Lemma \ref{l0} and the  compactness-concentration principle.
	
	\blue{We started by} establishing the subsequent lemma, which extracts the characteristics of $d_{\mu_1, \mu_2}$ in terms of the first eigenvalue associated with $l_0$ as defined in Lemma \ref{l0}.
	\begin{lemma}\label{3propd}
		Let $\mu_{1}, \mu_{2} \geq 0$. Then,   
		\begin{itemize}
			\item[(i)] $\left(\mu_{1}, \mu_{2}\right) \mapsto d_{\mu_{1}, \mu_{2}}$ is continuous.
			\item[(ii)] for $i=1,2$ such that $\mu_{i}=a_{i}+b_{i}$ and $a_{i}, b_{i} \geq 0$, we have $d_{\mu_{1}, \mu_{2}} \leq d_{a_{1}, a_{2}}+d_{b_{1}, b_{2}}$.
			\item[(iii)]  $$d_{\mu_{1}, \mu_{2}}<\frac{l_0 \mu_{1}}{2}+\frac{\sqrt{\kappa}l_0\mu_{2}}{2}.$$
		\end{itemize}
	\end{lemma}
	\begin{proof}
		(i)   Let $\mu_{1}, \mu_{2} \geq 0$. By definition of $d_{\mu_{1}, \mu_{2}}$, we have  for any $\epsilon>0$ that there exists a sequence $\left(\tilde{u}_{1}^{m}, \tilde{v}_{1}^{m}\right)\subset S(\mu_1,\mu_2)$ such that $$I\left(\tilde{u}_{1}^{m}, \tilde{v}_{1}^{m}\right) \leq d_{\mu_{1}, \mu_{2}}+\frac{\epsilon}{2}.$$
		We assume that $\left(\mu_{1}^{m}, \mu_{2}^{m}\right)$ is a sequence satisfying $\mu_{1}^{m}, \mu_{2}^{m} \geq 0$ and $\left(\mu_{1}^{m}, \mu_{2}^{m}\right) \rightarrow\left(\mu_{1}, \mu_{2}\right)$ as $m \rightarrow \infty$. By taking  $$u_{1}^{m}:=\left(\frac{\mu_{1}^m}{\mu_{1}}\right)^{\frac{1}{2}} \tilde{u}_{1}^{m}, \, \, v_{1}^{m}:=\left(\frac{\mu_{2}^m}{\mu_{2}}\right)^{\frac{1}{2}} \tilde{v}_{1}^{m},$$ we have $\left(u_{1}^{m}, v_{1}^{m}\right) \in S\left(\mu_{1}^m,\mu_{2}\blue{^{m}}\right)$. 
		
		Since $\left(\mu_{1}^{m}, \mu_{2}^{m}\right) \rightarrow \left(\mu_{1}, \mu_{2}\right)$ as $m \rightarrow \infty$,  we have for $m$ sufficiently large that
		\begin{equation*}
			\begin{aligned}
				d_{\mu_{1}^{m}, \mu_{2}^{m}} &\leq I\left(u_{1}^{m}, v_{1}^{m}\right)\\&=\frac{\mu_{1}^{m}}{2\mu_1}\int_{\rn}|\nabla \tilde{u}_{1}^{m}|^{2}\,d x+\frac{\kappa \mu_2^m}{\mu_{2}}\int_{\rn}|\nabla \tilde{v}_{1}^{m}|^{2}\,d x+\frac{\mu_1^m}{\mu_{1}}\int_{\rn}V_{2}(x)|\tilde{u}_{1}^{m}|^{2}\,d x\\
				&\quad +\frac{\mu_2^m}{\mu_{2}}\int_{\rn}V_{2}(x)|\tilde{v}_{1}^{m}|^{2}\,d x-\frac{\mu_{1}^m}{\mu_{1}}\left(\frac{\mu_{2}^m}{\mu_{2}}\right)^{1/2}\Re\int_{\rn}(\tilde{u}_{1}^{m})^{2}\tilde{v}_{1}^{m}\,d x\\
				&= I\left(\tilde{u}_{1}^{m}, \tilde{v}_{1}^{m}\right)+ \left(\frac{\mu_{1}^m}{\mu_{1}}-1\right)\int_{\rn}|\nabla \tilde{u}_{1}^{m}|^{2}\,d x +\left(\frac{\mu_{2}^m}{\mu_{2}}-1\right)\int_{\rn}|\nabla \tilde{v}_{1}^{m}|^{2}\,d x   \\
				&\quad +\left(\frac{\mu_{1}^m}{\mu_{1}}-1\right)\int_{\rn}V_2(x)|\tilde{u}_{1}^{m}|^{2}\,d x +\left(\frac{\mu_{2}^m}{\mu_{2}}-1\right)\int_{\rn}V_2(x)|\tilde{v}_{1}^{m}|^{2}\,d x\\
				&\quad - \left(\frac{\mu_{1}^m\sqrt{\mu_{2}^m}}{\mu_{1}\sqrt{\mu_{2}}}-1\right)K(\tilde{u}_{1}^{m},\tilde{v}_{1}^{m})\\
				&\leq d_{\mu_{1}, \mu_{2}}+\epsilon.    
			\end{aligned}
		\end{equation*}
On the other hand, by an argument involving simpler calculations than those above, one deduces that for sufficiently large
		$$
		d_{\mu_{1}, \mu_{2}} \leq d_{\mu_{1}^{m}, \mu_{2}^{m}}+\epsilon.
		$$
		Therefore,  $d_{\mu_{1}^{m}, \mu_{2}^{m}}^{\chi} \rightarrow d_{\mu_{1}, \mu_{2}}^{\chi}$ as $m\rightarrow \infty$.
		
		\item[(ii)]  By definition of $d_{\mu_1,\mu_2}$, there exists a $\left(u_1, v_1\right) \in S\left(a_1,a_2\right)$ such that for any $\epsilon>0$, 
		$$
		I\left(u_1, v_1\right) \leq d_{a_1, a_2}+\frac{\epsilon}{2}.
		$$
		Similarly, we can find $\left(u_2, v_2\right) \in S\left(b_1,b_2\right) $ such that
		$$
		I\left(u_2, v_2\right) \leq d_{b_1, b_2}+\frac{\epsilon}{2} .
		$$
		Without loss of generality, we assume $u_1, u_2, v_1, v_2 \geq 0$ by considering the Steiner rearrangement of each sequence. Additionally, we suppose that $\text{supp} (u_1) \cap \, \text{supp} (u_2) = \emptyset$ and $\text{supp} (v_1) \cap \text{supp} (v_2) = \emptyset$. If this condition is not initially satisfied, we redefine $(u_1, v_1)$ and $(u_2, v_2)$ as follows: $\tilde{u}_1$ is defined as $u_1$ where $\text{supp} (u_1) \cap \, \text{supp} (u_2) = \emptyset$ and as $0$ where $\text{supp} (u_1) \cap \, \text{supp} (u_2) \neq \emptyset$. Similar adjustments are made for the other variables.
		
		Setting
		$$
		u=u_1+u_2, \quad v=v_1+v_2,
		$$
		we obtain that $(u, v) \in S(\mu_1,\mu_2) \cap B( \chi/\epsilon_0 )$, and
		\begin{equation*}
			\begin{aligned}
				d_{\mu_1, \mu_2} &\leq I(u, v) \\
				&=\frac{1}{2} \int_{\rn}|\nabla u_1+\nabla u_2|^{2}+\kappa|\nabla v_1+\nabla v_2|^{2}+V_2(x)\left(|u_1+u_2|^{2}+|v_1+v_2|^{2}\right) d x \\
				&\quad - \frac{1}{2}K(u_1+u_2,v_1+v_2)\\
				&\leq I\left(u_1, v_1\right)+I\left(u_2, v_2\right)+\blue{2}\Re\int_{\rn}\nabla u_1  \nabla \blue{\overline{u}_2}\,d x + \blue{2\kappa}\Re\int_{\rn}\nabla v_1  \nabla \blue{\overline{v}_2}\,d x+ \blue{2}\Re\int_{\rn} V_2(x)(u_1   \blue{\overline{u}_2})\,d x\\
				&\quad + \blue{2\Re\int_{\rn} V_2(x)(v_1   \overline{v}_2)\,d x}- \blue{\Re\int_{\rn} u_1u_2   \overline{v}_1\,d x}- \blue{\Re\int_{\rn} u_1u_2   \overline{v}_2\,d x}  - \frac{1}{2}K (\blue{u_2} ,v_1) - \frac{1}{2}K( \blue{u_1}, v_2) \\
				&\leq d_{a_1, a_2}+d_{b_1, b_2}+\epsilon.  
			\end{aligned}
		\end{equation*}
		Then, it follows that $d_{\mu_1, \mu_2} \leq d_{a_1, a_2}+d_{b_1, b_2}$.
		
		\item[(iii)] First we introduce $\Psi_0\left(x^{\prime}\right), \Phi_0\left(x^{\prime}\right)$ and $l_0, m_0$  \blue{as follows}
		$$
		\begin{gathered}
			\begin{array}{ll}
				-\Delta_{x^{\prime}} \Psi_0+V_2(x) \Psi_0=l_0 \Psi_0,\, \,  \int_{\R^{n-1}}\left|\Psi_0\right|^2 d x^{\prime}=1,    \\
				-\kappa\Delta_{x^{\prime}} \Phi_0+V_2(x) \Phi_0=m_0 \Phi_0, \, \, \int_{\R^{n-1}}\left|\Phi_0\right|^2 d x^{\prime}=1. 
			\end{array}
		\end{gathered}
		$$
		From Lemma \ref{l0}, we know that $L_0=l_0$ and $M_0=m_0$. Now we set
		$$
		u(x)=\Psi_0(x^{\prime}) \zeta\left(x_{n}\right), \int_{\mathbb{R}}|\zeta|^2 d x_{n}=\mu_1,
		$$
		$$
		v(x)=\Phi_0(x^{\prime}) \varsigma\left(x_{n}\right), \int_{\mathbb{R}}|\varsigma|^2 d x_{n}=\mu_2,
		$$
		such that $\zeta\left(x_{n}\right)$ and $\varsigma\left(x_{n}\right)$ to be chosen later. Notice that $(u, v) \in S\left(\mu_1,\mu_2\right)$.
		Moreover, from Lemma \ref{l0}, we have $m_0=\sqrt{\kappa}l_0$, then
		$$
		\begin{aligned}
			I(u, v)&=  \frac{1}{2} \int_{\rn}|\nabla u|^2+\kappa|\nabla v|^2+V_2(x)\left(|u|^2+|v|^2\right) d x  -\frac{1}{2} \int_{\rn}u^{2}\overline{v} d x \\
			&=\frac{1}{2} \int_{\rn}\left(-\Delta_{x^{\prime}} u +V_2(x) u\right)\overline{u}\,d x+\frac{1}{2} \int_{\rn}\left(-\kappa\Delta_{x^{\prime}} v +V_2(x) v\right)\overline{v}\,d x\\
			&\quad + \frac{1}{2} \int_{\mathbb{R}}\left|\partial_{x_n} \zeta\right|^2 d x_{n}+\frac{1}{2} \int_{\mathbb{R}}\left|\partial_{x_n} \varsigma \right|^2 d x_{n}-\frac{1}{2} \left(\int_{\R^{n-1}}(\Psi_0)^{2}\Phi_{0} d x^{\prime}\right)\left(\int_{\mathbb{R}} \zeta ^{2}\varsigma d x_{n}\right)\\
			& =\frac{1}{2} \int_{\mathbb{R}}\left|\partial_{x_n} \zeta\right|^2 d x_{n}+\frac{1}{2} \int_{\mathbb{R}}\left|\partial_{x_n} \varsigma \right|^2 \,d x_{n}+\frac{l_0}{2} \int_{\mathbb{R}}|\zeta|^2 \,d x_{n}+\frac{\sqrt{\kappa}l_0}{2} \int_{\mathbb{R}}|\varsigma|^2 \,d x_{n} \\
			& \quad-\frac{1}{2} \left(\int_{\R^{n-1}}(\Psi_0)^{2}\Phi_{0}\,d x^{\prime}\right)\left(\int_{\mathbb{R}} \zeta ^{2}\varsigma \,d x_{n}\right).
		\end{aligned}
		$$
		It is sufficient to choose $(\zeta,\varsigma)$ such that
		\begin{equation}\label{3firtscon}
			\frac{1}{2} \int_{\mathbb{R}}\left|\partial_{x_n} \zeta\right|^2 d x_{n}+\frac{1}{2} \int_{\mathbb{R}}\left|\partial_{x_n} \varsigma \right|^2 \,d x_{n}-\frac{1}{2} \left(\int_{\R^{n-1}}(\Psi_0)^{2}\Phi_{0}\,d x^{\prime}\right)\left(\int_{\mathbb{R}}\zeta^{2}\varsigma \,d x_{n}\right)<0.   
		\end{equation}
		In fact, we define  $$\zeta\left(x_n\right)=\sqrt{\gamma} \psi\left(\gamma x_n\right), \text{where} \int_{\mathbb{R}}\left|\psi\left(x_{n}\right)\right|^2 d x_n=\mu_1$$
		and $$\varsigma\left(x_{n}\right)=\sqrt{\gamma} \phi\left(\gamma x_{n}\right), \text{where} \int_{\mathbb{R}}\left|\phi\left(x_{n}\right)\right|^2 d x_n=\mu_2.$$
		We claim that there exists   $\gamma_0>0$ such that $
		(\zeta,\varsigma)$ satisfies all the conditions above for every $\gamma<\gamma_0$. Then,  
		$$
		\begin{aligned}
			& \frac{1}{2} \int_{\mathbb{R}}\left|\partial_{x_n} \zeta\right|^2 d x_{n}+\frac{1}{2} \int_{\mathbb{R}}\left|\partial_{x_n} \varsigma \right|^2 \,d x_{n}-\frac{1}{2} \left(\int_{\R^{n-1}}(\Psi_0)^{2}\Phi_{0}\,d x^{\prime}\right)\left(\int_{\mathbb{R}}\zeta^{2}\varsigma \,d x_{n}\right) \\
			= & \frac{\gamma^2}{2} \int_{\mathbb{R}}\left|\partial_{x_n} \psi\right|^2 d x_{n}+\frac{\gamma^2}{2} \int_{\mathbb{R}}\left|\partial_{x_n} \phi\right|^2 d x_n-\frac{\sqrt{\gamma}}{2}  \left(\int_{\R^{n-1}}(\Psi_0)^{2}\Phi_{0}\,d x^{\prime}\right)\left( \int_{\mathbb{R}}\psi^{2} \phi\,d x_{n}\right).
		\end{aligned}
		$$
		And  $\eqref{3firtscon}$ is derived for $\gamma\ll1$.
	\end{proof}
	Similarly, as in the previous case, we present the following lemma.
	\begin{lemma}\label{3vani}
		Let $1 \leq n\leq 3$. Assume that $\left\{(u_{m},v_{m})\right\}$ is a minimizing sequence to \eqref{min5}. Then there exist $\delta>0$ such that
		$$
		\liminf _{m\rightarrow \infty} \int_{\rn}\left|u_{m}\right|^{3} \,d x>\delta
		\quad \text{and}  \quad
		\liminf _{m\rightarrow \infty} \int_{\rn}\left|v_{m}\right|^{3}\,d x>\delta.
		$$
	\end{lemma}
	\begin{proof}
		To prove this result, it is sufficient to follow the same strategies used in Lemma \ref{vani}.  
	\end{proof}
	Before proving Theorem \ref{theorem2-sub}$(ii)$, we present the following lemma.
	\begin{lemma}\label{solutionweak}
		Assume that $\sup_{m \in \N} \left\|(u_m,v_m)\right\|_{H}<\infty$ and there exists  $\epsilon_1>0$ such that
		$$
		\int_{\rn}\left|u_m\right|^{3}\,d x>\epsilon_1,  
		\quad
		\int_{\rn}\left|v_m\right|^{3}\,d x>\epsilon_1,\qquad \forall m \in \mathbb{N}^{+},
		$$
		then there exists   sequence $\left\{y_1^{m}\right\} \subset \mathbb{R}$ such that
		$$
		(u_m\left(x^{\prime}, x_n-y_1^{m}\right), v_m\left(x^{\prime}, x_n-y_{1}^{m}\right))\rightharpoonup (\overline{u}, \overline{v}) \neq (0,0)\text { in } H .$$
	\end{lemma}
	\begin{proof}
		See Lemma $3.4$ in \cite{Bell}.
	\end{proof}
	Now we can establish the existence of a solution to problem \eqref{min5}. 
	\begin{proof}[Proof of Theorem \ref{theorem2-sub}(ii)]
		Suppose that $\left\{\left(u_m, v_m\right)\right\} \subset S(\mu_1,\mu_2) $ is a minimizing sequence for $d_{\mu_1, \mu_2}$, this is, $$I\left(u_m, v_m\right) \rightarrow d_{\mu_1, \mu_2} \quad  \text{as}\, \, m  \rightarrow \infty,$$ then from \eqref{boundI}, we have  $\left\|(u_{m},v_{m})\right\|_{H}^2 \leq C$. So $\left\{\left(u_m, v_m\right)\right\}$ is uniformly bounded in $H$. The idea is to show that $\left\{\left(u_m, v_m\right)\right\}$  is compact, up to translation, in $H$.
		
		By Lemmas \ref{3vani} and \ref{solutionweak}, there exist a sequence $\left\{y_1^{m}\right\} \subset \mathbb{R}$ and $ (u,v) \in H \backslash\{(0,0)\}$  such that
		$$
		u_{m}\left(x^{\prime}, x_{n}-y_1^{m}\right) \rightharpoonup u, v_{m}\left(x^{\prime}, x_{n}-y_1^{m}\right) \rightharpoonup v \text { in } H
		$$
		as $m \rightarrow \infty$. Without loss of generality, assume that $u \not\equiv 0$ and notice that $$\|u\|_{\lt}^{2}\leq \liminf_{m \rightarrow \infty}\|u_{m}\|_{\lt}^{2}=\mu_{1}$$
		and $$\|v\|_{\lt}^{2}\leq \liminf_{m \rightarrow \infty}\|v_{m}\|_{\lt}^{2}=\mu_{2},$$
		hence $0<a_1:=\|u\|_{\lt}^{2} \leq \mu_{1}$, $0 \leq a_2:=\|v\|_{\lt}^{2} \leq \mu_{2}.$  We show that $a_1=\mu_1$.
		
		Set
		$$
		\begin{gathered}
			\widetilde{u}_n(x)=u_{m}\left(x^{\prime}, x_{n}-y_1^{m}\right)-u(x), \widetilde{v}_n(x)=v_{m}\left(x^{\prime}, x_{n}-y_1^{m}\right)-v(x).
		\end{gathered}
		$$
		By contradiction, we just suppose $a_1<\mu_{1}$. We consider two cases $a_2<\mu_{2}$ and $a_2=\mu_{2}$ to obtain a contradiction.
		
		If $a_2<\mu_{2}$, by the Brezis-Lieb Lemma and Lemma \ref{igualimit}, we have
		\begin{equation}\label{2seq1}
			\begin{aligned}
				I\left(u_m, v_m\right) & =I(u, v)+I\left(\widetilde{u}_m, \widetilde{v}_m\right)+o_m(1) \\
				& \geq I(u, v)+d_{b_1^m, b_2^m}^\chi+o_m(1),
			\end{aligned}    
		\end{equation}
		where
		$$
		\left(b_1^m, b_2^m\right)=\left(\left\|\widetilde{u}_m\right\|_{\lt}^2,\left\|\widetilde{v}_m\right\|_{\lt}^2\right).
		$$
		Moreover, again using the Brezis-Lieb Lemma, we have 
		\begin{equation}\label{2seq2}
			\left(b_1^m, b_2^m\right)+\left(a_1, a_2\right)=\left(\mu_{1}+o_m(1), \mu_{2}+o_m(1)\right).   
		\end{equation}
		We can assume that $\left(b_1^m, b_2^m\right) \rightarrow\left(b_1, b_2\right)$ and hence $\left(b_1, b_2\right)+\left(a_1, a_2\right)=\left(\mu_{1}, \mu_{2}\right)$ for $m$ sufficiently large. Therefore, by using   Lemma \ref{3propd},  \eqref{2seq1} and \eqref{2seq2}, it follows that 
		$$
		d_{\mu_1, \mu_2} \geq d_{a_1, a_2}+d_{b_1, b_2}.
		$$
		But,  we know from   Lemma \ref{3propd} that
		$$
		d_{\mu_1, \mu_2} < d_{a_1, a_2}+d_{b_1, b_2},
		$$
		hence,
		\begin{equation}\label{2iguality1}
			d_{\mu_1, \mu_2}=d_{a_1, a_2}+d_{b_1, b_2}.   
		\end{equation}
		Since
		\begin{equation}\label{2conv0}
			I\left(u_m, v_m\right)=d_{\mu_1, \mu_2}+o_m(1),
		\end{equation}
		we then obtain from \eqref{2seq1}  that
		$$d_{a_1, a_2}+d_{b_1, b_2} +o_m(1)\geq I(u,v)+d_{b_1^m, b_2^m}+o_m(1),$$
		then, for $m$ sufficiently large, by definition of  $d_{a_1, a_2}$,  we have
		\begin{equation}\label{2iguality3}
			I(u, v)=d_{a_1, a_2}.    
		\end{equation}
		Note that we can assume that $u, v \geq 0$, since if $(u, v)$ satisfies \eqref{2iguality3}, then $(|u|, |v|)$ also satisfies it. In fact, since $|\nabla |u|| \leq |\nabla u|$ and $|\nabla |v|| \leq |\nabla v|$, it is clear that
		$$d_{a_1,a_2}\leq I(|u|,|v|)\leq I(u,v)=d_{a_1,a_2}.$$
		Identity		\eqref{2iguality3} implies that $(u,v)$ is solutions to problem \eqref{elliptic}, then by the maximum principle we can suppose that $v>0$. 
		
		On the other hand, as $a_i<c_i(i=1,2)$, we have $b_i>0$. Then,  suppose that $\left\{\left(\overline{u}_m, \overline{v}_m\right)\right\}$ is a minimizing sequence to $d_{b_1, b_2}$. Similar to above, we can obtain that there exists a sequence $\left\{z_1^m \right\} \subset \mathbb{R}$ and  $(\overline{u},\overline{v}) \in H \backslash \{(0,0)\}$  such that
		$$
		\overline{u}_m\left(x^{\prime}, x_{n}-z_1^m\right) \rightharpoonup \overline{u}, \quad \overline{v}_m\left(x^{\prime}, x_{n}-z_1^m\right) \rightharpoonup \overline{v} \text { in } H
		$$
		as $m \rightarrow \infty$. Moreover, define
		$$
		d_1:=\|\overline{u}\|_{\lt}^2>0, \quad d_2:=\|\overline{v}\|_{\lt}^2 \geq 0.
		$$
		So, proceeding as above, we have
		\begin{equation}\label{2iguality2}
			d_{b_1, b_2}=d_{d_1, d_2}+d_{b_1-d_1, b_2-d_2}   
		\end{equation}
		and 
		\begin{equation}\label{2iguality4}
			I(\overline{u}, \overline{v})=d_{d_1, d_2}, \, \, \,\text{and} \, \,\, \,  \overline{v}>0.   
		\end{equation}
		It follows from \eqref{2iguality1} and \eqref{2iguality2} that
		$$
		d_{\mu_1, \mu_2}=d_{a_1, a_2}+d_{d_1, d_2}+d_{b_1-d_1, b_2-d_2}.
		$$
		Now, let $u^{\star}$ be the Steiner rearrangement of $u$, then using \eqref{sigma2},
		$$
		\int_{\rn}V_{2}(x)\left|u^{\star }\right|^2 \,d x=\int_{\rn}V_{2}(x)\left|u\right|^2 \,d x,
		$$
		and from Theorem $3.4$ in \cite{lieb},
		$$\int_{\rn}u^{2}v\,d x \leq  \int_{\rn}(u^{\star})^{2}v^{\star}\,d x= \int_{\rn}(u^2)^{\star}v^{\star}\,d x.$$
		In consequence, 
		\begin{equation*}
			d_{a_1, a_2} \leq I\left(u^{\star }, v^{\star }\right)\leq I\left(u, v\right) \text { and } d_{d_1, d_2} \leq I\left(\overline{u}^{\star }, \overline{v}^{\star }\right)\leq I(\overline{u},\overline{v}).    
		\end{equation*}
		Therefore, from \eqref{2iguality3} and \eqref{2iguality4}, it follows that 
		\begin{equation}\label{2equality5}
			I\left(u^{\star }, v^{\star }\right) =d_{a_1, a_2} \text { and } I\left(\overline{u}^{\star }, \overline{v}^{\star }\right)=d_{d_1, d_2}.    
		\end{equation}
		This implies that $\left(u^{\star }, v^{\star }\right)$  and $\left(\overline{u}^{\star }, \overline{v}^{\star }\right)$  are solution of system \eqref{elliptic}, then by elliptic regularity theory, they are functions of class $C^{2}(\rn)$ (see Proposition \eqref{regularity}) and by the maximum principle we can suppose that $u^{\star}, v^{\star}, \overline{u}^{\star}, \overline{v}^{\star} >0$. 
		
		Now, using   the items  $(ii)$ and $(iii)$ in  Lemma \ref{schwartz}, we obtain
		$$
		\int_{\rn}\left|\{u^{\star},\overline{u}^{\star }\}^{\star}\right|^2 d x  
		=\int_{\rn}\left(\left|u^{\star }\right|^2+\left|\overline{u}^{\star }\right|^2\right) d x= a_1+d_1, 
		$$
		$$\int_{\rn}\left|\{v^{\star},\overline{v}^{\star }\}^{\star}\right|^2 d x  
		=\int_{\rn}\left(\left|v^{\star }\right|^2+\left|\overline{v}^{\star }\right|^2\right) d x= a_2+d_2,$$
		$$
		\begin{aligned}
			\int_{\rn}V_{2}(x)\left|\{u^{\star},\overline{u}^{\star }\}^{\star}\right|^2 d x = \int_{\rn}V_{2}(x)\left(\left|u^{\star }\right|^2+\left|\overline{u}^{\star }\right|^2\right) d x,
		\end{aligned}
		$$
		$$
		\begin{aligned}
			\int_{\rn}V_{2}(x)\left|\{v^{\star} ,\overline{v}^{\star }\}^{\star}\right|^2 d x = \int_{\rn}V_{2}(x)\left(\left|v^{\star }\right|^2+\left|\overline{v}^{\star }\right|^2\right) d x,
		\end{aligned}
		$$
		$$\int_{\rn}((u^{\star})^{2}v^{\star}+(\overline{u}^{\star})^2 \overline{v}^{\star})\,d x \leq  \int_{\rn}\{(u^{\star})^{2},(\overline{u}^{\star})^2\}^{\star}\{v^{\star},\overline{v}^{\star}\}^{\star}\,d x=\int_{\rn}(\{u^{\star},\overline{u}^{\star}\}^{\star})^2 \{v^{\star},\overline{v}^{\star}\}^{\star}\,d x$$
		and from item $(v)$ in Lemma \ref{schwartz},
		$$\int_{\rn}\left|\nabla \{v^{\star},\overline{v}^{\star }\}^{\star}\right|^2 d x < \int_{\rn}\left|\nabla v^{\star}\right|^2 d x+ \int_{\rn}\left|\nabla \overline{v}^{\star}\right|^2 d x, $$
		$$\int_{\rn}\left|\nabla \{u^{\star},\overline{u}^{\star }\}^{\star}\right|^2 d x < \int_{\rn}\left|\nabla u^{\star}\right|^2 d x+ \int_{\rn}\left|\nabla \overline{u}^{\star}\right|^2 d x .$$
		Hence, we deduce from \eqref{2iguality3}, \eqref{2equality5} and the estimates above, that
		$$
		\begin{aligned}
			d_{a_1+d_1, a_2+d_2} & \leq I\left(\{u^{\star}, \overline{u}^{\star }\}^{\star}, \{v^{\star}, \overline{v}^{\star }\}^{\star}\right)<I\left(u^{\star }, v^{\star }\right)+I\left(\overline{u}^{\star }, \overline{v}^{\star }\right) \\
			& =d_{a_1, a_2}+d_{d_1, d_2} \\
			& =d_{\mu_1, \mu_2}-d_{b_1-d_1, b_2-d_2},
		\end{aligned}
		$$
		which contradicts $(ii)$ in Lemma \ref{3propd}.
		Now, if $a_2=\mu_{2}$. In this case, we can similarly obtain a contradiction by the same arguments as $a_2<\mu_{2}$. As a consequence, we have proved that $\mu_{1}=a_1$.
		As a result, we obtain $\|u\|_{\lt}^{2}=\mu_1$, and 
		\begin{equation}\label{2conv1}
			u_{m}\left(x^{\prime}, x_{n}-y_1^{m}\right) \rightarrow u(x) \text { in } L^2\left(\rn\right) \text { as } m \rightarrow \infty .   
		\end{equation}
		and using
		the Gagliardo-Nirenberg inequality  we get that 
		\begin{equation}\label{2conv2}
			u_{m}\left(x^{\prime}, x_{n}-y_1^{m}\right) \rightarrow u(x) \text { in } L^3\left(\rn\right) \text { as } m \rightarrow \infty .   
		\end{equation}
		We now reverse the role of $\left\{u_{m}\right\}$ and $\left\{v_{m}\right\}$, then there exist a sequence $\left\{y_2^{m}\right\} \subset \mathbb{R}$ and $(\tilde{u},\tilde{v}) \in H$, such that  $\|\tilde{v}\|_{\lt}^{2}=\mu_{2}$, and
		$$
		u_{m}\left(x^{\prime}, x_{n}-y_2^{m}\right) \rightharpoonup \tilde{u}(x), v_{m}\left(x^{\prime}, x_{n}-y_2^{m}\right) \rightharpoonup \tilde{v}(x) \quad\text { in } H
		$$
		as $m \rightarrow \infty$. Moreover,
		\begin{equation}\label{2conv3}
			v_{m}\left(x^{\prime}, x_{n}-y_2^{m}\right) \rightarrow \tilde{v}(x) \text { in } L^2\left(\rn\right) \text { as } m \rightarrow \infty.    
		\end{equation}
		and   we get  from Gagliardo-Nirenberg inequality that
		\begin{equation}\label{2conv4}
			v_{m}\left(x^{\prime}, x_{n}-y_2^{m}\right) \rightarrow \tilde{v}(x) \text { in } L^3\left(\rn\right) \text { as } m \rightarrow \infty.
		\end{equation}
		On the other hand, we note that 
		\begin{equation}\label{2trans}
			\lim _{m \rightarrow \infty}\left|y_1^{m}-y_2^{m}\right|<\infty .    
		\end{equation}
		If this does not hold, we get that $\tilde{u}=v=0$ due to $(u, \tilde{v}) \in S(\mu_1,\mu_{2})$. Indeed, suppose $v\not\equiv 0$, then since $$v_{m}\left(x^{\prime}, x_{n}-y_1^{m}\right) \rightharpoonup v(x) \text { in } \lt,$$
		hence $$v_{m}\left(x^{\prime}, x_{n}-y_2^{m}\right) \rightharpoonup v\left(x^{\prime}, x_{n}+y_{1}^{m}-y_2^{m}\right) \text { in } \lt,$$
		then from \eqref{2conv3}, we have $\tilde{v}(x)=v\left(x^{\prime}, x_{n}+y_{1}^{m}-y_2^{m}\right)$, thereby since $\left|y_1^{m}-y_2^{m}\right|  \rightarrow \infty$ as $m  \rightarrow \infty$, it follows that $$\mu_{2}=\int_{\rn}|\tilde{v}(x)|^{2}\,d x=\int_{\rn}|v\left(x^{\prime}, x_{n}+y_{1}^{m}-y_2^{m}\right)|^{2}\,d x=0,$$
		which is a contradiction. Consequently, we get from the Brezis-Lieb Lemma and Lemma \ref{igualimit} that
		$$
		\begin{aligned}
			I\left(u_m, v_m\right)= & I\left(u_{m}\left(x^{\prime}, x_{n}-y_1^{m}\right), v_{m}\left(x^{\prime}, x_{n}-y_1^{m}\right)\right) \\
			= & I\left(u_{m}\left(x^{\prime}, x_{n}-y_1^{m}\right)-u, v_{m}\left(x^{\prime}, x_{n}-y_1^{m}\right)\right)+I(u, 0)+o_m(1) \\
			= & I\left(u_{m}\left(x^{\prime}, x_{n}-y_2^{m}\right)-u\left(x^{\prime}, x_{n}+y_1^{m}-y_2^{m}\right), v_{m}\left(x^{\prime}, x_{n}-y_2^{m}\right)\right) +I(u, 0)+o_m(1) \\
			= & I(0, \tilde{v})+I\left(u_{m}\left(x^{\prime}, x_{n}-y_2^{m}\right)-u\left(x^{\prime}, x_{n}+y_1^{m}-y_2^{m}\right), v_{m}\left(x^{\prime}, x_{n}-y_2^{m}\right)-\tilde{v}\right) \\
			& +I(u, 0)+o_m(1)\\
			&\geq I(0, \tilde{v})+I(u, 0) +o_{m}(1)\\&\quad -\int_{\rn}\left(u_{m}\left(x^{\prime}, x_{n}-y_2^{m}\right)-u\left(x^{\prime}, x_{n}+y_1^{m}-y_2^{m}\right)\right)^{2}\left(v_{m}\left(x^{\prime}, x_{n}-y_2^{m}\right)-\tilde{v}\right).
		\end{aligned}
		$$
		Hence, it follows from \eqref{2conv2} and \eqref{2conv4} that
		$$
		d_{\mu_1, \mu_2} \geq I(u, 0)+I(0, \tilde{v}) .
		$$
		So,  from Lemma \ref{steiner}, we derive
		$$
		d_{\mu_1, \mu_2}\geq I(u, 0)+I(0, \tilde{v})  \geq I\left(u^{\star }, 0\right)+I\left(0, \tilde{v}^*\right) \geq d_{\mu_{1}, 0}+d_{0, \mu_{2}}.
		$$
		Hence, we infer from (ii) in Lemma \ref{3propd}  that
		\begin{equation}\label{2contrad}
			d_{\mu_1, \mu_2}=I\left(u^{\star }, 0\right)+I\left(0, \tilde{v}^*\right) .
		\end{equation}
		Since  $\left(u^{\star }, \tilde{v}^*\right) \in S(\mu_1,\mu_2)$, we have 
		$$
		d_{\mu_1, \mu_2} \leq I\left(u^{\star }, \tilde{v}^*\right)<I\left(u^{\star }, 0\right)+I\left(0, \tilde{v}^*\right),
		$$
		(remembering that $u^{\star}, v^{\star} > 0$). This contradicts \eqref{contrad}. Therefore, \eqref{trans} holds.
		
		Now, we can find a $y \in \mathbb{R}$ such that $y_1^{m}=y_2^{m}+y+o_m(1)$. Define
		$$
		(\varphi_m(x),\psi_m(x)):=(u_{m}\left(x^{\prime}, x_{n}-y_1^{m}\right), v_{m}\left(x^{\prime}, x_{n}-y_1^{m}\right)) \in S\left(\mu_{1},\mu_{2}\right) .
		$$
		Accordingly there is a $(\varphi(x), \psi(x)):=\left(u(x), \tilde{v}\left(x^{\prime}, x_{n}-y\right)\right) \in S\left(\mu_{1}\right) \times S\left(\mu_{2}\right)$ such that
		$$
		\begin{gathered}
			\left(\varphi_n, \psi_n\right) \rightharpoonup (\varphi, \psi) \text { in } H, \\
			\left(\varphi_n, \psi_n\right) \rightarrow(\varphi, \psi) \text { in } L^3\left(\rn\right) \times L^3\left(\rn\right).
		\end{gathered}
		$$
		We conclude from the Brezis-Lieb Lemma and Lemma \ref{igualimit} that
		$$
		\begin{aligned}
			I\left(u_m, v_m\right)=I\left(\varphi_n, \psi_n\right) & =I(\varphi, \psi)+I\left(\varphi_n-\varphi, \psi_n-\psi\right)+o_m(1) \\
			& \geq d_{\mu_1, \mu_2}+\frac{\min\{1,\kappa\}}{2}\left\|(\varphi_m,\psi_m)-(\varphi,\psi)\right\|_{\dot{H}}^2\\
			&\quad-\int_{\rn}(\varphi_{m}-\varphi)^{2}(\psi_{m}-\psi)\,d x+o_m(1).
		\end{aligned}
		$$
		From \eqref{2conv0}, we obtain that $\left(\varphi_n, \psi_n\right) \rightarrow(\varphi, \psi)$ in $H$ as $m\rightarrow \infty$. Hence, any minimizing sequence for $d_{\mu_1, \mu_2}$, up to translation, is compact and $\mathcal{D}_{\mu_{1}, \mu_{2}} \neq \emptyset$ follows.  
	\end{proof}

	\subsection*{\underline{Critical and Supercritical cases}}
	Now, we proceed to ensure the existence of  normalized {\blue{solution for \eqref{elliptic} in the  case $n=4,5$.}} 
	
	Since $H \hookrightarrow L^{q}(\rn)$    is not compact, we must take a different approach than before. The idea is to establish the compactness of any minimizing sequence, up to translation.
	We consider once again the local minimization problem \eqref{min3}.
	
	\begin{lemma}\label{nosetempty} 
		Let  $\mu_{1}, \mu_{2}>0$. For any $\chi>0$ with $\mu_{1}+\mu_{2} \leq \frac{\chi}{\epsilon_{0}l_0}$, there holds that
		$$
		S\left(\mu_{1},\mu_2\right) \cap B(\chi) \neq \emptyset .
		$$
		In particular, $I(u, v)$ is bounded from below on $S\left(\mu_{1},\mu_2\right)\cap B(\chi)$.    
	\end{lemma}
	\begin{proof}
		For any $\chi>0$, let $\Psi_{0}$ satisfy  
		
		$$
		-\Delta_{x^{\prime}} \Psi_{0}+V(x) \Psi_{0}=l_{0} \Psi_{0}, \quad \int_{\R^{n-1}}\left|\Psi_{0}\right|^{2} dx^{\prime}=1.
		$$
		Note that, from the equality above, it is clear that $\left\|\Psi_{0}\right\|_{\dot{H}}^{2}=l_0$. Then,
		$$
  \left\|\sqrt{\mu_{1}} \Psi_{0}\right\|_{\dot{H}}^{2}+\left\|\sqrt{\mu_{2}} \Psi_{0}\right\|_{\dot{H}}^{2}=\left(\mu_{1}+\mu_{2}\right)\left\|\Psi_{0}\right\|_{\dot{H}}^{2}=l_{0}\left(\mu_{1}+\mu_{2}\right) \leq \frac{\chi}{\epsilon_{0}}.
  $$ 
				Hence,  $\left(\sqrt{\mu_{1}} \Psi_{0}, \sqrt{\mu_{2}} \Psi_{0}\right) \in S\left(\mu_{1},\mu_{2}\right)\cap B(\chi)$. It follows from \eqref{welldef} that $I(u, v)$ is bounded from below on $S\left(\mu_{1},\mu_2\right)\cap B(\chi)$.
	\end{proof}
	The previous lemma establishes that   problem \eqref{min3} is well-defined if  for any $\chi>0$, it is verified that $\mu_{1}+\mu_{2} \leq \frac{\chi}{\epsilon_{0}l_0}$.
	\begin{lemma}\label{pointcrit}
		Let $n\geq4$. If $(u, v) \in S(\mu_{1},\mu_{2})$ is a critical point of $I(u, v)$ constrained on $S(\mu_{1},\mu_{2})$, then
		$$
		\begin{gathered}
			\blue{B(u, v)}:=\int_{\rn}\left(|\nabla u|^{2}+\kappa|\nabla v|^{2}\right) \,d x-\int_{\rn}V(x)\left(|u|^{2}+|v|^{2}\right) \,d x-\frac{n}{4}K(u,v)=0 .
		\end{gathered}
		$$
				Furthermore, $d_{\mu_{1}, \mu_{2}}^{\chi}>0$.
	\end{lemma}
	\begin{proof}
		Since $(u, v) \in S(\mu_{1},\mu_{2})$ is a critical point of $I(u, v)$ constrained on $S(\mu_{1},\mu_{2})$, from the Lagrange multiplier theorem there exists $\left(\lambda_{\mu_{1}}, \lambda_{\mu_{2}}\right) \in \mathbb{R}^{2}$ such that $(u, v)$ satisfies  $$I^{\prime}(u,v)=\lambda_{\mu_1}\|u\|_\lt^2+\lambda_{\mu_2}\|v\|_\lt^2.$$
		This means such that $(u, v)$ satisfies
		the following system
		\begin{equation}\label{system3}
			\left\{\begin{array}{l}
				-\Delta u+V(x) u=\lambda_{\mu_{1}} u+uv,\\
				-\kappa\Delta v+V(x) v=\lambda_{\mu_{2}} v+\frac{1}{2}u^2.
			\end{array}\right.
		\end{equation}
 Define $(\delta_{a}u)(x)=u(\frac{x}{a})$ and $(\delta_{a}v)(x)=v(\frac{x}{a})$. Then, the function $a  \rightarrow h(a)=\tilde{I}(\delta_{a}u,\delta_{a}v)$ has a critical point at $a=1$,  where 
		$$
		\begin{aligned}
			\tilde{I}(u,v)&=\frac{1}{2} \int_{\rn}(|\nabla u|^{2}+\kappa|\nabla v|^{2})\,d x+\frac{1}{2} \int_{\rn}V(x)\left(|u|^{2}+|v|^{2}\right) \,d x -\frac{1}{2} K(u,v)\\
			&\quad-\frac{1}{2}\lambda_{\mu_1}\int_{\rn}|u|^{2}\,d x-\frac{1}{2}\lambda_{\mu_2}\int_{\rn}| v|^{2}\,d x.   
		\end{aligned}$$   
  Moreover, 
		$$
		\begin{aligned}
			0=h^{\prime}(1)&=\frac{n-2}{2} \int_{\rn}|\nabla u|^{2}+\kappa|\nabla v|^{2} \,d x+\frac{n+2}{2} \int_{\rn}V(x)\left(|u|^{2}+|v|^{2}\right) \,d x -\frac{n}{2}K(u,v)\\
			&\quad-\frac{n}{2}\lambda_{\mu_1}\int_{\rn}|u|^{2}\,d x-\frac{n}{2}\lambda_{\mu_2}\int_{\rn}| v|^{2}\,d x.
		\end{aligned}
		$$
		On the other hand, testing \eqref{system3} with $(u, v)$, we find
		
		$$
		\begin{aligned}
			& \lambda_{\mu_{1}} \int_{\rn}|u|^{2} \,d x=\int_{\rn}\left(|\nabla u|^{2}+V(x)|u|^{2}\right) \,d x-K(u,v), \\
			& \lambda_{\mu_{2}} \int_{\rn}|v|^{2} \,d x=\int_{\rn}\left(\kappa|\nabla v|^{2}+V(x)|v|^{2}\right) \,d x-\frac{1}{2}K(u,v).
		\end{aligned}
		$$
 Then, we have
		$$
		\begin{aligned}
			& \int_{\rn}\left(|\nabla u|^{2}+\kappa|\nabla v|^{2}\right) \,d x-\int_{\rn}V(x)\left(|u|^{2}+|v|^{2}\right) \,d x-\frac{n}{4} K(u,v) =0,
		\end{aligned}
		$$
		so that, $\blue{B(u,v)}=0$. Therefore
		$$
		\begin{aligned}
			I(u, v) & =I(u, v)-2\frac{\blue{B(u, v)}}{n} \\
			& =\frac{n-4}{2n} \int_{\rn}\left(|\nabla u|^{2}+\kappa|\nabla v|^{2}\right) \,d x+\frac{n+4}{2n}\int_{\rn}V(x)\left(|u|^{2}+|v|^{2}\right) \,d x   >0.
		\end{aligned}
		$$
		This implies that $d_{\mu_{1}, \mu_{2}}^{\chi} \geq d_{\mu_{1}, \mu_{2}}>0$.
	\end{proof}
	\begin{lemma}\label{ineqinf}
		Let		$\epsilon_{0}=\min\{1,\kappa\}$ and $\mu_{1}, \mu_{2}>0$. If $S(\mu_1,\mu_2) \cap\left(B(\chi) \backslash B\left(\frac{\chi}{2 \epsilon_{0}}\right)\right) \neq \emptyset$,   then there exists for any $\chi>0$ a $\bar{\mu}:=\bar{\mu}(\chi,\kappa)$ such that for $\mu_{1}+\mu_{2} \leq \bar{\mu}$,
		
		$$
		\inf _{(u, v) \in S(\mu_1,\mu_2) \cap B\left(\frac{\chi}{4\delta_{0}}\right)} I(u, v)<\inf _{(u, v) \in S(\mu_1,\mu_2) \cap \left(B(\chi) \backslash B\left(\frac{\chi}{2 \epsilon_{0}}\right)\right)} I(u, v),
		$$
		where  $\delta_{0}=\max\{1,\kappa\}$.
	\end{lemma}
	\begin{proof}
		For any $(u, v) \in S(\mu_1,\mu_2)\cap\left(B(\chi) \backslash B\left(\frac{\chi}{2 \epsilon_{0}}\right)\right)$,  we have from the Gagliardo-Nirenberg inequality that
		$$
		\begin{aligned}
			I(u, v)= & \frac{1}{2} \int_{\rn}|\nabla u|^{2}+\kappa|\nabla v|^{2}+V(x)\left(|u|^{2}+|v|^{2}\right) d x - \frac{1}{2}K(u,v)\\
			\geq & \frac{\epsilon_{0}}{2}\|(u,v)\|_{\grave{H}}^{2}-\frac{1}{2}C_{\circ}\mu_{1}^{(6-n)/6}\mu_{2}^{(6-n)/12}\left(\|(u,v)\|_{\grave{H}}^{2}\right)^{n/4}.
		\end{aligned}
		$$
		Since
		\begin{equation*}
			\begin{aligned}
				&\frac{\epsilon_{0}}{2}\|(u,v)\|_{\grave{H}}^{2}-\frac{1}{2}C_{\circ}\mu_{1}^{(6-n)/6}\mu_{2}^{(6-n)/12}\left(\|(u,v)\|_{\grave{H}}^{2}\right)^{n/4}\\& =\frac{1}{2}\|(u,v)\|_{\grave{H}}^{2}\left(\epsilon_{0}-C_{\circ}\mu_{1}^{(6-n)/6}\mu_{2}^{(6-n)/12}\left(\|(u,v)\|_{\grave{H}}^{(n-4)/4}\right)\right),
			\end{aligned}
		\end{equation*}
		we can take $$\mu^{\frac{18-3n}{12}} < \frac{1}{C_{\circ}\|(u,v)\|_{\grave{H}}^{(n-4)/8}}\left(\frac{1}{2}\epsilon_{0}\right), \,\, \text{with} \, \, \frac{\chi}{2 \epsilon_{0}}<\|(u,v)\|_{\grave{H}}^{2}\leq \chi,$$
		so,  there exists a $\bar{\mu}(\chi,\kappa)>0$ such that for $\mu_{1}+\mu_{2} \leq \bar{\mu}$, $$I(u, v) \geq \frac{3 \chi}{8}.$$ 
		On the other hand, if   $(u, v) \in S(\mu_1,\mu_2) \cap B\left(\frac{\chi}{4 \delta_{0}}\right)$ with $u,v\geq 0$, then 
		$$
		\begin{aligned}
			I(u, v)= & \frac{1}{2} \int_{\rn}|\nabla u|^{2}+\kappa|\nabla v|^{2}+V(x)\left(|u|^{2}+|v|^{2}\right) d x - \frac{1}{2}K(u,v)\\
			\leq & \frac{ \delta_{0}}{2}\|(u,v)\|_{\dot{H}}^{2}  
			\leq   \frac{\chi}{8}.
		\end{aligned}
		$$
		Collecting all above inequalities, we have
		$$
		\inf _{(u, v) \in S(\mu_1,\mu_2) \cap B\left(\frac{\chi}{4\delta_{0}}\right)} I(u, v)<\inf _{(u, v) \in S(\mu_1,\mu_2) \cap \left(B(\chi) \backslash B\left(\frac{\chi}{2 \epsilon_{0}}\right)\right)} I(u, v).
		$$
	\end{proof}
	Now we establish the following lemma, which provides us with some properties of $d_{\mu_1, \mu_2}^{\chi}$.
	\begin{lemma}\label{propd}
		Let $\mu_{1}, \mu_{2} \geq 0$, for every $\chi>0$, we have
		\begin{itemize}
			\item[(i)] $\left(\mu_{1}, \mu_{2}\right) \mapsto d_{\mu_{1}, \mu_{2}}^{\chi}$ is continuous.
			\item[(ii)] For $i=1,2$ such that $\mu_{i}=a_{i}+b_{i}$ and $a_{i}, b_{i} \geq 0$, we have $d_{\mu_{1}, \mu_{2}}^{\chi} \leq d_{a_{1}, a_{2}}^{\chi}+d_{b_{1}, b_{2}}^{\chi}$.
			\item[(iii)] For any $\chi>0$, there exists a $\mu^{\star}=\mu^{\star}(\chi,\kappa)$, such that if $\mu_{1}+\mu_{2} \leq \mu^{\star}$, then $$d_{\mu_{1}, \mu_{2}}^{\chi}<\frac{l_0 \mu_{1}}{2}+\frac{\sqrt{\kappa}l_0\mu_{2}}{2}.$$
		\end{itemize}
	\end{lemma}
	\begin{proof}
		
		(i)   Let $\mu_{1}, \mu_{2} \geq 0$. We assume $\left(\mu_{1}^{m}, \mu_{2}^{m}\right)$ is a sequence satisfying $\mu_{1}^{m}, \mu_{2}^{m} \geq 0$ and $\left(\mu_{1}^{m}, \mu_{2}^{m}\right) \rightarrow\left(\mu_{1}, \mu_{2}\right)$ as $m \rightarrow \infty$. Now, by definition of $d_{\mu_{1}^{m}, \mu_{2}^{m}}^{\chi}$, we have that for any $\epsilon>0$, there exists a sequence $\left(u_{1}^{m}, v_{1}^{m}\right)$ such that $$I\left(u_{1}^{m}, v_{1}^{m}\right) \leq d_{\mu_{1}^{m}, \mu_{2}^{m}}^{\chi}+\frac{\epsilon}{2}.$$
		Using Lemma \ref{ineqinf}, we can assume that $\left\{\left(u_{1}^{m}, v_{1}^{m}\right)\right\} \subset S\left(\mu_{1}^{m},\mu_{2}^{m}\right)\cap B\left(\frac{\chi}{2\epsilon_{0}}\right)$.
		Taking  $$\tilde{u}_{1}^{m}:=\left(\frac{\mu_{1}}{\mu_{1}^{m}}\right)^{\frac{1}{2}} u_{1}^{m}, \, \, \tilde{v}_{1}^{m}:=\left(\frac{\mu_{2}}{\mu_{2}^{m}}\right)^{\frac{1}{2}} v_{1}^{m},$$ we have $\left(\tilde{u}_{1}^{m}, \tilde{v}_{1}^{m}\right) \in S\left(\mu_{1},\mu_{2}\right)$ and
		\begin{equation}\label{newseque}
			\begin{aligned}
				&\left\|\tilde{u}_{1}^{m}\right\|_{\dot{H}}^{2}+\left\|\tilde{v}_{1}^{m}\right\|_{\dot{H}}^{2}\\&=\int_{\rn}\frac{\mu_1}{\mu_{1}^{m}}|\nabla u_{1}^{m}|^{2}\,d x+\int_{\rn}\frac{\mu_2}{\mu_{2}^{m}}|\nabla v_{1}^{m}|^{2}\,d x+\int_{\rn}\frac{\mu_1}{\mu_{1}^{m}}V_{2}(x)|u_{1}^{m}|^{2}\,d x\\
				&\quad+\int_{\rn}\frac{\mu_2}{\mu_{2}^{m}}V_{2}(x)|v_{1}^{m}|^{2}\,d x\\
				&=\left\|u_{1}^{m}\right\|_{\dot{H}}^{2}+\left\|v_{1}^{m}\right\|_{\dot{H}}^{2}+\left(\frac{\mu_{1}}{\mu_{1}^{m}}-1\right)\int_{\rn}|\nabla u_{1}^{m}|^{2}\,d x +\left(\frac{\mu_{2}}{\mu_{2}^{m}}-1\right)\int_{\rn}|\nabla v_{1}^{m}|^{2}\,d x   \\
				&\quad +\left(\frac{\mu_{1}}{\mu_{1}^{m}}-1\right)\int_{\rn}V_2(x)|u_{1}^{m}|^{2}\,d x +\left(\frac{\mu_{2}}{\mu_{2}^{m}}-1\right)\int_{\rn}V_2(x)|v_{1}^{m}|^{2}\,d x.  
			\end{aligned}
		\end{equation}
		Since $\left(\mu_{1}^{m}, \mu_{2}^{m}\right) \rightarrow \left(\mu_{1}, \mu_{2}\right)$ as $m \rightarrow \infty$, for $m \in \mathbb{N}$ sufficiently large, it  follows   from \eqref{newseque} that		$$\left\|\tilde{u}_{1}^{m}\right\|_{\dot{H}}^{2}+\left\|\tilde{v}_{1}^{m}\right\|_{\dot{H}}^{2} <\frac{\chi}{\epsilon_{0}}.$$
		Moreover, for sufficiently large $m$,
		\begin{equation*}
			\begin{aligned}
				d_{\mu_{1}, \mu_{2}}^{\chi} &\leq I\left(\tilde{u}_{1}^{m}, \tilde{v}_{1}^{m}\right)\\&=\frac{\mu_1}{2\mu_{1}^{m}}\int_{\rn}|\nabla u_{1}^{m}|^{2}\,d x+\frac{\kappa \mu_2}{\mu_{2}^{m}}\int_{\rn}|\nabla v_{1}^{m}|^{2}\,d x+\frac{\mu_1}{\mu_{1}^{m}}\int_{\rn}V_{2}(x)|u_{1}^{m}|^{2}\,d x\\
				&\quad +\frac{\mu_2}{\mu_{2}^{m}}\int_{\rn}V_{2}(x)|v_{1}^{m}|^{2}\,d x-\frac{\mu_{1}}{\mu_{1}^{m}}\left(\frac{\mu_{2}}{\mu_{2}^{m}}\right)^{1/2}K(u_{1}^{m},v_{1}^{m})\\
				&= I\left(u_{1}^{m}, v_{1}^{m}\right)+ \left(\frac{\mu_{1}}{\mu_{1}^{m}}-1\right)\int_{\rn}|\nabla u_{1}^{m}|^{2}\,d x +\left(\frac{\mu_{2}}{\mu_{2}^{m}}-1\right)\int_{\rn}|\nabla v_{1}^{m}|^{2}\,d x   \\
				&\quad +\left(\frac{\mu_{1}}{\mu_{1}^{m}}-1\right)\int_{\rn}V_2(x)|u_{1}^{m}|^{2}\,d x +\left(\frac{\mu_{2}}{\mu_{2}^{m}}-1\right)\int_{\rn}V_2(x)|v_{1}^{m}|^{2}\,d x\\
				&\quad - \left(\frac{\mu_{1}\sqrt{\mu_{2}}}{\mu_{1}^{m}\sqrt{\mu_{2}^{m}}}-1\right)K(u_{1}^{m},v_{1}^{m})\\
				&\leq d_{\mu_{1}^{m}, \mu_{2}^{m}}^{\chi}+\epsilon.    
			\end{aligned}
		\end{equation*}
On the other hand, repeating the argument with simpler calculations than those above, one deduces that for sufficiently large $m$,
		$$
		d_{\mu_{1}^{m}, \mu_{2}^{m}}^{\chi} \leq d_{\mu_{1}, \mu_{2}}^{\chi}+\epsilon.
		$$
		Therefore,  $d_{\mu_{1}^{m}, \mu_{2}^{m}}^{\chi} \rightarrow d_{\mu_{1}, \mu_{2}}^{\chi}$ as $m\rightarrow \infty$.
		
		\item[(ii)]  Let $\left\{\left(u_1^m, v_1^m\right)\right\} \subset S(a_1,a_2) \cap  B(\chi)$ be such that $I\left(u_1^m, v_1^m\right)  \rightarrow d_{a_1, a_2}^\chi$ as $m \rightarrow \infty$. Without loss of generality, we can assume that  $u_1^m,u_2^m \geq 0$ (by using the Steiner rearrangement). Moreover, by Lemma \ref{ineqinf}, we can assume that $\left(u_1^m, v_1^m\right) \in S(a_1,a_2) \cap$ $B\left(\frac{\chi}{2\epsilon_{0}}\right)$ for every $m$ large enough. So, for any $\epsilon>0$, there exists a $\left(u_1, v_1\right) \in S\left(a_1,a_2\right) \cap B\left(\frac{\chi}{2\epsilon_0}\right)$ such that
		$$
		I\left(u_1, v_1\right) \leq d_{a_1, a_2}^\chi+\frac{\epsilon}{2} .
		$$
		Similarly, we can find $\left(u_2, v_2\right) \in S\left(b_1,b_2\right) \cap B\left(\frac{\chi}{2\epsilon_0}\right)$ such that
		$$
		I\left(u_2, v_2\right) \leq d_{b_1, b_2}^\chi+\frac{\epsilon}{2} .
		$$
		Without loss of generality, we assume that $\operatorname{supp} u_1 \cap \operatorname{supp} u_2 = \emptyset$ and $\operatorname{supp} v_1 \cap \operatorname{supp} v_2 = \emptyset$. Otherwise, we can redefine $(u_1, v_1)$ and $(u_2, v_2)$ so that this condition is met. For instance, we can define $\tilde{u}_1$ as $u_1$ where $\operatorname{supp} u_1 \cap \operatorname{supp} u_2 = \emptyset$ and as $0$ elsewhere. Similarly, we can make analogous adjustments for the others.		
		Setting
		$ 
		u=u_1+u_2$ and $v=v_1+v_2,
		$ 
		we obtain that $(u, v) \in S(\mu_1,\mu_2) \cap B(\frac{\chi}{\epsilon_0})$, and
		\begin{equation*}
			\begin{aligned}
				d_{\mu_1, \mu_2}^\chi &\leq I(u, v) \\
				&=\frac{1}{2} \int_{\rn}|\nabla u_1+\nabla u_2|^{2}+\kappa|\nabla v_1+\nabla v_2|^{2}+V_2(x)\left(|u_1+u_2|^{2}+|v_1+v_2|^{2}\right) d x \\
				&\qquad - \frac{1}{2}K(u_1+u_2,v_1+v_2)\\
				&\leq I\left(u_1, v_1\right)+I\left(u_2, v_2\right)+\blue{2}\Re\int_{\rn}\nabla u_1  \nabla \blue{\overline{u}_2}\,d x + \blue{2\kappa}\Re\int_{\rn}\nabla v_1  \nabla \blue{\overline{v}_2}\,d x+ \blue{2}\Re\int_{\rn} V_2(x)(u_1   \blue{\overline{u}_2})\,d x\\
				&\quad + \blue{2\Re\int_{\rn} V_2(x)(v_1   \overline{v}_2)\,d x}- \blue{\Re\int_{\rn} u_1u_2   \overline{v}_1\,d x}- \blue{\Re\int_{\rn} u_1u_2   \overline{v}_2\,d x}  - \frac{1}{2}K (\blue{u_2} ,v_1) - \frac{1}{2}K( \blue{u_1}, v_2)\\
				&\leq d_{a_1, a_2}^\chi+d_{b_1, b_2}^\chi+\epsilon.  
			\end{aligned}
		\end{equation*}
		Then, it follows that $d_{\mu_1, \mu_2}^\chi \leq d_{a_1, a_2}^\chi+d_{b_1, b_2}^\chi$.
		
		\item[(iii)] We introduce $\Psi_j\left(x^{\prime}\right), \Phi_j\left(x^{\prime}\right)$ and $l_j, m_j$ for $j \geq 0$ such that
		$$
		\begin{gathered}
			\begin{array}{ll}
				-\Delta_{x^{\prime}} \Psi_j+V_2(x) \Psi_j=l_j \Psi_j,\, \,  \int_{\R^{n-1}}\left|\Psi_j\right|^2 d x^{\prime}=1,    \\
				-\kappa\Delta_{x^{\prime}} \Phi_j+V_2(x) \Phi_j=m_j \Phi_j, \, \, \int_{\R^{n-1}}\left|\Phi_j\right|^2 d x^{\prime}=1,  
			\end{array}
		\end{gathered}
		$$
		where  $l_j \leq l_{j+1},\,m_j \leq m_{j+1}$, with $j=0,1,2, \ldots$.		
		It is well known that $\left\{\Psi_j\right\}_j$ and $\left\{\Phi_j\right\}_j$ are   Hilbert bases for $L^2\left(\R^{n-1}\right)$. From Lemma \ref{l0}, we know that $L_0=l_0$. Set
		$$
		u(x)=\Psi_0(x^{\prime}) \zeta\left(x_{n}\right), \int_{\mathbb{R}}|\zeta|^2 d x_{n}=\mu_1,
		$$
		$$
		v(x)=\Phi_0(x^{\prime}) \varsigma\left(x_{n}\right), \int_{\mathbb{R}}|\varsigma|^2 d x_{n}=\mu_2,
		$$
		such that $\zeta\left(x_{n}\right)$ and $\varsigma\left(x_{n}\right)$ to be chosen later. Notice that $(u, v) \in S\left(\mu_1,\mu_2\right)$.
		Moreover, from Lemma \ref{l0}, we have $m_0=\sqrt{\kappa}l_0$, then
		$$
		\begin{aligned}
			I(u, v)= & \frac{1}{2} \int_{\rn}|\nabla u|^2+\kappa|\nabla v|^2+V_2(x)\left(|u|^2+|v|^2\right) d x  -\frac{1}{2}K(u,v) \\
			&=\frac{1}{2} \int_{\rn}\left(-\Delta_{x^{\prime}} u +V_2(x) u\right)\overline{u}\,d x+\frac{1}{2} \int_{\rn}\left(-\kappa\Delta_{x^{\prime}} v +V_2(x) v\right)\overline{v}\,d x\\
			&\quad + \frac{1}{2} \int_{\mathbb{R}}\left|\partial_{x_n} \zeta\right|^2 d x_{n}+\frac{1}{2} \int_{\mathbb{R}}\left|\partial_{x_n} \varsigma \right|^2 d x_{n}-\frac{1}{2}\Re\left(\int_{\R^{n-1}}(\Psi_0)^{2}\Phi_{0} d x^{\prime}\right)\left(\int_{\mathbb{R}}\zeta^{2}\varsigma d x_{n}\right)\\
			& =\frac{1}{2} \int_{\mathbb{R}}\left|\partial_{x_n} \zeta\right|^2 d x_{n}+\frac{1}{2} \int_{\mathbb{R}}\left|\partial_{x_n} \varsigma \right|^2 \,d x_{n}+\frac{l_0}{2} \int_{\mathbb{R}}|\zeta|^2 \,d x_{n}+\frac{\sqrt{\kappa}l_0}{2} \int_{\mathbb{R}}|\varsigma|^2 \,d x_{n} \\
			& \quad-\frac{1}{2}\Re\left(\int_{\R^{n-1}}(\Psi_0)^{2}\Phi_{0}\,d x^{\prime}\right)\left(\int_{\mathbb{R}}\zeta^{2}\varsigma \,d x_{n}\right).
		\end{aligned}
		$$
		It is sufficient to choose $(\zeta,\varsigma)$ such that
		\begin{equation}\label{firtscon}
			\frac{1}{2} \int_{\mathbb{R}}\left|\partial_{x_n} \zeta\right|^2 d x_{n}+\frac{1}{2} \int_{\mathbb{R}}\left|\partial_{x_n} \varsigma \right|^2 \,d x_{n}-\frac{1}{2}\Re\left(\int_{\R^{n-1}}(\Psi_0)^{2}\Phi_{0}\,d x^{\prime}\right)\left(\int_{\mathbb{R}}\zeta^{2}\varsigma \,d x_{n}\right)<0    
		\end{equation}
		and
		\begin{equation}\label{secondcon}
			\left\|(u,v)\right\|_{\dot{H}}^2 \leq \frac{\chi}{\epsilon_{0}}.   
		\end{equation}
		In fact, we define  $$\zeta\left(x_n\right)=\sqrt{\gamma} \psi\left(\gamma x_n\right), \quad\text{where} \int_{\mathbb{R}}\left|\psi\left(x_{n}\right)\right|^2 d x_n=\mu_1$$
		and $$\varsigma\left(x_{n}\right)=\sqrt{\gamma} \phi\left(\gamma x_{n}\right), \quad\text{where} \int_{\mathbb{R}}\left|\phi\left(x_{n}\right)\right|^2 d x_n=\mu_2.$$
We observe that there exists a constant 
   $\gamma_0>0$ such that $
		(\zeta,\varsigma)$ satisfies all the conditions above for any $\gamma<\gamma_0$. Then,  
		$$
		\begin{aligned}
			& \frac{1}{2} \int_{\mathbb{R}}\left|\partial_{x_n} \zeta\right|^2 d x_{n}+\frac{1}{2} \int_{\mathbb{R}}\left|\partial_{x_n} \varsigma \right|^2 \,d x_{n}-\frac{1}{2}\Re\left(\int_{\R^{n-1}}(\Psi_0)^{2}\Phi_{0}\,d x^{\prime}\right)\left(\int_{\mathbb{R}}\zeta^{2}\varsigma \,d x_{n}\right) \\
			= & \frac{\gamma^2}{2} \int_{\mathbb{R}}\left|\partial_{x_n} \psi\right|^2 d x_{n}+\frac{\gamma^2}{2} \int_{\mathbb{R}}\left|\partial_{x_n} \phi\right|^2 d x_n-\frac{\sqrt{\gamma}}{2} \Re\left(\int_{\R^{n-1}}(\Psi_0)^{2}\Phi_{0}\,d x^{\prime}\right)\left( \int_{\mathbb{R}}\psi^{2} \phi\,d x_{n}\right).
		\end{aligned}
		$$
		And   $\eqref{firtscon}$ follows for $\gamma\ll1$.
		Note that 
		$$
		\begin{aligned}
			\left\|(u,v)\right\|_{\dot{H}}^2&= \int_{\mathbb{R}}\left|\partial_{x_n} \zeta\right|^2 d x_{n}+ \int_{\mathbb{R}}\left|\partial_{x_n} \varsigma \right|^2 \,d x_{n}+l_0 \int_{\mathbb{R}}|\zeta|^2 \,d x_{n}+\sqrt{\kappa}l_0 \int_{\mathbb{R}}|\varsigma|^2 \,d x_{n}\\
			&=\gamma^2 \int_{\mathbb{R}}\left|\partial_{x_n} \psi\right|^2 d x_{n}+\gamma^2 \int_{\mathbb{R}}\left|\partial_{x_n} \phi\right|^2 d x_n+l_0 \mu_1+\sqrt{\kappa}l_0 \mu_2,   
		\end{aligned}$$
		hence  by choosing $\gamma$  small enough,  $$\gamma^2 \left(\int_{\mathbb{R}}\left|\partial_{x_n} \psi\right|^2 d x_{n}+ \int_{\mathbb{R}}\left|\partial_{x_n} \phi\right|^2 d x_n\right)\leq l_{0}\max\{
		1,\sqrt{\kappa}\}(\mu_{1}+\mu_{2})$$ and $\mu^{\star}=\mu^{\star}(\chi,\kappa)$ such that $2 l_0 \max\{
		1,\sqrt{\kappa}\} \mu^{\star}<\frac{\chi}{\epsilon_{0}}$. This shows \eqref{secondcon}. And the proof is complete.
	\end{proof}

	 \blue{Now we proceed to establish a lemma that supports the existence} of solutions for the system \eqref{elliptic}.
	\begin{lemma}\label{novani}
		Let $\chi, \mu_1, \mu_2>0$, and $\mu^{\star}=c^{\star}(\chi)>0$ be as in Lemma \ref{propd}. If $\mu_1+\mu_2 \leq \mu^{\star}$, and $\left\{\left(u_m, v_m\right)\right\}$ is a sequence such that
		$$
		\begin{aligned}
			\left(u_m, v_m\right) \in S(\mu_{1},\mu_{2}) \cap B(\chi),\, \,  I\left(u_m, v_m\right)  \rightarrow d_{\mu_1, \mu_2}^\chi, \, \, \text{as}\,\, m \rightarrow\infty,
		\end{aligned}
		$$
		then, there is a $\delta>0$ such that
		$$
		\liminf _{m \rightarrow \infty} \int_{\rn}\left|u_m\right|^3 d x \geq \delta, \quad \liminf _{m \rightarrow \infty} \int_{\rn}\left|v_m\right|^3 d x \geq \delta.
		$$
	\end{lemma}
	\begin{proof}
		By contradiction, we   suppose that $\int_{\rn}\left|u_m\right|^3 d x \rightarrow 0$ as $m \rightarrow \infty$. Since $\left(u_{m}, v_{m}\right) \in$ $S(\mu_{1},\mu_{2}) \cap B(\chi)$, we know, from \eqref{bound}, that $\left\{\left(u_{m}, v_{m}\right)\right\}$ is bounded in $H$, then  it follows from the Hölder  inequality that $$\int_{\rn}(u_{m})^2  v_{m}\,d x  \rightarrow 0  \quad \text{as} \, \, m  \rightarrow \infty.$$ Thus,
		$$
		\begin{aligned}
			\lim_{m \rightarrow \infty}I\left(u_{m}, v_{m}\right)= & \frac{1}{2}\lim_{m \rightarrow \infty}  \int_{\rn}\left|\nabla u_{m}\right|^2+\left|\nabla v_{m}\right|^2+V_{2}(x)\left(\left|u_{m}\right|^2+\left|v_{m}\right|^2\right) d x,
		\end{aligned}
		$$
		that is,
		$$
		d_{\mu_1, \mu_2}^\chi = \frac{l_0 \mu_1}{2}+ \frac{\sqrt{\kappa}l_0 \mu_2}{2},
		$$
		which is a contradiction with Lemma \ref{propd}. Thus there is a $\delta>0$ such that
		$$
		\liminf _{m \rightarrow \infty} \int_{\rn}\left|u_{m}\right|^3 d x \geq \delta
		$$
		Using the same method as above, we can prove
		$$
		\liminf _{m \rightarrow \infty} \int_{\rn}\left|v_{m}\right|^3 d x \geq \delta
		$$
	\end{proof}
	
	Once it was shown that $ S(\mu_1,\mu_2)  \cap B(\chi) \neq \emptyset$ (see Lemma \ref{nosetempty}), $d_{\mu_1, \mu_2}^\chi$ is achieved, i.e. $\mathcal{D}_{\mu_{1}, \mu_{2}}^\chi\neq\emptyset$
	and the minimizer of \eqref{min3} is a critical point of $\left.I\right|_{S(\mu_1,\mu_2) }$ as well as a normalized solution to \eqref{elliptic}. \blue{We continue with the demonstration of   Theorem \ref{theorem1}}.

	\begin{proof}[Proof of Theorem \ref{theorem1}]
		First, we will prove that $\mathcal{D}_{\mu_{1}, \mu_{2}}^\chi \neq \emptyset$. In fact, for any $\chi>0$, suppose that $\left\{\left(u_m, v_m\right)\right\} \subset S(\mu_1,\mu_2) \cap B(\chi)$ is a minimizing sequence for $d_{\mu_1, \mu_2}^\chi$, this is, $$I\left(u_m, v_m\right) \rightarrow d_{\mu_1, \mu_2}^\chi \, \,  \text{as}\, \, m  \rightarrow \infty,$$ then $\left\|(u_{m},v_{m})\right\|_{H}^2 \leq \chi+\mu_{1}+\mu_{2}$. So $\left\{\left(u_m, v_m\right)\right\}$ is uniformly bounded in $H$. The idea is to show that $\left\{\left(u_m, v_m\right)\right\}$, up to translation, is compact in $H$.		
		By Lemmas \ref{solutionweak} and \ref{novani} , there exist a sequence $\left\{y_1^{m}\right\} \subset \mathbb{R}$ and $ (u,v) \in H \backslash\{(0,0)\}$  such that
		$$
		u_{m}\left(x^{\prime}, x_{n}-y_1^{m}\right) \rightharpoonup u, v_{m}\left(x^{\prime}, x_{n}-y_1^{m}\right) \rightharpoonup v \text { in } H
		$$
		as $m \rightarrow \infty$. Now, without loss of generality, assume that $u \neq 0$ and notice that $$\|u\|_{\lt}^{2}\leq \liminf_{m \rightarrow \infty}\|u_{m}\|_{\lt}^{2}=\mu_{1}$$
		and $$\|v\|_{\lt}^{2}\leq \liminf_{m \rightarrow \infty}\|v_{m}\|_{\lt}^{2}=\mu_{2},$$
		hence $0<a_1:=\|u\|_{\lt}^{2} \leq \mu_{1}, 0 \leq a_2:=\|v\|_{\lt}^{2} \leq \mu_{2}.$  Let us  see that $a_1=\mu_1$.
		
		Set
		$$
		\begin{gathered}
			\widetilde{u}_n(x)=u_{m}\left(x^{\prime}, x_{n}-y_1^{m}\right)-u(x), \widetilde{v}_n(x)=v_{m}\left(x^{\prime}, x_{n}-y_1^{m}\right)-v(x).
		\end{gathered}
		$$
		By contradiction, we just suppose $a_1<\mu_{1}$. We consider two cases $a_2<\mu_{2}$ and $a_2=\mu_{2}$ to obtain a contradiction.
		
		If $a_2<\mu_{2}$, by the Brezis-Lieb Lemma and Lemma \ref{igualimit}, we have
		\begin{equation}\label{seq1}
			\begin{aligned}
				I\left(u_m, v_m\right) & =I(u, v)+I\left(\widetilde{u}_m, \widetilde{v}_m\right)+o_m(1) \\
				& \geq I(u, v)+d_{b_1^m, b_2^m}^\chi+o_m(1),
			\end{aligned}    
		\end{equation}
		where
		$$
		\left(b_1^m, b_2^m\right)=\left(\left\|\widetilde{u}_m\right\|_{\lt}^2,\left\|\widetilde{v}_m\right\|_{\lt}^2\right).
		$$
		Additionally, by applying the Brezis-Lieb Lemma once more, we obtain
		\begin{equation}\label{seq2}
			\left(b_1^m, b_2^m\right)+\left(a_1, a_2\right)=\left(\mu_{1}+o_m(1), \mu_{2}+o_m(1)\right).   
		\end{equation}
		We can assume that $\left(b_1^m, b_2^m\right) \rightarrow\left(b_1, b_2\right)$ and hence $\left(b_1, b_2\right)+\left(a_1, a_2\right)=\left(\mu_{1}, \mu_{2}\right)$ for sufficiently large $m$. Therefore, by using  Lemma \ref{propd},  \eqref{seq1} and \eqref{seq2}, it follows that 
		$$
		d_{\mu_1, \mu_2}^\chi \geq d_{a_1, a_2}^\chi+d_{b_1, b_2}^\chi .
		$$
		But,  we know from   Lemma \ref{propd} that
		$$
		d_{\mu_1, \mu_2}^\chi < d_{a_1, a_2}^\chi+d_{b_1, b_2}^\chi,
		$$
		so that
		\begin{equation}\label{iguality1}
			d_{\mu_1, \mu_2}^\chi=d_{a_1, a_2}^\chi+d_{b_1, b_2}^\chi.   
		\end{equation}
		Since
		\begin{equation}\label{conv0}
			I\left(u_m, v_m\right)=d_{\mu_1, \mu_2}^\chi+o_m(1),
		\end{equation}
		we then obtain from \eqref{seq1}  that
		$$d_{a_1, a_2}^\chi+d_{b_1, b_2}^\chi +o_m(1)\geq I(u,v)+d_{b_1^m, b_2^m}^\chi+o_m(1).$$
		Thus, for $m\gg1$, by definition of  $d_{a_1, a_2}^\chi$,  we have
		\begin{equation}\label{iguality3}
			I(u, v)=d_{a_1, a_2}^\chi.    
		\end{equation}
		Note that we can assume that $u, v \geq 0$, since if $(u, v)$ satisfies \eqref{iguality3}, then $(|u|, |v|)$ also satisfies it. In fact, since $|\nabla |u|| \leq |\nabla u|$ and $|\nabla |v|| \leq |\nabla v|$, it is clear that
		$$d_{a_1,a_2}^{\chi}\leq I(|u|,|v|)\leq I(u,v)=d_{a_1,a_2}^{\chi}.$$
		Equality		\eqref{iguality3} implies that $(u,v)$ is solution  to   \eqref{elliptic}, then by the maximum principle we can suppose that it is positive. 
		
		On the other hand, as $a_i<c_i(i=1,2)$, we have $b_i>0$. Then,  suppose that $\left\{\left(\overline{u}_m, \overline{v}_m\right)\right\}$ is a minimizing sequence to $d_{b_1, b_2}^\chi$. Similar to above, we can obtain that there exists a sequence $\left\{z_1^m \right\} \subset \mathbb{R}$ and  $(\overline{u},\overline{v}) \in H \backslash \{(0,0)\}$  such that
		$$
		\overline{u}_m\left(x^{\prime}, x_{n}-z_1^m\right) \rightharpoonup \overline{u}, \quad \overline{v}_m\left(x^{\prime}, x_{n}-z_1^m\right) \rightharpoonup \overline{v}\quad \text { in } H
		$$
		as $m \rightarrow \infty$. Moreover, define
		$$
		d_1:=\|\overline{u}\|_{\lt}^2>0, \quad d_2:=\|\overline{v}\|_{\lt}^2 \geq 0.
		$$
		So, proceeding as above, we have
		\begin{equation}\label{iguality2}
			d_{b_1, b_2}^\chi=d_{d_1, d_2}^\chi+d_{b_1-d_1, b_2-d_2}^\chi    
		\end{equation}
		and 
		\begin{equation}\label{iguality4}
			I(\overline{u}, \overline{v})=d_{d_1, d_2}^\chi, \, \, \,\text{and} \, \,\, \,  \overline{v}>0.   
		\end{equation}
		It follows from \eqref{iguality1} and \eqref{iguality2} that
		$$
		d_{\mu_1, \mu_2}^\chi=d_{a_1, a_2}^\chi+d_{d_1, d_2}^\chi+d_{b_1-d_1, b_2-d_2}^\chi .
		$$
		Now, let $u^{\star}$ be the Steiner rearrangement of $u$, then using \eqref{sigma2},
		$$
		\int_{\rn}V_{2}(x)\left|u^{\star }\right|^2 \,d x=\int_{\rn}V_{2}(x)\left|u\right|^2 \,d x,
		$$
		and from Theorem $3.4$ in \cite{lieb},
		$$\int_{\rn}u^{2}v\,d x \leq  \int_{\rn}(u^{\star})^{2}v^{\star}\,d x= \int_{\rn}(u^2)^{\star}v^{\star}\,d x.$$
		Consequently,
		\begin{equation*}
			d_{a_1, a_2}^\chi \leq I\left(u^{\star }, v^{\star }\right)\leq I\left(u, v\right) \text { and } d_{d_1, d_2}^\chi \leq I\left(\overline{u}^{\star }, \overline{v}^{\star }\right)\leq I(\overline{u},\overline{v}).    
		\end{equation*}
		Therefore,  it follows from \eqref{iguality3} and \eqref{iguality4} that 
		\begin{equation}\label{equality5}
			I\left(u^{\star }, v^{\star }\right) =d_{a_1, a_2}^\chi  \text { and } I\left(\overline{u}^{\star }, \overline{v}^{\star }\right)=d_{d_1, d_2}^\chi.    
		\end{equation}
		This implies that $\left(u^{\star }, v^{\star }\right)$  and $\left(\overline{u}^{\star }, \overline{v}^{\star }\right)$  are solutions of system \eqref{elliptic}, then by elliptic regularity theory, they are functions of class $C^{2}(\rn)$ (see Proposition \eqref{regularity}) and by the maximum principle  $u^{\star}, v^{\star}, \overline{u}^{\star}, \overline{v}^{\star} >0$. 
		
		Now, using   the items  $(ii)$ and $(iii)$ in  Lemma \ref{schwartz}, we obtain
		$$
		\int_{\rn}\left|\{u^{\star},\overline{u}^{\star }\}^{\star}\right|^2 d x  
		=\int_{\rn}\left(\left|u^{\star }\right|^2+\left|\overline{u}^{\star }\right|^2\right) d x= a_1+d_1, 
		$$
		$$\int_{\rn}\left|\{v^{\star},\overline{v}^{\star }\}^{\star}\right|^2 d x  
		=\int_{\rn}\left(\left|v^{\star }\right|^2+\left|\overline{v}^{\star }\right|^2\right) d x= a_2+d_2,$$
		$$
		\begin{aligned}
			\int_{\rn}V_{2}(x)\left|\{u^{\star},\overline{u}^{\star }\}^{\star}\right|^2 d x = \int_{\rn}V_{2}(x)\left(\left|u^{\star }\right|^2+\left|\overline{u}^{\star }\right|^2\right) d x,
		\end{aligned}
		$$
		$$
		\begin{aligned}
			\int_{\rn}V_{2}(x)\left|\{v^{\star} ,\overline{v}^{\star }\}^{\star}\right|^2 d x = \int_{\rn}V_{2}(x)\left(\left|v^{\star }\right|^2+\left|\overline{v}^{\star }\right|^2\right) d x,
		\end{aligned}
		$$
		$$\int_{\rn}((u^{\star})^{2}v^{\star}+(\overline{u}^{\star})^2 \overline{v}^{\star})\,d x \leq  \int_{\rn}\{(u^{\star})^{2},(\overline{u}^{\star})^2\}^{\star}\{v^{\star},\overline{v}^{\star}\}^{\star}\,d x=\int_{\rn}(\{u^{\star},\overline{u}^{\star}\}^{\star})^2 \{v^{\star},\overline{v}^{\star}\}^{\star}\,d x$$
		and from   $(v)$  of Lemma \ref{schwartz},
		$$\int_{\rn}\left|\nabla \{v^{\star},\overline{v}^{\star }\}^{\star}\right|^2 d x < \int_{\rn}\left|\nabla v^{\star}\right|^2 d x+ \int_{\rn}\left|\nabla \overline{v}^{\star}\right|^2 d x, $$
		$$\int_{\rn}\left|\nabla \{u^{\star},\overline{u}^{\star }\}^{\star}\right|^2 d x < \int_{\rn}\left|\nabla u^{\star}\right|^2 d x+ \int_{\rn}\left|\nabla \overline{u}^{\star}\right|^2 d x .$$
		Hence, we deduce from \eqref{iguality3} and \eqref{equality5}, and the  above estimates, that
		$$
		\begin{aligned}
			d_{a_1+d_1, a_2+d_2}^\chi & \leq I\left(\{u^{\star}, \overline{u}^{\star }\}^{\star}, \{v^{\star}, \overline{v}^{\star }\}^{\star}\right)<I\left(u^{\star }, v^{\star }\right)+I\left(\overline{u}^{\star }, \overline{v}^{\star }\right) \\
			& =d_{a_1, a_2}^\chi+d_{d_1, d_2}^\chi \\
			& =d_{\mu_1, \mu_2}^\chi-d_{b_1-d_1, b_2-d_2}^\chi,
		\end{aligned}
		$$
		which contradicts $(ii)$ in  Lemma \ref{propd}.
		Now, if $a_2=\mu_{2}$. In this case, we can similarly obtain a contradiction by the same arguments as $a_2<\mu_{2}$. As a consequence, we have proved that $\mu_{1}=a_1$.
		Then, we obtain $\|u\|_{\lt}^{2}=\mu_1$, and 
		\begin{equation}\label{conv1}
			u_{m}\left(x^{\prime}, x_{n}-y_1^{m}\right) \rightarrow u(x) \text { in } L^2\left(\rn\right) \text { as } m \rightarrow \infty .   
		\end{equation}
		The Gagliardo-Nirenberg inequality    shows that
		\begin{equation}\label{conv2}
			u_{m}\left(x^{\prime}, x_{n}-y_1^{m}\right) \rightarrow u(x) \text { in } L^3\left(\rn\right) \text { as } m \rightarrow \infty .   
		\end{equation}
		We now change the role of $\left\{u_{m}\right\}$ and $\left\{v_{m}\right\}$, then there exist a sequence $\left\{y_2^{m}\right\} \subset \mathbb{R}$, $(\tilde{u},\tilde{v}) \in H$, such that  $\|\tilde{v}\|_{\lt}^{2}=\mu_{2}$ such that
		$$
		u_{m}\left(x^{\prime}, x_{n}-y_2^{m}\right) \rightharpoonup \tilde{u}(x), v_{m}\left(x^{\prime}, x_{n}-y_2^{m}\right) \rightharpoonup \tilde{v}(x) \text { in } H
		$$
		as $m \rightarrow \infty$. Moreover,
		\begin{equation}\label{conv3}
			v_{m}\left(x^{\prime}, x_{n}-y_2^{m}\right) \rightarrow \tilde{v}(x) \text { in } L^2\left(\rn\right) \text { as } m \rightarrow \infty,   
		\end{equation}
		and 
		  we get  from the Gagliardo-Nirenberg inequality that 
		\begin{equation}\label{conv4}
			v_{m}\left(x^{\prime}, x_{n}-y_2^{m}\right) \rightarrow \tilde{v}(x) \text { in } L^3\left(\rn\right) \text { as } m \rightarrow \infty.
		\end{equation}
		\blue{Besides}, notice that 
		\begin{equation}\label{trans}
			\lim _{m \rightarrow \infty}\left|y_1^{m}-y_2^{m}\right|<\infty .    
		\end{equation}
		If this does not hold, we get that $\tilde{u}=v=0$ due to $(u, \tilde{v}) \in S(\mu_1,\mu_{2})$. Indeed, if $v\neq 0$, then   $$v_{m}\left(x^{\prime}, x_{n}-y_1^{m}\right) \rightharpoonup v(x) \text { in } \lt,$$
		shows that $$v_{m}\left(x^{\prime}, x_{n}-y_2^{m}\right) \rightharpoonup v\left(x^{\prime}, x_{n}+y_{1}^{m}-y_2^{m}\right) \text { in } \lt,$$
		From \eqref{conv3}, we have $\tilde{v}(x)=v\left(x^{\prime}, x_{n}+y_{1}^{m}-y_2^{m}\right)$, so   it follows from $\left|y_1^{m}-y_2^{m}\right|  \rightarrow \infty$, as $m  \rightarrow \infty$, that $$\mu_{2}=\int_{\rn}|\tilde{v}(x)|^{2}\,d x=\int_{\rn}|v\left(x^{\prime}, x_{n}+y_{1}^{m}-y_2^{m}\right)|^{2}\,d x=0,$$
		which is a contradiction. 
		
		Then, we get from the Brezis-Lieb Lemma and Lemma \ref{igualimit} that
		$$
		\begin{aligned}
			I\left(u_m, v_m\right)&=  I\left(u_{m}\left(x^{\prime}, x_{n}-y_1^{m}\right), v_{m}\left(x^{\prime}, x_{n}-y_1^{m}\right)\right) \\
			&=  I\left(u_{m}\left(x^{\prime}, x_{n}-y_1^{m}\right)-u, v_{m}\left(x^{\prime}, x_{n}-y_1^{m}\right)\right)+I(u, 0)+o_m(1) \\
			&=  I\left(u_{m}\left(x^{\prime}, x_{n}-y_2^{m}\right)-u\left(x^{\prime}, x_{n}+y_1^{m}-y_2^{m}\right), v_{m}\left(x^{\prime}, x_{n}-y_2^{m}\right)\right) +I(u, 0)+o_m(1) \\
			&=  I(0, \tilde{v})+I\left(u_{m}\left(x^{\prime}, x_{n}-y_2^{m}\right)-u\left(x^{\prime}, x_{n}+y_1^{m}-y_2^{m}\right), v_{m}\left(x^{\prime}, x_{n}-y_2^{m}\right)-\tilde{v}\right) \\
			&\qquad +I(u, 0)+o_m(1)\\
			&\geq I(0, \tilde{v})+I(u, 0) +o_{m}(1)\\&\qquad -\int_{\rn}\left(u_{m}\left(x^{\prime}, x_{n}-y_2^{m}\right)-u\left(x^{\prime}, x_{n}+y_1^{m}-y_2^{m}\right)\right)^{2}\left(v_{m}\left(x^{\prime}, x_{n}-y_2^{m}\right)-\tilde{v}\right).
		\end{aligned}
		$$
		Hence, \eqref{conv2} and \eqref{conv4} show that
		$$
		d_{\mu_1, \mu_2}^\chi \geq I(u, 0)+I(0, \tilde{v}) .
		$$
		So,  from Lemma \ref{steiner}, we obtain
		$$
		d_{\mu_1, \mu_2}^\chi \geq I(u, 0)+I(0, \tilde{v})  \geq I\left(u^{\star }, 0\right)+I\left(0, \tilde{v}^*\right) \geq d_{\mu_{1}, 0}^\chi+d_{0, \mu_{2}}^\chi .
		$$
		Hence, we infer from (ii) in Lemma \ref{propd}  that
		\begin{equation}\label{contrad}
			d_{\mu_1, \mu_2}^\chi=I\left(u^{\star }, 0\right)+I\left(0, \tilde{v}^*\right) .
		\end{equation}
		Since  $\left(u^{\star }, \tilde{v}^*\right) \in S(\mu_1,\mu_2) \cap B(\chi)$, then 
		$$
		d_{\mu_1, \mu_2}^\chi \leq I\left(u^{\star }, \tilde{v}^*\right)<I\left(u^{\star }, 0\right)+I\left(0, \tilde{v}^*\right),
		$$
		This is a contradiction with \eqref{contrad}. So \eqref{trans} holds.
		
		Now, we can find a $y \in \mathbb{R}$ such that $y_1^{m}=y_2^{m}+y+o_m(1)$. Define
		$$
		(\varphi_m(x),\psi_m(x)):=(u_{m}\left(x^{\prime}, x_{n}-y_1^{m}\right), v_{m}\left(x^{\prime}, x_{n}-y_1^{m}\right)) \in S\left(\mu_{1},\mu_{2}\right) .
		$$
		Accordingly there is a $(\varphi(x), \psi(x)):=\left(u(x), \tilde{v}\left(x^{\prime}, x_{n}-y\right)\right) \in S\left(\mu_{1}\right) \times S\left(\mu_{2}\right)$ such that
		$$
		\begin{gathered}
			\left(\varphi_n, \psi_n\right) \rightharpoonup (\varphi, \psi) \text { in } H, \\
			\left(\varphi_n, \psi_n\right) \rightarrow(\varphi, \psi) \text { in } L^3\left(\rn\right) \times L^3\left(\rn\right).
		\end{gathered}
		$$
		We conclude from the Brezis-Lieb Lemma and Lemma \ref{igualimit} that
		$$
		\begin{aligned}
			I\left(u_m, v_m\right)=I\left(\varphi_n, \psi_n\right) & =I(\varphi, \psi)+I\left(\varphi_n-\varphi, \psi_n-\psi\right)+o_m(1) \\
			& \geq d_{\mu_1, \mu_2}^\chi+\frac{\min\{1,\kappa\}}{2}\left\|(\varphi_m,\psi_m)-(\varphi,\psi)\right\|_{\dot{H}}^2\\
			&\quad-\int_{\rn}(\varphi_{m}-\varphi)^{2}(\psi_{m}-\psi)\,d x+o_m(1).
		\end{aligned}
		$$
		From \eqref{conv0}, we obtain that $\left(\varphi_n, \psi_n\right) \rightarrow(\varphi, \psi)$ in $H$ as $m\rightarrow \infty$. Hence, any minimizing sequence for $d_{\mu_1, \mu_2}^\chi$, up to translation, is compact and $\mathcal{D}_{\mu_{1}, \mu_{2}}^\chi \neq \emptyset$ follows.  
		
		Therefore, there exists $(u_{\mu_1},v_{\mu_2})\in \mathcal{D}_{\mu_1,\mu_2}^{\chi}$  and $\left(\lambda_{\mu_{1}}, \lambda_{\mu_{2}}\right) \in \mathbb{R}^2$  such that $\left(u_{\mu_{1}}, v_{\mu_{2}}, \lambda_{\mu_{1}}, \lambda_{\mu_{2}}\right)$ is a couple of weak solution to problem \eqref{elliptic}.
	\end{proof}
	
	Now we are in a position to  prove Theorem \ref{2theorem}.
	
	\begin{proof}[Proof of Theorem \ref{2theorem}]
		Let $(u,v) \in H^1(\rn : \mathbb{C})\times H^1(\rn: \mathbb{C})$ be a complex valued minimizer. A standard elliptic regularity bootstrap (see Proposition \ref{regularity}) shows that $u,v$ are of class $C^1(\rn : \mathbb{C})$. We know that $(|u|,|v|)$ satisfies the following $|u|,|v| \in C^1(\rn : \mathbb{C})$  and are   minimizers. Moreover,   by using the strong maximum principle we get $|u|,|v|>0$ and thus $(u,v) \in C^1(\rn : \mathbb{C}\backslash\{0\})\times C^1(\rn ; \mathbb{C} \backslash\{0\})$. Now, since $(u, v)$ and $(|u|, |v|)$ are minimizers, we have $I(u, v) = I(|u|, |v|)$, from which $$\int_{\rn}|\nabla u |^{2}\,d x+ \kappa \int_{\rn}|\nabla v |^{2}\,d x  =\int_{\rn}|\nabla u |^{2}\,d x+ \kappa \int_{\rn}|\nabla v |^{2}\,d x. $$
		We then write $u = \rho_{1}w_1$ and $v = \rho_{2}w_2$, where $\rho_1 = |u|$ and $\rho_2 = |v|$. Since $u, |u| \in C^1(\rn: \mathbb{C} \backslash\{0\})$, it follows that $w_1 \in C^1(\rn:\mathbb{C} \backslash\{0\})$. Since $|w_1| = 1$, we have $\Re(\bar{w}_1\nabla w_1) = 0$. Thus, by
		$$
		\nabla u=(\nabla \rho_1) w_1+\rho_1 \nabla w_1,$$
		we have
		$$|\nabla u|^2=|\nabla \rho_1|^2+\rho_1^2|\nabla w_1|^2.$$ Similarly, we obtain $$|\nabla v|^2=|\nabla \rho_2|^2+\rho_2^2|\nabla w_2|^2.$$ 
		Thus, since $\rho_1=|u|$ and $\rho_2=|v|$, we have that $|\nabla| u||=|\nabla \rho_1|$ and $|\nabla| v||=|\nabla \rho_2|$. Hence, 
		$$
		\begin{aligned}
			\int_{\rn}|\nabla \rho_1|^2 \,d x+\kappa\int_{\rn}|\nabla \rho_2|^2 \,d x&=\int_{\rn}|\nabla u|^2 \,d x+\kappa\int_{\rn}|\nabla v|^2 \,d x\\&=\int_{\rn}\left(|\nabla \rho_1|^2+|\rho_1 \nabla w_1|^2\right) \,d x+\kappa\int_{\rn}\left(|\nabla \rho_2|^2+|\rho_2 \nabla w_2|^2\right) \,d x.   
		\end{aligned}
		$$
		Therefore, $$\int_{\rn}|\rho_1 \nabla w_1|^2\,d x=-\kappa\int_{\rn}|\rho_2 \nabla w_2|^2\,d x\leq 0,$$
		implying    $\int_{\rn}|\rho_1 \nabla w_1|^2 d x=0$ and  $\int_{\rn}|\rho_2 \nabla w_2|^2 d x=0$, which leads to $\nabla w_1=0$ and  $\nabla w_2=0$. Then,   $w_1, w_2$ are constants, from which, using the complex polar form of $u$ and $v$, it follows that there exist $\theta_1$ and $\theta_2$ in $\mathbb{R}$ such that $u = e^{i\theta_1}f_1(x^{\prime},x_n)$ and $v = e^{i\theta_2}f_2(x^{\prime},x_n)$, with $f_1(x^{\prime},x_n)=\rho_1(x^{\prime},x_n)$ and $f_2(x^{\prime},x_n)=\rho_2(x^{\prime},x_n)$.
		
		Now, let $\left(u_{\mu_1}, v_{\mu_2}\right) \in \mathcal{D}_{\mu_1, \mu_2}^\chi$. We know that $\left(u{\mu_1}, v_{\mu_2}\right)$ stays away from the boundary of $S(\mu_1,\mu_2) \cap B(\chi)$, since, by Lemma \ref{ineqinf}, we see that $\left(u_{\mu_1}, v_{\mu_2}\right) \in B\left(\frac{\chi}{2\epsilon_{0}}\right)$. Then $\left(u_{\mu_1}, v_{\mu_2}\right)$ is indeed a critical point of $I(u, v)$ restricted on $S(\mu_1,\mu_2)$. So, there exists $\left(\lambda_{\mu_1}, \lambda_{\mu_2}\right) \in \mathbb{R}^2$ such that $\left(u_{\mu_1}, v_{\mu_2}, \lambda_{\mu_1}, \lambda_{\mu_2}\right)$ is a couple of weak solution to problem \eqref{elliptic}. Namely, $\left(u_{\mu_1}, v_{\mu_2}, \lambda_{\mu_1}, \lambda_{\mu_2}\right)$ satisfies
		$$
		\left\{\begin{array}{l}
			-\Delta u_{\mu_1}+V_2(x) u_{\mu_1}=\lambda_{\mu_1} u_{\mu_1}+u_{\mu_1}v_{\mu_2}, \\
			-\kappa \Delta v_{\mu_2}+V_2(x) v_{\mu_2}=\lambda_{\mu_2} v_{\mu_2}+\frac{1}{2}u_{\mu_1}^{2}.
		\end{array}\right. 
		$$
		Or equivalently, 
		\begin{equation}\label{lam1}
			\lambda_{\mu_1}=\frac{1}{\mu_1}\left(\int_{\rn}\left|\nabla u_{\mu_1}\right|^{2}\,d x+\int_{\rn}V_{2}(x)\left|u_{\mu_1}\right|^{2}\,d x-\frac{1}{2}K (u_{\mu_1},v_{\mu_2})\right)    
		\end{equation}
		and 
		\begin{equation}\label{lam2}
			\lambda_{\mu_2}=\frac{1}{\mu_2}\left(\kappa\int_{\rn}\left|\nabla v_{\mu_2}\right|^{2}\,d x+\int_{\rn}V_{2}(x)\left|v_{\mu_2}\right|^{2}\,d x-\frac{1}{2}K( u_{\mu_1},v_{\mu_2})\right).  
		\end{equation}
		If $\mu_1=\mu_2$, then
		\begin{equation*}
			\begin{aligned}
				\lambda_{\mu_1}+\lambda_{\mu_2} &=\frac{1}{\mu_1}\left(\int_{\rn}\left|\nabla u_{\mu_1}\right|^{2}\,d x+\int_{\rn}V_{2}(x)\left|u_{\mu_1}\right|^{2}\,d x-\frac{1}{2}K (u_{\mu_1},v_{\mu_2})\right)\\
				&\quad+\frac{1}{\mu_2}\left(\kappa\int_{\rn}\left|\nabla v_{\mu_2}\right|^{2}\,d x+\int_{\rn}V_{2}(x)\left|v_{\mu_2}\right|^{2}\,d x-\frac{1}{2}K (u_{\mu_1},v_{\mu_2})\right)\\
				&\leq \frac{1}{\mu_1}\Big(\int_{\rn}\left|\nabla u_{\mu_1}\right|^{2}\,d x+\kappa\int_{\rn}\left|\nabla v_{\mu_2}\right|^{2}\,d x+\int_{\rn}V_{2}(x)\left(|u_{\mu_1}|^{2}+|v_{\mu_2}|^{2}\right) dx\\&\quad \quad \quad-K (u_{\mu_1},v_{\mu_2})\Big)\\
				&<\frac{2 I\left(u_{\mu_1}, v_{\mu_2}\right)}{\mu_1}.
			\end{aligned}   
		\end{equation*}
		By Lemma \ref{propd}, we have
		\begin{equation*}
			I\left(u_{\mu_1}, v_{\mu_2}\right)<\frac{l_0\left(\mu_1+\sqrt{\kappa}\mu_2\right)}{2}=\frac{l_0\left(1+\sqrt{\kappa}\right)}{2}\mu_1,   
		\end{equation*}
		and hence
		\begin{equation}\label{lambda1}
			\lambda_{\mu_1}+\lambda_{\mu_2}<l_{0}(1+\sqrt{\kappa})\leq 2l_{0}\max\{1,\sqrt{\kappa}\}.   
		\end{equation}
		Moreover, 
		\begin{equation*}
			\begin{aligned}
				\lambda_{\mu_1}+\lambda_{\mu_2}&=\frac{1}{\mu_1}\Big(\int_{\rn}\left|\nabla u_{\mu_1}\right|^{2}\,d x+\kappa\int_{\rn}\left|\nabla v_{\mu_2}\right|^{2}\,d x+\int_{\rn}V_{2}(x)\left(|u_{\mu_1}|^{2}+|v_{\mu_2}|^{2}\right)\,d x\\&\quad \quad \quad-K (u_{\mu_1},v_{\mu_2}) \Big).    
			\end{aligned}   
		\end{equation*}
		So that
		\begin{equation*}
			\begin{aligned}
				&\lambda_{\mu_1}+\lambda_{\mu_2}\\ &\geq \frac{1}{\mu_1}\Big(\int_{\rn}\left|\nabla u_{\mu_1}\right|^{2}\,d x+\kappa\int_{\rn}\left|\nabla v_{\mu_2}\right|^{2}\,d x+\int_{\rn}V_{2}(x)\left(|u_{\mu_1}|^{2}+|v_{\mu_2}|^{2}\right)\,d x\\&\quad \quad \quad-C_{\circ}\mu_{1}^{(6-n)/6}\mu_{2}^{(6-n)/12}\left(\|\nabla u_{\mu_1}\|^{2}\right)^{n/6}\left(\|\nabla v_{\mu_2}\|^{2}\right)^{n/12} \Big)\\
				&=\frac{1}{\mu_1}\left(\int_{\rn}\left(\left|\nabla u_{\mu_1}\right|^{2}+\kappa\left|\nabla v_{\mu_2}\right|^{2}+V_{2}(x)\left(|u_{\mu_1}|^{2}+|v_{\mu_2}|^{2}\right)\right)\,d x\right)\\&
				\quad \quad \quad \times \left(1-\frac{1}{\kappa^{n/12}}C_{\circ}\mu_{1}^{(6-n)/6}\mu_{2}^{(6-n)/12}\chi^{\frac{n-4}{4}}\right).
			\end{aligned}    
		\end{equation*}
		So, from Lemma \ref{l0}, we have that if $(\mu_1,\mu_{2}) \rightarrow 0$, 
		\begin{equation}\label{lambda2}
			\lambda_{\mu_1}+\lambda_{\mu_2} \geq   2\min\{1,\sqrt{\kappa}\}l_{0}.
		\end{equation}
		If $\mu_2<\mu_1$, then
		\begin{equation*}
			\begin{aligned}
				\lambda_{\mu_1}+\lambda_{\mu_2} &=\frac{1}{\mu_1}\left(\int_{\rn}\left|\nabla u_{\mu_1}\right|^{2}\,d x+\int_{\rn}V_{2}(x)\left|u_{\mu_1}\right|^{2}\,d x-\frac{1}{2}K (u_{\mu_1},v_{\mu_2})\right)\\
				&\quad+\frac{1}{\mu_2}\left(\kappa\int_{\rn}\left|\nabla v_{\mu_2}\right|^{2}\,d x+\int_{\rn}V_{2}(x)\left|v_{\mu_2}\right|^{2}\,d x-\frac{1}{2}K (u_{\mu_1},v_{\mu_2})\right)\\
				&< \frac{1}{\mu_2}\Big(\int_{\rn}\left|\nabla u_{\mu_1}\right|^{2}\,d x+\kappa\int_{\rn}\left|\nabla v_{\mu_2}\right|^{2}\,d x+\int_{\rn}V_{2}(x)\left(|u_{\mu_1}|^{2}+|v_{\mu_2}|^{2}\right)\,d x\\&\quad \quad \quad-K (u_{\mu_1},v_{\mu_2})\Big)\\
				&<\frac{2 I\left(u_{\mu_1}, v_{\mu_2}\right)}{\mu_2}.
			\end{aligned}   
		\end{equation*}
		By Lemma \ref{propd}, we observe that
		\begin{equation}\label{lambda3}
			\lambda_{\mu_1}+\lambda_{\mu_2}<\frac{2 I\left(u_{\mu_1}, v_{\mu_2}\right)}{\mu_2}\leq \frac{l_0\left(\mu_1+\sqrt{\kappa}\mu_2\right)}{\mu_{2}}\leq l_{0}\max\{1,\sqrt{\kappa}\}\left(1+\frac{\mu_1}{\mu_2}\right).
		\end{equation}
		Now, since 
		$$
		\begin{aligned}
			\lambda_{\mu_1}+\lambda_{\mu_2} &=\frac{1}{\mu_1}\left(\int_{\rn}\left|\nabla u_{\mu_1}\right|^{2}\,d x+\int_{\rn}V_{2}(x)\left|u_{\mu_1}\right|^{2}\,d x-\frac{1}{2}K (u_{\mu_1},v_{\mu_2})\right)\\
			&\quad+\frac{1}{\mu_2}\left(\kappa\int_{\rn}\left|\nabla v_{\mu_2}\right|^{2}\,d x+\int_{\rn}V_{2}(x)\left|v_{\mu_2}\right|^{2}\,d x-\frac{1}{2}K (u_{\mu_1},v_{\mu_2})\right)\\
			&\geq \frac{1}{\mu_1}\left(\int_{\rn}\left(\left|\nabla u_{\mu_1}\right|^{2}+\kappa\left|\nabla v_{\mu_2}\right|^{2}+V_{2}(x)\left(|u_{\mu_1}|^{2}+|v_{\mu_2}|^{2}\right)\right)\,d x\right)\\
			&\quad \quad \quad \times \left(1-\frac{1}{\kappa^{n/12}}C_{\circ}\mu_{1}^{(6-n)/6}\mu_{2}^{(6-n)/12}\chi^{\frac{n-4}{4}}\right), 
		\end{aligned}$$ 
		then  from Lemma \ref{l0}, 
		$$\begin{aligned}
			\lambda_{\mu_1}+\lambda_{\mu_2} &\geq \frac{1}{\mu_1}\int_{\rn}\left(\left|\nabla u_{\mu_1}\right|^{2}+V_{2}(x)|u_{\mu_1}|^{2}\right)\,d x+\frac{\mu_2}{\mu_1}\frac{1}{\mu_2}\int_{\rn}\left(\kappa\left|\nabla v_{\mu_2}\right|^{2}+V_{2}(x)|v_{\mu_2}|^{2}\right)\,d x \\
			&\geq l_{0}\left(1+\frac{\mu_2}{\mu_1}\sqrt{\kappa}\right)\geq l_{0}\min\{1,\sqrt{\kappa}\}\left(1+\frac{\mu_2}{\mu_1}\right),
		\end{aligned}$$
		if $(\mu_1,\mu_2) \rightarrow (0,0)$.
		So, \begin{equation}\label{lambda4}
			\lambda_{\mu_1}+\lambda_{\mu_2} \geq l_{0}\min\{1,\sqrt{\kappa}\}\left(1+\frac{\mu_2}{\mu_1}\right).
		\end{equation}
		On the other hand, if $\mu_1 < \mu_2$, performing the same process, we obtain
		\begin{equation}\label{lambda5}
			\lambda_{\mu_1}+\lambda_{\mu_2}<l_{0}\max\{1,\sqrt{\kappa}\}\left(1+\frac{\mu_2}{\mu_1}\right).
		\end{equation}
		and if $(\mu_1,\mu_2) \rightarrow (0,0)$
		\begin{equation}\label{lambda6}
			\lambda_{\mu_1}+\lambda_{\mu_2} \geq l_{0}\min\{1,\sqrt{\kappa}\}\left(1+\frac{\mu_1}{\mu_2}\right).
		\end{equation}
		Therefore, combining \eqref{lambda1}-\eqref{lambda6}, the result follows.
		
		We will now try to show \eqref{estimat1}. Let $(u_{\mu_1},v_{\mu_2}) \in \mathcal{D}_{\mu_1,\mu_2}^{\chi}$. First, we consider $\Psi_j\left(x^{\prime}\right), \Phi_j\left(x^{\prime}\right)$ and $l_j, m_j$ for $j \geq 0$ such that
		\begin{equation}\label{equation1}
			\begin{array}{ll}
				-\Delta_{x^{\prime}} \Psi_j+V_2(x) \Psi_j=l_j \Psi_j,\, \,  \int_{\R^{n-1}}\left|\Psi_j\right|^2 d x^{\prime}=1,    \\
				-\kappa\Delta_{x^{\prime}} \Phi_j+V_2(x) \Phi_j=m_j \Phi_j, \, \, \int_{\R^{n-1}}\left|\Phi_j\right|^2 d x^{\prime}=1,  
			\end{array}
		\end{equation}
		where  $l_j \leq l_{j+1}, j=0,1,2, \ldots$ and $m_j \leq m_{j+1}, j=0,1,2, \ldots$. Then, we write 
		$$
		u_{\mu_1}(x)=\sum_{j=0}^{\infty} \varphi_j\left(x_n\right) \Psi_j\left(x^{\prime}\right),\, v_{\mu_2}(x)=\sum_{j=0}^{\infty} \psi_j\left(x_n\right) \Phi_j\left(x^{\prime}\right),
		$$
		where
		$$
		\varphi_j\left(x_n\right)=\int_{\R^{n-1}} u_{\mu_1}(x) \Psi_j\left(x^{\prime}\right) d x^{\prime}, \, \, \psi_j\left(x_3\right)=\int_{\R^{n-1}} v_{\mu_2}(x) \Phi_j\left(x^{\prime}\right) d x^{\prime},
		$$
		and $\left\{\Psi_j\right\}$ and $\left\{\Phi_j\right\}$  are   Hilbert bases for $L^2\left(\R^{n-1}\right)$. Thus, from \eqref{equation1}
		$$
		\begin{aligned}			 \mu_1&=\|u_{\mu_1}\|_{L^2(\rn)}^2=\sum_{j=0}^{\infty}\left\|\varphi_j\right\|_{L^2(\rn)}^2\left\|\Psi_j\right\|_{L^2(\rn)}^2=\sum_{j=0}^{\infty}\left\|\varphi_j\right\|_{L^2(\rn)}^2,\\
			\mu_2&=\|v_{\mu_2}\|_{L^2(\rn)}^2=\sum_{j=0}^{\infty}\left\|\psi_j\right\|_{L^2(\rn)}^2\left\|\Phi_j\right\|_{L^2(\rn)}^2=\sum_{j=0}^{\infty}\left\|\psi_j\right\|_{L^2(\rn)}^2\\
			\|(u_{\mu_1},v_{\mu_2})\|_{\dot{H}}^2& =\sum_{j=0}^{\infty}\left(\left\|\varphi_j\right\|_{L^2(\rn)}^2\left\|\Psi_j\right\|_{\dot{H}}^2+\left\|\psi_j\right\|_{L^2(\rn)}^2 \left\|\Phi_j\right\|_{\dot{H}}^2\right)\\&=\sum_{j=0}^{\infty} \left(l_j\left\|\varphi_j\right\|_{L^2(\rn)}^2+m_j\left\|\psi_j\right\|_{L^2(\rn)}^2\right).
		\end{aligned}
		$$
		Now, since $\epsilon_0=\min\{1,\kappa\}$, then there exists $C>0$ such that
		\begin{equation}\label{estimat2}
			\begin{aligned}
				I(u_{\mu_1}, v_{\mu_2}) 
				&=\frac{1}{2}\int_{\rn}\left(\left|\nabla u_{\mu_1}\right|^{2}+\kappa\left|\nabla v_{\mu_2}\right|^{2}+V_{2}(x)\left(|u_{\mu_1}|^{2}+|v_{\mu_2}|^{2}\right) \right)\,d x-K (u_{\mu_1},v_{\mu_2})\\
				&\geq \frac{1}{2}\Big(\int_{\rn}\left|\nabla u_{\mu_1}\right|^{2}\,d x+\kappa\int_{\rn}\left|\nabla v_{\mu_2}\right|^{2}\,d x+\int_{\rn}V_{2}(x)\left(|u_{\mu_1}|^{2}+|v_{\mu_2}|^{2}\right)\,d x\\&\quad \quad \quad-C_{\circ}\mu_{1}^{(6-n)/6}\mu_{2}^{(6-n)/12}\chi^{n/4} \Big)\\
				&\geq  \frac{\epsilon_{0}}{2}\|(u_{\mu_1},v_{\mu_2})\|_{\dot{H}}^2-C_{\circ}\mu_{1}^{(6-n)/6}\mu_{2}^{(6-n)/12}\chi^{n/4}   \\
				&\geq  \frac{\epsilon_{0}}{2}\|(u_{\mu_1},v_{\mu_2})\|_{\dot{H}}^2-C\left(\mu_1^2+\mu_2\right)^{\frac{6-n}{6}} \\
				&=  \frac{\epsilon_{0}}{2} \sum_{j=0}^{\infty} \left(l_j\left\|\varphi_j\right\|_{L^2(\rn)}^2+m_j\left\|\psi_j\right\|_{L^2(\rn)}^2\right)-C\left(\mu_1^2+\mu_2\right)^{\frac{6-n}{6}}\\
				&=\frac{\epsilon_{0}}{2} \sum_{j=1}^{\infty}\left(l_j\left\|\varphi_j\right\|_{L^2(\rn)}^2+m_j\left\|\psi_j\right\|_{L^2(\rn)}^2\right)+\frac{\epsilon_{0}l_0}{2}\left\|\varphi_j\right\|_{L^2(\rn)}^2+\frac{\epsilon_{0}m_0}{2}\left\|\psi_j\right\|_{L^2(\rn)}^2\\
				&\quad-C\left(\mu_1^2+\mu_2\right)^{\frac{6-n}{6}}\\
				&=\frac{\epsilon_{0}}{2} \sum_{j=1}^{\infty}\left(l_j\left\|\varphi_j\right\|_{L^2(\rn)}^2+m_j\left\|\psi_j\right\|_{L^2(\rn)}^2\right)+\frac{\epsilon_{0}l_0}{2}\left(\mu_1+\sqrt{\kappa}\mu_2\right)-C\left(\mu_1^2+\mu_2\right)^{\frac{6-n}{6}}.
			\end{aligned}    
		\end{equation}
		Now we know, from $(iii)$ in Lemma \ref{propd}, that
		\begin{equation}\label{estimat3}
			I(u, v)=d_{\mu_1,\mu_2}^{\chi}<\frac{l_0}{2}\left(\mu_1+\sqrt{\kappa}\mu_2\right).   
		\end{equation}
		Hence, from \eqref{estimat2} and \eqref{estimat3}, we have
		$$
		\begin{aligned}
			l_1\sum_{j=1}^{\infty}\left\|\varphi_j\right\|_{L^2(\rn)}^2+m_1\sum_{j=1}^{\infty}\left\|\psi_j\right\|_{L^2(\rn)}^2&\leq \sum_{j=1}^{\infty}l_j\left\|\varphi_j\right\|_{L^2(\rn)}^2+\sum_{j=1}^{\infty}m_j\left\|\psi_j\right\|_{L^2(\rn)}^2\\& \leq 2 C\left(\mu_1^2+\mu_2\right)^{\frac{6-n}{6}}+(1-\epsilon_{0})l_0\left(\mu_1+\sqrt{\kappa}\mu_2\right),   
		\end{aligned}
		$$
		that is
		\begin{equation}\label{O1}
			\sum_{j=1}^{\infty}\left(\left\|\varphi_j\right\|_{L^2(\rn)}^2+\left\|\psi_j\right\|_{L^2(\rn)}^2\right) \leq \frac{2 C\left(\mu_1^2+\mu_2\right)^{\frac{6-n}{6}}+(1-\epsilon_{0})l_0\left(\mu_1+\sqrt{\kappa}\mu_2\right)}{\min\{l_1,m_1\}},    
		\end{equation}
		and 
		\begin{equation}\label{O2}
			\sum_{j=1}^{\infty}l_j\left\|\varphi_j\right\|_{L^2(\rn)}^2+\sum_{j=1}^{\infty}m_j\left\|\psi_j\right\|_{L^2(\rn)}^2 \leq 2 C\left(\mu_1^2+\mu_2\right)^{\frac{6-n}{6}}+(1-\epsilon_{0})l_0\left(\mu_1+\sqrt{\kappa}\mu_2\right).   
		\end{equation}
		Therefore, for any $(u, v) \in \mathcal{D}_{\mu_1, \mu_2}^\chi$, we see from \eqref{O2}, that as $\left(\mu_1,\mu_2\right) \rightarrow(0,0)$
		$$
		\begin{aligned}
			\left\|(u,v)-(\varphi_0 \Psi_0,\psi_0 \Phi_0)\right\|_{\dot{H}}^2 & =\sum_{j=1}^{\infty}\left\|(\varphi_j \Psi_j,\psi_j \Phi_j\right\|_{\dot{H}}^2 \\
			& =\sum_{j=1}^{\infty}\left(l_j\left\|\varphi_j\right\|_{L^2(\rn)}^2+m_j\left\|\psi_j\right\|_{L^2(\rn)}^2\right)\\
			&=O\left(\mu_1+\mu_2\right)
		\end{aligned}
		$$
		and from \eqref{O1}
		$$
		\begin{aligned}
			\left\|(u,v)-(\varphi_0 \Psi_0,\psi_0 \Phi_0)\right\|_{L^2(\rn)}^2& =\sum_{j=1}^{\infty}\left\|\varphi_j \right\|_{L^2(\rn)}^2\left\|\Psi_j\right\|_{L^2(\rn)}^2+\sum_{j=1}^{\infty}\left\|\psi_j\right\|_{L^2(\rn)}^2 \left\|\Phi_j\right\|_{L^2(\rn)}^2 \\
			& =\sum_{j=1}^{\infty}\left(\left\|\varphi_j\right\|_{L^2(\rn)}^2+\left\|\psi_j\right|_{L^2(\rn)}^2\right)\\&=O\left(\mu_1+\mu_2\right).
		\end{aligned}
		$$
		To show the symmetry of the minimizers in $\mathcal{D}_{\mu_1, \mu_2}^\chi$, let $\left(u, v\right) \in \mathcal{D}_{\mu_1, \mu_2}^\chi$ and denote
		$$
		\left(\tilde{u}\left(x^{\prime}, x_n\right), \tilde{v}\left(x^{\prime}, x_n\right)\right)=\left(u_{x_n}^{\sharp}\left(x^{\prime}\right), v_{x_n}^{\sharp}\left(x^{\prime}\right)\right),
		$$
		where $\left(u_{x_n}\left(x^{\prime}\right), v_{x_n}\left(x^{\prime}\right)\right)=\left(u\left(x^{\prime}, x_n\right), v\left(x^{\prime}, x_n\right)\right)$ and $\sharp$ denotes the Steiner rearrangement a.e.  $x^{\prime}$. Then, from Lemma \ref{steiner} and Theorem \ref{theoremV}, the following properties hold:
		\begin{equation}\label{sc2}
			\begin{gathered}
				\int_{\rn}\left|\nabla_x \tilde{u}\right|^2 \,d x=\int_{\mathbb{R}} \int_{\R^{n-1}} \left|\nabla_x \tilde{u}\right|^2 \,d x' \,d x_{n}\leq \int_{\mathbb{R}} \int_{\R^{n-1}} \left|\nabla_x u\right|^2 \,d x^{\prime}\,d x_{n}= \int_{\rn}\left|\nabla_x u\right|^2 \,d x, \\
				\int_{\rn}V_2(x)|\tilde{u}|^2 \,d x \leq \int_{\rn}V_2(x)|u|^2\,d x\\
				\int_{\rn}V_2(x)|\tilde{v}|^2 \,d x \leq \int_{\rn}V_2(x)|v|^2 \,d x\\
				\int_{\rn}\tilde{u}^2\tilde{v}\,d x \geq \int_{\rn} u^2v \,d x.
			\end{gathered}    
		\end{equation}
		Hence, from \eqref{sc2}, we have 
		$$d_{\mu_1, \mu_2}^\chi\leq I(\tilde{u},\tilde{v})\leq I(u,v)=d_{\mu_1, \mu_2}^\chi,$$
		then $I(\tilde{u},\tilde{v})= I(u,v)$, that is 
		
		$$
		\begin{aligned}
			\int_{\rn}V_2(x)\left(|\tilde{u}|^2+|\tilde{v}|^2\right)\,d x -\int_{\rn}\tilde{u}^2\tilde{v}\,d x   =\int_{\rn}V_2(x)\left(|u|^2+|v|^2 \right)\,d x-\int_{\rn} u^2v \,d x.
		\end{aligned}
		$$
		So, from \eqref{sc2}
		$$
		\begin{aligned}
			\int_{\rn}V_2(x)\left(|\tilde{u}|^2+|\tilde{v}|^2\right)\,d x   \geq \int_{\rn}V_2(x)\left(|u|^2+|v|^2 \right)\,d x.
		\end{aligned}
		$$
		then, $$
		\begin{aligned}
			\int_{\rn}V_2(x)\left(|\tilde{u}|^2-|u|^2\right)\,d x   \geq \int_{\rn}V_2(x)\left(|v|^2 -|\tilde{v}|^{2}\right)\,d x\geq 0.
		\end{aligned}
		$$
		It follows that $$\int_{\rn}V_2(x)|\tilde{u}|^2 \,d x \geq \int_{\rn}V_2(x)|u|^2\,d x.$$
		Hence, \eqref{sc2}  implies that $$\int_{\rn}V_2(x)|\tilde{u}|^2 \,d x = \int_{\rn}V_2(x)|u|^2\,d x.$$
		Similarly, we obtain $$\int_{\rn}V_2(x)|\tilde{v}|^2 \,d x = \int_{\rn}V_2(x)|v|^2\,d x.$$
		Thus,$$
		\begin{aligned}
			& \int_{\mathbb{R}^2}V_2(x)|\tilde{u}|^2 \,d x^{\prime}=\int_{\mathbb{R}^2}V_2(x)|u|^2 \,d x^{\prime}, \,\, \text{a.e} \,\, x_n \in \R\\
			& \int_{\mathbb{R}^2}V_2(x)|\tilde{v}|^2 \,d x^{\prime}=\int_{\mathbb{R}^2}V_2(x)|v|^2 \,d x^{\prime}, \,\, \text{a.e} \,\, x_n \in \R.
		\end{aligned}
		$$
		Using Theorem \ref{theoremV}, we get $\left(u\left(x^{\prime}, x_n\right), v\left(x^{\prime}, x_n\right)\right)=\left(\tilde{u}\left(|x^{\prime}|, x_n\right), \tilde{v}\left(|x^{\prime}|, x_n\right)\right)$ for a.e. $x_n \in \mathbb{R}$. Since $u\left(x^{\prime}, x_n\right)$ and $v\left(x^{\prime}, x_n\right)$ are continuous, then for any $x_n$, we have  $u\left(x^{\prime}, x_n\right)$ and $v\left(x^{\prime}, x_n\right)$ are radially symmetric and  nonincreasing a.e. $x^{\prime}$.
		
		Finally, we establish that the minimizers within $\mathcal{D}_{\mu_1, \mu_2}^\chi$ represent the least-energy normalized solutions to problem \eqref{elliptic}. In fact, let $\left(u_{\mu_1}, v_{\mu_2}\right) \in\mathcal{D}_{\mu_1, \mu_2}^\chi $. From    Lemma \ref{pointcrit}, we have $d_{\mu_1, \mu_2}^\chi=I\left(u_{\mu_1}, v_{\mu_2}\right)>0$. It follows by item $(iii)$ of  Lemma \ref{propd} that
		\begin{equation}\label{conver}
			d_{\mu_1, \mu_2}^\chi \rightarrow 0 \text { as }\left(\mu_1, \mu_2\right) \rightarrow(0,0)   
		\end{equation}
		Let us assume by contradiction that there exists a critical point $(\overline{u}, \overline{v})$ for $I$ on $S\left(\mu_1,\mu_2\right)$ with $I(\overline{u}, \overline{v})<d_{\mu_1, \mu_2}^\chi$. Hence, remembering $\blue{B}$ in Lemma \ref{pointcrit}, we have
		$$
		\begin{aligned}
			d_{\mu_1, \mu_2}^\chi&>I(\overline{u}, \overline{v})\\ & =I(\tilde{u}, \tilde{v})-2\frac{\blue{B(\tilde{u}, \tilde{v})}}{n} \\
			& =\frac{n-4}{2n} \int_{\rn}\left(|\nabla \tilde{u}|^{2}+\kappa|\nabla \tilde{v}|^{2}\right) d x+\frac{n+4}{2n}\int_{\rn}V(x)\left(|\tilde{u}|^{2}+|\tilde{v}|^{2}\right) d x \\
			& >\frac{\epsilon_{0}(n-4)}{2n}\|(\overline{u},\overline{v})\|_{\dot{H}}^2, 
		\end{aligned}
		$$
		where $\epsilon_0=\min\{1,\kappa\}$.
		In particular, for $\mu_1, \mu_2>0$ small enough, by \eqref{conver}, we get
		$ \|(\overline{u},\overline{v})\|_{\dot{H}}^2<\chi$. It follows that 
		$$
		I(\overline{u}, \overline{v}) \geq d_{\mu_1, \mu_2}^\chi,
		$$ which is a contradiction. Therefore, when $\mu_1$ and $\mu_2$ are small enough, the minimizer $\left(u_{\mu_1}, v_{\mu_2}\right)$ of $d_{\mu_1, \mu_2}^\chi$ is a critical point of $I(u, v)$ constrained on $S\left(\mu_1,\mu_2\right)$ and
		$$
		I\left(u_{\mu_1}, v_{\mu_2}\right)=\inf  \sett{I(u, v),\;(u, v) \in S\left(\mu_1,\mu_2\right), I' (u, v)=0 }.
		$$
	\end{proof}
		\vspace{3mm}
		
		\subsection{Asymptotic behavior}
Here, we aim to establish a direct relationship between the existence of the previously obtained minimizer and a minimizer for the   one-dimensional system \eqref{4elliptic}, in the sense of Theorem \ref{theorem2.27}.  
  
Initially we are going to consider a minimization problem associated with the system \eqref{4elliptic}. In fact, 
  we define the following minimization problem
	$$\mathcal{J}_{\infty}^{\mu_1,\mu_2}=\inf_{(u,v)\in H^1(\R)\times H^1(\R)}\left\{I_{\infty}(u,v),\;\|u\|_{\lt}^2=\mu_1,\, \, \|v\|_{\lt}^2=\mu_2\right\}$$ 
		$$I_{\infty}(u,v)=\frac{1}{2}\|\partial_{x_n}u\|_{\lt}^2+\frac{\kappa}{2}\|\partial_{x_n}v\|_{\lt}^2-\frac{\kappa}{3\sqrt{\pi}}\Re\int_{\R}u^2\bar v\,d x_n.$$
		We recall that the problem $\mathcal{J}_{\infty}^{\mu_1,\mu_2}$ possesses a positive symmetric decreasing ground state $(\mathcal{D}_{\infty}^1,\mathcal{D}_{\infty}^2)$, and it solves the Euler-Lagrange system \eqref{4elliptic} (see \cite{pas1}). 
Replacing $(\lambda_1,\lambda_2)$ with $(\lambda_1-L_0, \lambda_2-M_0)$, we rewrite \eqref{elliptic} as
  \begin{equation}\label{5elliptic}
		\left\{\begin{array}{l}
			(H_1-L_0)\blue{\phi}-\partial_{x_n}^2\blue{\phi}-\blue{\overline{\phi}_1} \blue{\psi}=-\lambda_1 \blue{\phi}, \\ \\
				(H_2-M_0)\blue{\psi}-\kappa \partial_{x_n}^2 \blue{\psi}-\frac{1}{2}\blue{\overline{\phi}}^2=-\lambda_2\blue{\psi},
		\end{array}\right.    
	\end{equation}
		Thus, we consider the   minimization problem $\mathcal{J}_{\mu_1,\mu_2}$ associated with \eqref{5elliptic} in \eqref{NNew-min-as}.

		Therefore, following the strategies used in Theorem \ref{2theorem}, with some   modifications, we can obtain the existence of a minimizer $\blue{(\mathcal{D}_1, \mathcal{D}_2)}$ satisfying the properties imposed in Theorem \ref{2theorem}.
		
		We introduce $\Psi_j\left(x^{\prime}\right), \Phi_j\left(x^{\prime}\right)$ and $L_j, M_j$ for $j \geq 0$ such that
		$$
		\begin{gathered}
			\begin{array}{ll}
				-\Delta_{x^{\prime}} \Psi_j+V_2(x) \Psi_j=L_j \Psi_j,\, \,  \int_{\mathbb{R}^{n-1}}\left|\Psi_j\right|^2 d x^{\prime}=1,    \\
				-\kappa\Delta_{x^{\prime}} \Phi_j+V_2(x) \Phi_j=M_j \Phi_j, \, \, \int_{\mathbb{R}^{n-1}}\left|\Phi_j\right|^2 d x^{\prime}=1,  
			\end{array}
		\end{gathered}
		$$
		where  $L_j \leq L_{j+1}, j=0,1,2, \ldots$ and $M_j \leq M_{j+1}, j=0,1,2, \ldots$.
		It is well known that $\left\{\Psi_j\right\}$ and $\left\{\Phi_j\right\}$ are   Hilbert bases of $L^2\left(\mathbb{R}^{n-1}\right)$, and the corresponding eigenfunctions for $L_0$ and $ M_0$ are given by $\Psi_0(x^{\prime})=\pi^{-\frac n2}e^{-\frac{|x^{\prime}|^2}{2}}$ and $\Phi_0(x^{\prime})=\pi^{-\frac n2}e^{-\frac{|x^{\prime}|^2}{2\kappa}}$ respectively. Hence, from the spectral representation, we can write
		$$
		u(x)=\sum_{j=0}^{\infty}\langle u(\cdot,x_n),\Psi_{j}\rangle_{L^{2}(\R^{n-1})}\Psi_{j}(x_n) 
		$$
		and 
		$$ 
		v(x)=\sum_{j=0}^{\infty}\langle v(\cdot,x_n),\Phi_{j}\rangle_{L^{2}(\R^{n-1})}\Phi_{j}(x_n),
		$$
		where 
		$$\langle u(\cdot,x_n),\Psi_{j}\rangle_{L^{2}(\R^{n-1})}=\int_{\R^{n-1}}u(x^{\prime},x_n)\Psi_j(x^{\prime})\,d x^{\prime}
		$$
		and	
		$$
		\langle v(\cdot,x_n),\Phi_{j}\rangle_{L^{2}(\R^{n-1})}
		=\int_{\R^{n-1}}v(x^{\prime},x_n)\Phi_j(x^{\prime})\,d x^{\prime}.
		$$
		We denote the $\Psi_0$-directional component of $u$ by
		$$
		\tilde{u}(x_n)=\left\langle u(\cdot, x_n), \Psi_0\right\rangle_{L^{2}(\R^{n-1})},
		$$
		and we define the  projection onto the lowest eigenspace by
		$$
		\left(P_0^1 u\right)(x):=\tilde{u}(x_n) \Psi_0(x^{\prime}),
		$$
		and let $P_1^1=1-P_0^1$ be the projection to the orthogonal complement, precisely,
		$$
		\left(P_1^1 u\right)(x)=\sum_{j=1}^{\infty}\left\langle u(\cdot, x_n), \Psi_j\right\rangle_{L^{2}(\R^{n-1})}\Psi_j(x^{\prime}).
		$$
		Then, we have
		$$
		\|u\|_{\lt}^2=\sum_{j=0}^{\infty}\left\|\left\langle u(\cdot, x_n), \Phi_j\right\rangle_{L^{2}(\R^{n-1})}\right\|_{L^2(\R)}^2=\left\|P_0^1 u\right\|_{L^2\left(\rn\right)}^2+\left\|P_1^1 u\right\|_{L^2\left(\rn\right)}^2.
		$$
		Similarly, we define for $v$. That is, 
		$$
		\tilde{v}(x_n)=\left\langle v(\cdot, x_n), \Phi_0\right\rangle_{L^{2}(\R^{n-1})},
		\qquad
		\left(P_0^2 v\right)(x):=\tilde{v}(x_n) \Phi_0(x^{\prime})
		$$
		and $$
		\left(P_1^2 u\right)(x)=\sum_{j=1}^{\infty}\left\langle v(\cdot, x_n), \Phi_j\right\rangle_{L^{2}(\R^{n-1})}\Phi_j(x^{\prime}).
		$$
		Now, we present a preliminary result that establishes a relationship between the two minimization problems.
		\begin{lemma}\label{lemD1}
			Let $(\mathcal{D}_1, \mathcal{D}_2)$ be the minimizer constructed for the system \eqref{5elliptic}. Then, 
			
			\begin{equation}\label{0eqD1}
				\mathcal{J}_{\mu_1,\mu_2}=\mathcal{J}_{\infty}^{\mu_1,\mu_2}(m)+O\left(\mu_1+\mu_2\right).   
			\end{equation}
			Moreover, 
			\begin{equation}\label{eqD1}
				\left\|P_1^1 \mathcal{D}_1\right\|_{L^2\left(\rn\right)}+\left\|P_1^2 \mathcal{D}_2\right\|_{L^2\left(\rn\right)} \leq C\left\|(\mathcal{D}_1,\mathcal{D}_2)\right\|_{H_{x^{\prime}}} \leq C(\mu_1+\mu_2)    
			\end{equation}
			\begin{equation}\label{eqD12}
				\left\|\partial_{x_n}\left(P_1^1\mathcal{D}_1\right)\right\|_{L^2\left(\rn\right)} +\left\|\partial_{x_n}\left(P_1^2\mathcal{D}_2\right)\right\|_{L^2\left(\rn\right)} \leq C(\mu_1+\mu_2)
			\end{equation}
			and \begin{equation}\label{eqD2}
				|\lambda_1-\lambda_{\infty}^1|+|\lambda_2-\lambda_{\infty}^2|\leq C(\mu_1+\mu_2).
			\end{equation}
		\end{lemma}
		\begin{proof}
			First, we show that
			$$
			\left\|P_1^1 \mathcal{D}_1\right\|_{L^2\left(\rn\right)}+\left\|P_1^2 \mathcal{D}_2\right\|_{L^2\left(\rn\right)} \leq C\left\|(\mathcal{D}_1,\mathcal{D}_2)\right\|_{H_{x^{\prime}}} \leq C(\mu_1+\mu_2). 
			$$
			Note that, $$\begin{aligned}
				(H_1-L_0)\mathcal{D}_1&=(H_1-L_0)P_0^1\mathcal{D}_1+(H_1-L_0)P_1^1\mathcal{D}_1 \\
				&=H_1P_0^1\mathcal{D}_1-L_0P_0^1\mathcal{D}_1+\blue{(H_1-L_0)P_1^1\mathcal{D}_1}\\
				&=L_0P_0^1\mathcal{D}_1-L_0P_0^1\mathcal{D}_1+\blue{(H_1-L_0)P_1^1\mathcal{D}_1}\\
				&=(H_1-L_0)P_1^1\mathcal{D}_1\\
				&=H_1P_1^1\mathcal{D}_1-L_0P_1^1\mathcal{D}_1\\
				&=(L_1-L_0)P_1^1\mathcal{D}_1,
			\end{aligned}$$
			then, since $L_1>L_0$, we have  $$\|P_1^1\mathcal{D}_1\|_{\lt}\leq C\|(H_1-L_0)\mathcal{D}_1\|_{\lt}.$$
			Moreover,  
			$$\begin{aligned}
				(H_1-L_0)\mathcal{D}_1-\partial_{x_n}^2 \mathcal{D}_1+\lambda_1\mathcal{D}_1&=(H_1-L_0)P_0^1\mathcal{D}_1-\partial_{x_n}^2 P_0^1\mathcal{D}_1+\lambda_1P_0^1\mathcal{D}_1\\
				& \quad +(H_1-L_0)P_1^1\mathcal{D}_1-\partial_{x_n}^2 P_1^1\mathcal{D}_1+\lambda_1 P_1^1\mathcal{D}_1\\
				&=(H_1-L_0)P_1^1\mathcal{D}_1-\partial_{x_n}^2 \mathcal{D}_1+\lambda_1 \mathcal{D}_1.
			\end{aligned}$$
			Since, $\lambda_1\geq 0$ (see Theorem \ref{2theorem}), we have
			$$
			\begin{aligned}
				\|P_1^1\mathcal{D}_1\|_{\lt}&\leq C\|(H_1-L_0)\mathcal{D}_1\|_{\lt}\\
				&\leq C\| (H_1-L_0)\mathcal{D}_1-\partial_{x_n}^2 \mathcal{D}_1+\lambda_1\mathcal{D}_1\|_{\lt}\\&
				=C\|\mathcal{D}_1\mathcal{D}_2\|_{\lt}\\
				&\leq C (\mu_1+\mu_2).
			\end{aligned}
			$$
			Similarly, we obtain
			$$
			\begin{aligned}
				\|P_1^2\mathcal{D}_2\|_{\lt}\leq C (\mu_1+\mu_2).
			\end{aligned}
			$$
			We denote by  $\tilde{\mathcal{D}}_1(x_n)=\left\langle \mathcal{D}_1(\cdot, x_n), \Psi_0\right\rangle_{L^2\left(\R^{n-1}\right)}$ and $\tilde{\mathcal{D}}_2(x_n)=\left\langle \mathcal{D}_2(\cdot, x_n), \Phi_0\right\rangle_{L^2\left(\R^{n-1}\right)}$. Note that if we denote by  $\tilde{\mu}_1=\left\|\tilde{\mathcal{D}}_1\right\|_{L^2(\R)}^2$ and $\tilde{\mu}_2=\left\|\tilde{\mathcal{D}}_2\right\|_{L^2(\R)}^2$, 
			then we have  $$\begin{aligned}
				\mu_j =\left\|\mathcal{D}_j\right\|_{L^2(\R)}^2    =\left\|P_0^j\mathcal{D}_j\right\|_{L^2(\R)}^2+\left\|P_1^j\mathcal{D}_j\right\|_{L^2(\R)}^2 =\tilde{\mu}_j+\left\|P_1^j\mathcal{D}_j\right\|_{L^2(\R)}^2.
			\end{aligned}$$
			Hence, from \eqref{eqD1}, it follows that $$\tilde{\mu}_j=\mu_j+O(\mu_1+\mu_2).$$
			Therefore, $$
			\begin{aligned}
				I_{\infty}\left(\mathcal{D}_{\infty}^{1},\mathcal{D}_{\infty}^{2}\right) & \leq I_{\infty}\left(\sqrt{\frac{\mu_1+\mu_2}{\tilde{\mu}_1+\tilde{\mu}_2}} (\tilde{\mathcal{D}}_1,\tilde{\mathcal{D}}_2)\right)=I_{\infty}\left(\tilde{\mathcal{D}}_1,\tilde{\mathcal{D}}_2\right)+O\left(\mu_1+\mu_2\right) \\
				& =\frac{1}{2}\left\|\partial_{x_n}\left(P_0^1 \mathcal{D}_1\right)\right\|_{L^2\left(\rn\right)}^2+\frac{\kappa}{2}\left\|\partial_{x_n}^2\left(P_0^1 \mathcal{D}_1\right)\right\|_{L^2\left(\rn\right)}^2\\
				&\quad-\frac{\kappa}{3\sqrt{\pi}}\int_{\rn}(P_0^1\mathcal{D}_1)^2P_0^2\mathcal{D}_2\,d x+O\left(\mu_1+\mu_2\right)\\
				& =J\left(\mathcal{D}_1,\mathcal{D}_2\right)-\frac{1}{2}\left\|(\mathcal{D}_1,\mathcal{D}_2)\right\|_{\dot{H}_{x^{\prime}}}^2-\frac{1}{2}\left\|\partial_{x_n}\left(P_1^1 \mathcal{D}_1\right)\right\|_{L^2\left(\rn\right)}^2\\
				&\quad-\frac{1}{2}\left\|\partial_{x_n}\left(P_1^2\mathcal{D}_2\right)\right\|_{L^2\left(\rn\right)}^2+O\left(\mu_1+\mu_2\right).
			\end{aligned}
			$$
			Since $J\left(\mathcal{D}_1,\mathcal{D}_2\right) \leq I_{\infty}\left(\mathcal{D}_{\infty}^{1},\mathcal{D}_{\infty}^{2}\right)$, it follows that
			\begin{equation}\label{estimativeQ}
				\frac{1}{2}\left\|(\mathcal{D}_1,\mathcal{D}_2)\right\|_{\dot{H}_{x^{\prime}}}^2+\frac{1}{2}\left\|\partial_{x_n}\left(P_1^1 \mathcal{D}_1\right)\right\|_{L^2\left(\rn\right)}^2
				+\frac{1}{2}\left\|\partial_{x_n}\left(P_1^2\mathcal{D}_2\right)\right\|_{L^2\left(\rn\right)}^2\leq C(\mu_1+\mu_2),  
			\end{equation}
			hence,  \eqref{eqD12}  is deduced.
			Then,   we obtain $$\mathcal{J}_{\mu_1,\mu_2}=\mathcal{J}_{\infty}^{\mu_1,\mu_2}+O\left(\mu_1+\mu_2\right).$$
			Finally, to obtain the last term, we proceed to multiply system \eqref{5elliptic}  by $(\mathcal{D}_1, \mathcal{D}_2)$ and $(x_n\partial_{x_n}\mathcal{D}_1, x_n\partial_{x_n}\mathcal{D}_2)$. Then we obtain, 
			$$
			\begin{aligned}
				0 & =\left\langle\left(H_1-L_0\right) \mathcal{D}_1-\partial_{x_n}^2 \mathcal{D}_1-\mathcal{D}_1\mathcal{D}_2+\lambda_1 \mathcal{D}_1, \mathcal{D}_1\right\rangle_{L^2\left(\rn\right)} \\
				& =\int_{\rn}\left(\left|\nabla_{x^{\prime}} \mathcal{D}_1\right|^2+V_2|\mathcal{D}_1|^2-\blue{L_0}|\mathcal{D}_1|^2\right)\,d x+\left\|\partial_{x_n} \mathcal{D}_1\right\|_{L^2\left(\rn\right)}^2-K(\mathcal{D}_1,\mathcal{D}_2)+\lambda_1 \mu_1
			\end{aligned}
			$$
			and 
			$$
			\begin{aligned}
				0 & =\left\langle\left(H_2-\blue{M_0}\right) \mathcal{D}_2-\partial_{x_n}^2 \mathcal{D}_2-\frac{1}{2}(\mathcal{D}_1)^2+\lambda_2 \mathcal{D}_2, \mathcal{D}_2\right\rangle_{L^2\left(\rn\right)} \\
				& =\int_{\rn}\left(\kappa\left|\nabla_{x^{\prime}} \mathcal{D}_2\right|^2+V_2|\mathcal{D}_2|^2-M_0|\mathcal{D}_2|^2\right)\,d x+\kappa\left\|\partial_{x_n} \mathcal{D}_2\right\|_{L^2\left(\rn\right)}^2-\frac{1}{2}K(\mathcal{D}_1,\mathcal{D}_2)+\lambda_2 \mu_2.
			\end{aligned}
			$$
			Hence, we have 
			\begin{equation}\label{firstequat}
				\begin{split}
					0&=\left\|(\mathcal{D}_1,\mathcal{D}_2)\right\|_{\dot{H}_{x^{\prime}}}^2+\kappa\left\|\partial_{x_n} \mathcal{D}_2\right\|_{L^2\left(\rn\right)}^2+\left\|\partial_{x_n} \mathcal{D}_1\right\|_{L^2\left(\rn\right)}^2\\&\qquad-\frac{3}{2}K(\mathcal{D}_1,\mathcal{D}_2)+\lambda_2 \mu_2+\lambda_1 \mu_1.
				\end{split}
			\end{equation}
			Furthermore, by performing the same process changing $(\mathcal{D}_1, \mathcal{D}_2)$ by  $(x_n\partial_{x_n}\mathcal{D}_1, x_n\partial_{x_n}\mathcal{D}_2)$, we have 
			$$\begin{aligned}
				0 & =\left\langle\left(H_1-L_0\right) \mathcal{D}_1-\partial_{x_n}^2 \mathcal{D}_1-\mathcal{D}_1\mathcal{D}_2+\lambda_1 \mathcal{D}_1, x_n\partial_{x_n}\mathcal{D}_1\right\rangle_{L^2\left(\rn\right)} \\
				& =\int_{\rn} \left(\frac{ x_n}{2} \partial_{x_n}\left(\sqrt{H_1-\blue{L_0}} \mathcal{D}_1\right)^2-\frac{x_n}{2} \partial_{x_n}\left(\partial_{x_n} \mathcal{D}_1\right)^2-\frac{x_n}{2} \mathcal{D}_2\partial_{x_n}\mathcal{D}_1^2+\frac{\lambda_1 x_n}{2} \partial_{x_n}\mathcal{D}_1^2 \right)\,d x 
			\end{aligned}
			$$
			and $$\begin{aligned}
				0 & =\left\langle\left(H_2-\blue{M_0}\right) \mathcal{D}_2-\partial_{x_n}^2 \mathcal{D}_2-\frac{1}{2}\mathcal{D}_1^2+\lambda_2 \mathcal{D}_2, x_n\partial_{x_n}\mathcal{D}_2\right\rangle_{L^2\left(\rn\right)} \\
				& =\int_{\rn} \left(\frac{ x_n}{2} \partial_{x_n}\left(\sqrt{H_2-\blue{M_0}} \mathcal{D}_2\right)^2-\frac{x_n}{2} \partial_{x_n}\left(\partial_{x_n} \mathcal{D}_2\right)^2-\frac{x_n}{2} \mathcal{D}_1^2\partial_{x_n}\mathcal{D}_2+\frac{\lambda_2 x_n}{2} \partial_{x_n}\mathcal{D}_2^2 \right)\,d x.
			\end{aligned}
			$$
			Therefore, 
			\begin{equation}\label{secondequat}
				\begin{split}
					0&=-\frac{1}{2}\left\|(\mathcal{D}_1,\mathcal{D}_2)\right\|_{\dot{H}_{x^{\prime}}}^2+\frac{\kappa}{2}\left\|\partial_{x_n} \mathcal{D}_2\right\|_{L^2\left(\rn\right)}^2+\frac{1}{2}\left\|\partial_{x_n} \mathcal{D}_1\right\|_{L^2\left(\rn\right)}^2\\
					&\qquad+\frac{1}{2}K(\mathcal{D}_1,\mathcal{D}_2)-\frac{\lambda_2}{2} \mu_2-\frac{\lambda_1}{2} \mu_1.  
			\end{split}\end{equation}
			Now, from \eqref{firstequat} and \eqref{secondequat}, it follows 
			$$ \frac{1}{2}
			K(\mathcal{D}_1,\mathcal{D}_2)=\frac{2}{5}\left\|(\mathcal{D}_1,\mathcal{D}_2)\right\|_{\dot{H}_{x^{\prime}}}^2+\frac{2\lambda_2}{5} \mu_2+\frac{2\lambda_1}{5} \mu_1  $$
			and 
			$$\frac{\kappa}{2}\left\|\partial_{x_n} \mathcal{D}_2\right\|_{L^2\left(\rn\right)}^2+\frac{1}{2}\left\|\partial_{x_n} \mathcal{D}_1\right\|_{L^2\left(\rn\right)}^2=\frac{1}{10}\left\|(\mathcal{D}_1,\mathcal{D}_2)\right\|_{\dot{H}_{x^{\prime}}}^2+\frac{\lambda_2}{10} \mu_2+\frac{\lambda_1}{10} \mu_1.$$
			Then, it is seen from \eqref{estimativeQ} that
			$$\begin{aligned}
				J(\mathcal{D}_1,\mathcal{D}_2)&=\|(\mathcal{D}_1,\mathcal{D}_2)\|_{\dot{H}_{x^{\prime}}}^2+\frac{1}{2}\left\|\partial_{x_n} \mathcal{D}_1\right\|_{\lt}^2+\frac{\kappa}{2}\left\|\partial_{x_n} \mathcal{D}_2\right\|_{\lt}^2-\frac{1}{2}K(\mathcal{D}_1,\mathcal{D}_2)\\
				&=\|(\mathcal{D}_1,\mathcal{D}_2)\|_{\dot{H}_{x^{\prime}}}^2+\frac{1}{10}\left\|(\mathcal{D}_1,\mathcal{D}_2)\right\|_{\dot{H}_{x^{\prime}}}^2+\frac{\lambda_2}{10} \mu_2+\frac{\lambda_1}{10} \mu_1-\frac{2}{5}\left\|(\mathcal{D}_1,\mathcal{D}_{2})\right\|_{\dot{H}_{x^{\prime}}}^2\\
				&\qquad-\frac{2\lambda_2}{5} \mu_2-\frac{2\lambda_1}{5} \mu_1\\
				&=\|(\mathcal{D}_1,\mathcal{D}_2)\|_{\dot{H}_{x^{\prime}}}^2-\frac{3}{10}\left\|(\mathcal{D}_1,\mathcal{D}_{2})\right\|_{\dot{H}_{x^{\prime}}}^2-\frac{3\lambda_2}{10} \mu_2+\frac{3\lambda_1}{10} \mu_1\\&=
				-\frac{3\lambda_2}{10} \mu_2+\frac{3\lambda_1}{10} \mu_1 +O(\mu_1+\mu_2).
			\end{aligned}
			$$
			On the other hand, by  a similar argument, it is not difficult to show that 
			$$I_{\infty}(\mathcal{D}_{\infty}^{1},\mathcal{D}_{\infty}^{2})= -\frac{3\lambda_{\infty}^2}{10} \mu_2+\frac{3\lambda_{\infty}^1}{10} \mu_1 +O(\mu_1+\mu_2).$$
			Therefore, using \eqref{0eqD1}, it follows that
			$$|\lambda_1-\lambda_{\infty}^1|+|\lambda_2-\lambda_{\infty}^2|\leq C(\mu_1+\mu_2).$$
		\end{proof}
		We recall a result on an estimate of the one-dimensional linearized operators
		$
		\mathcal{L}_{\infty}^1=-\partial_{x_n}^2+\lambda_{\infty}^1-2s_n^1 \mathcal{D}_{\infty}^{2}
		$
		and $
		\mathcal{L}_{\infty}^2=-\kappa\partial_{x_n}^2+\lambda_{\infty}^2-s_n^2\mathcal{D}_{\infty}^{1}.
		$ 
		The proof is similar to one in    \cite{hong}.
		\begin{lemma}\label{nodeg}
			The linearized operator  $\mathcal{L}_{\infty}^j$ for $j=1,2$  satisfies $\left\|\mathcal{L}_{\infty}^j \varphi\right\|_{H^{-1}(\R)} \geq C\|\varphi\|_{H^1(\R)}$ for all even $\varphi \in H^1(\R)$.
		\end{lemma}
		
		Now, we finally establish the desired result in Theorem \ref{theorem2.27}.

		\begin{proof}[Proof of Theorem \ref{theorem2.27}]
			First, we   note that from Theorem \ref{2theorem}  that
			\begin{equation}\label{theD0}
				\left\|(\mathcal{D}_1,\mathcal{D}_2)-(\tilde{\mathcal{D}}_1\Psi_0,\tilde{\mathcal{D}}_2 \Phi_0)\right\|_{(L^2(\rn)\times L^2(\rn)) \cap \dot{H}_{x^{\prime}}} \leq C(\mu_1+\mu_2).   
			\end{equation}		
			Moreover, from Lemma \ref{lemD1}, it is seen that
			 $$\begin{aligned}
				& \left\|(\partial_{x_n}(\mathcal{D}_1-\tilde{\mathcal{D}}_1\Psi_0), \partial_{x_n}(\mathcal{D}_2-\tilde{\mathcal{D}}_2 \Phi_0))\right\|_{L^2(\rn)\times L^2(\rn)}\\&\qquad\qquad=\left\|(\partial_{x_n}(\mathcal{D}_1-P_0^1 \mathcal{D}_1), \partial_{x_n}(\mathcal{D}_2-P_0^2 \mathcal{D}_2))\right\|_{L^2(\rn)\times L^2(\rn)}\\&\qquad\qquad=\left\|\partial_{x_n}(P_1^1\mathcal{D}_1,P_1^2\mathcal{D}_2)\right\|_{L^2(\rn)\times L^2(\rn)) }\\
				&\qquad\qquad\leq C(\mu_1+\mu_2). 
			\end{aligned}$$ Then, by Lemma \ref{lemD1}, it suffices to show that
			\begin{equation}\label{theD1}
				\begin{aligned}
					& \left\|(\tilde{\mathcal{D}}_{1}(x_n)\Psi_0(x^{\prime}),\tilde{\mathcal{D}}_{2}(x_n))\Phi_0(x^{\prime})-(\mathcal{D}_{\infty}^{1}(x_n)\Psi_0(x^{\prime}),\mathcal{D}_{\infty}^{1}(x_n)\Phi_0(x^{\prime}))\right\|_{H^1(\rn)\times H^1(\rn)}\\&\qquad=\left\|(\tilde{\mathcal{D}}_{1}(x_n),\tilde{\mathcal{D}}_{2}(x_n))-(\mathcal{D}_{\infty}^{1}(x_n),\mathcal{D}_{\infty}^{1}(x_n))\right\|_{H^1(\R)\times H^1(\R)}\\&\qquad \leq O(\mu_1+\mu_2). 
				\end{aligned}  
			\end{equation}
			From the proof of Lemma \ref{lemD1}, we have $$
			J\left(\mathcal{D}_1,\mathcal{D}_2\right) \leq I_{\infty}\left(\mathcal{D}_{\infty}^{1},\mathcal{D}_{\infty}^{2}\right) \leq I_{\infty}\left(\sqrt{\frac{\mu_1+\mu_2}{\tilde{\mu}_1+\tilde{\mu}_2}} (\tilde{\mathcal{D}}_1,\tilde{\mathcal{D}}_2)\right)=I_{\infty}\left(\tilde{\mathcal{D}}_1,\tilde{\mathcal{D}}_2\right)+O\left(\mu_1+\mu_2\right).
			$$ 
			Hence, $\sqrt{\frac{\mu_1+\mu_2}{\tilde{\mu}_1+\tilde{\mu}_2}} (\tilde{\mathcal{D}}_1,\tilde{\mathcal{D}}_2)$  is a minimizing sequence for the variational problem $\mathcal{J}_{\infty}^{\mu_1,\mu_2}$. Then, by the well-known variational property of $\mathcal{J}_{\infty}^{\mu_1,\mu_2}$ and  the uniqueness of the minimizer $(\mathcal{D}_{\infty}^{1},\mathcal{D}_{\infty}^{2})$ (see \cite{acs,lopes}), it follows that $(\tilde{\mathcal{D}}_1,\tilde{\mathcal{D}}_2)\rightarrow (\mathcal{D}_{\infty}^{1},\mathcal{D}_{\infty}^{2})$ in $H^1(\R)\times H^1(\R)$.
			Now notice that\begin{equation*}
				\begin{aligned}
					(H_1-L_0)\tilde{\mathcal{D}}_1&=(H_1-L_0)\langle \mathcal{D}_1(x^{\prime},x_n),\Psi_0(x^{\prime})\rangle_{L^2(\R^{n-1})}\\
					&\to (H_1-L_0)\langle \mathcal{D}_1^{\infty}(x_n),\Psi_0(x^{\prime})\rangle_{L^2(\R^{n-1})}\\&=\langle \mathcal{D}_1^{\infty}(x_n),(H_1-L_0)\Psi_0(x^{\prime})\rangle_{L^2(\R^{n-1})}=0.
				\end{aligned}
			\end{equation*} 
			Then  from   \eqref{theD1} and   \eqref{5elliptic}, we write the equation for $\tilde{\mathcal{D}}_1$ as
			$$
			\left(-\partial_{x_n}^2+\lambda_1\right) \tilde{\mathcal{D}}_1=\left\langle  \mathcal{D}_1(\cdot, x_n)\mathcal{D}_2(\cdot, x_n), \Psi_0(\cdot)\right\rangle_{L^2\left(\R^{n-1}\right)}.
			$$
			Furthermore, applying the same argument to $H_2-M_0$, we obtain that 
			\begin{equation}\label{convR}
				\left(-\kappa\partial_{x_n}^2+\lambda_2\right) \tilde{\mathcal{D}}_2=\frac{1}{2}\left\langle  \mathcal{D}_1(\cdot, x_n)^2, \Phi_0(\cdot)\right\rangle_{L^2\left(\R^{n-1}\right)}.  
			\end{equation}
			Then,   $\tilde{r}_1=\tilde{\mathcal{D}}_1-\mathcal{D}_{\infty}^{1}$ satisfies
			$$
			\begin{aligned}
				\mathcal{L}_{\infty}^1 \tilde{r}_1 &=-\partial_{x_n}^2\tilde{\mathcal{D}}_1+\lambda_{\infty}^1\tilde{\mathcal{D}}_1-2s_n^1\mathcal{D}_{\infty}^{2}\tilde{\mathcal{D}}_1-\left(-\partial_{x_n}^2\mathcal{D}_{\infty}^{1}+\lambda_{\infty}^1\mathcal{D}_{\infty}^{1}\right)+2s_n^1\mathcal{D}_{\infty}^{2}\mathcal{D}_{\infty}^{1}\\
				& = \left(\lambda_1-\lambda_{\infty}^1\right) \tilde{\mathcal{D}}_{1}-s_n^1\mathcal{D}_{\infty}^{1}\mathcal{D}_{\infty}^{2}-2s_n^1\mathcal{D}_{\infty}^{2}\tilde{\mathcal{D}}_1+2s_n^1\mathcal{D}_{\infty}^{2}\mathcal{D}_{\infty}^{1}+s_n^1\tilde{\mathcal{D}}_1\tilde{\mathcal{D}}_2\\
				&\qquad +\left\langle \mathcal{D}_1(\cdot, x_n)\mathcal{D}_2(\cdot, x_n), \Psi_0\right\rangle_{L^2\left(\R^{n-1}\right)}-s_n^1 \tilde{\mathcal{D}}_1\tilde{\mathcal{D}}_2\\
				& =\left(\lambda_1-\lambda_{\infty}^1\right) \tilde{\mathcal{D}}_{1}+s_n^1\mathcal{D}_{\infty}^{1}\mathcal{D}_{\infty}^{2}-2s_n^1\mathcal{D}_{\infty}^{2}\tilde{\mathcal{D}}_1+s_n^1\tilde{\mathcal{D}}_1\tilde{\mathcal{D}}_2\\
				& \qquad+\left\langle \mathcal{D}_1(\cdot, x_n)\mathcal{D}_2(\cdot, x_n), \Psi_0\right\rangle_{L^2\left(\R^{n-1}\right)}-s_n^1 \tilde{\mathcal{D}}_1\tilde{\mathcal{D}}_2.
			\end{aligned}
			$$
			Now, note from Lemma \ref{lemD1} that   $$\|\left(\lambda_1-\lambda_{\infty}^1\right) \tilde{\mathcal{D}}_{1}\|_{\lt}\leq |\lambda_1-\lambda_{\infty}^1|\| \tilde{\mathcal{D}}_{1}\|_{\lt}\leq O(\mu_1+\mu_2).$$ 
			Also, since $$s_n^1=\int_{\R^{n-1}}(\Psi_0(y))^2\Phi_0(y)\,dy,$$ it follows that $$s_n^1\tilde{\mathcal{D}}_1\tilde{\mathcal{D}}_2=\langle (\tilde{\mathcal{D}}_1\Psi_0)(\tilde{\mathcal{D}}_2\Phi_0),\Psi_0\rangle_{L^2\left(\R^{n-1}\right)}=\langle (P_0^1\mathcal{D}_1)(P_0^2\mathcal{D}_2),\Psi_0\rangle_{L^2\left(\R^{n-1}\right)}.$$
			Hence, 
			$$\begin{aligned}
				&\left\langle \mathcal{D}_1(\cdot, x_n)\mathcal{D}_2(\cdot, x_n), \Psi_0\right\rangle_{L^2\left(\R^{n-1}\right)}-s_n^1\tilde{\mathcal{D}}_1\tilde{\mathcal{D}}_2\\ &= \left\langle \mathcal{D}_1(\cdot, x_n)\mathcal{D}_2(\cdot, x_n), \Psi_0\right\rangle_{L^2\left(\R^{n-1}\right)}-\langle (P_0^1\mathcal{D}_1)(P_0^2\mathcal{D}_2),\Psi_0\rangle_{L^2\left(\R^{n-1}\right)}\\&
				= \left\langle \mathcal{D}_1\mathcal{D}_2-((1-P_1^1)\mathcal{D}_1)((1-P_1^2)\mathcal{D}_2), \Psi_0\right\rangle_{L^2\left(\R^{n-1}\right)}\\
				&=\left\langle \mathcal{D}_1\mathcal{D}_2-\mathcal{D}_1\mathcal{D}_2+\mathcal{D}_1(P_1^2\mathcal{D}_2)-(P_1^1\mathcal{D}_1)\mathcal{D}_2-(P_1^1\mathcal{D}_1)(P_1^2\mathcal{D}_2), \Psi_0\right\rangle_{L^2\left(\R^{n-1}\right)}
			\end{aligned}$$
			Then,   Lemma \ref{lemD1} with the uniform bound  yields
			$$
			\begin{aligned}
				&\left\|\left\langle \mathcal{D}_1(\cdot, x_n)\mathcal{D}_2(\cdot, x_n), \Psi_0\right\rangle_{L^2\left(\R^{n-1}\right)}-s_n^1 \tilde{\mathcal{D}}_1\tilde{\mathcal{D}}_2\right\|_{L^2(\R)}\\& \leq C\left(\left\|P_1^1 \mathcal{D}_1\right\|_{L^2\left(\rn\right)} +\left\|P_1^2 \mathcal{D}_2\right\|_{L^2\left(\rn\right)} \right)\leq O(\mu_1+\mu_2). 
			\end{aligned}$$
			Finally, notice that $$\begin{aligned}
				\left\|s_n^1\mathcal{D}_{\infty}^{1}\mathcal{D}_{\infty}^{2}-2s_n^1\mathcal{D}_{\infty}^{2}\tilde{\mathcal{D}}_1+s_n^1\tilde{\mathcal{D}}_1\tilde{\mathcal{D}}_2\right\|_{L^2(\R)}&=\left\|s_n^1\tilde{r}_1(\tilde{\mathcal{D}}_2-2\mathcal{D}_{\infty}^{2})+s_n^1\tilde{r}_2\mathcal{D}_{\infty}^{1}\right\|_{L^2(\R)}\\
				&\leq C\left(\left\|\tilde{r}_1\right\|_{H^1(\R)}+\left\|\tilde{r}_2\right\|_{H^1(\R)}\right)
			\end{aligned}$$
			Collecting   all above eqtimates, we obtain from Lemma \ref{lemD1} that \[\left\|\mathcal{L}_{\infty}^1 \tilde{r}_1\right\|_{L^2(\R)} \leq O(\mu_1+\mu_2).\] 
			Similarly, taking $\tilde{r}_2=\tilde{\mathcal{D}}_2-\mathcal{D}_{\infty}^{2}$ we have from \eqref{convR} that $$
			\begin{aligned}
				\mathcal{L}_{\infty}^2 \tilde{r}_2 &=-\kappa\partial_{x_n}^2\tilde{\mathcal{D}}_2+\lambda_{\infty}^2\tilde{\mathcal{D}}_2-s_n^2\mathcal{D}_{\infty}^{1}\tilde{\mathcal{D}}_2-\left(-\kappa\partial_{x_n}^2\mathcal{D}_{\infty}^{2}+\lambda_{\infty}^2\mathcal{D}_{\infty}^{2}\right)+s_n^2\mathcal{D}_{\infty}^{1}\mathcal{D}_{\infty}^{2}\\
				&= \left(\lambda_2-\lambda_{\infty}^2\right) \tilde{\mathcal{D}}_{2}-\frac{s_n^2}{2}(\mathcal{D}_{\infty}^{1})^2-2s_n^2\mathcal{D}_{\infty}^{1}\tilde{\mathcal{D}}_2+2s_n^2\mathcal{D}_{\infty}^{2}\mathcal{D}_{\infty}^{1}+\frac{s_n^2}{2}(\tilde{\mathcal{D}}_1)^2\\
				&\qquad +\frac{1}{2}\left\langle \mathcal{D}_1(\cdot, x_n)^2, \Phi_0\right\rangle_{L^2\left(\R^{n-1}\right)}-\frac{s_n^2}{2} \tilde{\mathcal{D}}_1^2.
			\end{aligned}
			$$
			Now, since  $$s_n^2=\int_{\R^{n-1}}(\Phi(y))^3\,dy,$$ we have 
			$$\frac{s_n^2}{2} \tilde{\mathcal{D}}_1^2=\frac{1}{2}\left\langle \tilde{\mathcal{D}}_1^2 \Phi_0^2, \Phi_0\right\rangle_{L^2\left(\R^{n-1}\right)}=\frac{1}{2}\left\langle (P_0^1\mathcal{D}_1)^2, \Phi_0\right\rangle_{L^2\left(\R^{n-1}\right)}.$$
			Then, $$\begin{aligned}
				\frac{1}{2}\left\langle \mathcal{D}_1(\cdot, x_n)^2, \Phi_0\right\rangle_{L^2\left(\R^{n-1}\right)}-\frac{s_n^2}{2} \tilde{\mathcal{D}}_1^2&=\frac{1}{2}\langle \mathcal{D}_1^2-(P_0^1\mathcal{D}_1)^2,\Phi_{0}\rangle_{L^2\left(\R^{n-1}\right)}\\&\leq C\langle \mathcal{D}_1-P_0^1\mathcal{D}_1,\Phi_{0}\rangle_{L^2\left(\R^{n-1}\right)}   \\
				&=C\langle P_1^1\mathcal{D}_1,\Phi_{0}\rangle_{L^2\left(\R^{n-1}\right)}.
			\end{aligned}$$
			Therefore, $$\left\| \frac{1}{2}\left\langle \mathcal{D}_1(\cdot, x_n)^2, \Phi_0\right\rangle_{L^2\left(\R^{n-1}\right)}-\frac{s_n^2}{2} \tilde{\mathcal{D}}_1^2 \right\|_{L^2(\R)}\leq C \left\| P_1^1\mathcal{D}_1 \right\|_{L^2(\R^{n-1})}.$$
			Moreover, $$\begin{aligned}
				\left\|-\frac{s_n^2}{2}(\mathcal{D}_{\infty}^{1})^2-s_n^2\mathcal{D}_{\infty}^{1}\tilde{\mathcal{D}}_2+s_n^2\mathcal{D}_{\infty}^{2}\mathcal{D}_{\infty}^{1}+\frac{s_n^2}{2}(\tilde{\mathcal{D}}_1)^2\right\|_{L^2(\R)}&\leq C \left\|\mathcal{D}_{\infty}^{1}-\tilde{\mathcal{D}}_{1}\right\|_{L^2(\R)}+C \left\|\mathcal{D}_{\infty}^{2}-\tilde{\mathcal{D}}_{2}\right\|_{L^2(\R)}\\
				&\leq C\left(\left\|\tilde{r}_1\right\|_{H^1(\R)}+\left\|\tilde{r}_2\right\|_{H^1(\R)}\right).
			\end{aligned}$$
			Combining all the above and proceed as in the case of the operator $\mathcal{L}_{\infty}^1$, we have   
			\[
			\left\|\mathcal{L}_{\infty}^2 \tilde{r}_2\right\|_{L^2(\R)} \leq O(\mu_1+\mu_2).
			\]
			Finally, by applying   Lemma \ref{nodeg}, we complete the proof.
		\end{proof}
		\subsection{Mountain pass solution}
		
		In this section, we want to address the existence of a solution for the elliptic problem \eqref{elliptic} from another perspective, presented in Theorem \ref{4theorem}. That is, we aim to demonstrate that there exists a  normalized solution which is also a mountain-pass solution, for $N>0$ sufficiently small. Consider the system,
		\begin{equation}\label{2elliptic}
			\left\{\begin{array}{l}
				-\Delta u+\lambda_1 u+V(x)u=uv, \\
				-\kappa \Delta v+ \lambda_2 v+V(x)v=\frac{1}{2}u^{2}.
			\end{array}\right.   
		\end{equation} 
  \blue{subject to the condition $$  \int_{\rn}|u |^{2} d x+ \kappa\int_{\rn}|v|^{2} d x=N^2.$$}

			

			\vspace{3mm}	
Recall that $\lam_N:=\lam_1=\lam_2/\kappa$.			Before proving this result, we need to establish some tools that will be of great utility. In fact, let \((u_{N}, v_{N})\) be a solution of \eqref{2elliptic}.
			The Pohozaev identity   (see Lemma \ref{pointcrit})  
			$$
			\begin{aligned}
				0&=\frac{n-2}{2} \int_{\rn}|\nabla u_{N}|^{2}+\kappa|\nabla v_{N}|^{2} d x+\frac{n+2}{2} \int_{\rn}V(x)\left(|u_{N}|^{2}+|v_{N}|^{2}\right) d x -\frac{n}{2}K(u_{N},v_{N})\\
				&\qquad+\frac{n}{2}\lambda_{1}\int_{\rn}|u_{N}|^{2}\,d x+\frac{n}{2} \lambda_{2} \int_{\rn}| v_{N}|^{2}\,d x,
			\end{aligned}
			$$
			together with \eqref{2elliptic} imply that
			$$
			\begin{aligned}
				& \lambda_1 \int_{\rn}|u_{N}|^{2} \,d x=K(u_{N},v_{N})-\int_{\rn}\left(|\nabla u_{N}|^{2}+V(x)|u_{N}|^{2}\right) d x, \\
				& \lambda_2 \int_{\rn}|v_{N}|^{2} d x=\frac{1}{2}K(u_{N},v_{N})-\int_{\rn}\left(\kappa|\nabla v_{N}|^{2}+V(x)|v_{N}|^{2}\right) d x.
			\end{aligned}
			$$
			Combining the equations from above, we have
			\begin{equation}\label{normal}
				\begin{aligned}
				 \blue{\lambda_N\left(\int_{\rn}|u_{N}|^{2} d x+ 	\kappa \int_{\rn}|v_{N}|^{2} d x\right)}  &=\lambda_1  \int_{\rn}|u_{N}|^{2} d x+ 	\lambda_2 \int_{\rn}|v_{N}|^{2} d x\\
      &=\frac{n+6}{4}K(u_{N},v_{N})-2\int_{\rn}V(x)\left(|u_{N}|^{2}+|v_{N}|^{2}\right)\,d x  .  
				\end{aligned} 
			\end{equation}
			We define
			\begin{equation}\label{def}
				\begin{aligned}
					&w_{N}^1(x)=\lambda_N^{-1} \kappa ^{-\frac12} u_{N}\left(\lambda_{N}^{-\frac{1}{2}} x\right) \quad \text { and } \quad t_{N}=\lambda_{N},\\
					&w_{N}^{2}(x)=\lambda_{N}^{-1} v_{N}\left(\lambda_{N}^{-\frac{1}{2}} x\right) \quad \text { and } \quad t_{N}=\lambda_{N}.
				\end{aligned}    
			\end{equation}
			Then $\left((w_{N}^1,w_{N}^2), t_{N}\right)$ is a solution of the following system:
			\begin{equation}\label{systemloc}
				\left\{\begin{array}{l}
					-\Delta w_1+w_1+t_{N}^{-2} V(x) w_1=w_1 w_2,  \\\\
					- \Delta w_2+w_2+t_{N}^{-2} V(x) w_2=\frac{1}{2}w_1^2,\\\\
					N^{2}=t_{N}^{\frac{4-n}{2}}\left(\frac{n+6}{4}K(w_1,w_2)-2t^{-2}\int_{\rn}V(x)\left(|w_{1}|^{2}+|w_{2}|^{2}\right)\,d x\right).
				\end{array}\right.   
			\end{equation}
			To find a normalized solution of \eqref{2elliptic}, it is equivalent to studying the existence of solutions to \eqref{systemloc}. For this purpose, let us first consider the following equation:
			\begin{equation}\label{2systemloc}
				\left\{\begin{array}{l}
					-\Delta w_1+w_1+t_{N}^{-2} V(x) w_1=w_1 w_2 , \\\\
					- \Delta w_2+w_2+t_{N}^{-2} V(x) w_2=\frac{1}{2}w_1^2.
				\end{array}\right.   
			\end{equation}
			The corresponding functional of \eqref{2systemloc} is given by
			$$
			\begin{aligned}
				\mathcal{J}_t(w_1,w_2) &=\frac{1}{2}\left(\|\nabla w_1\|_{\lt}^2+ \|\nabla w_2\|_{\lt}^2+\|w_1\|_{\lt}^2+\|w_2\|_{\lt}^2\right.\\ &\qquad\quad\left.+t_{N}^{-2}\int_{\rn}  V(x) (|w_1|^2 + |w_2|^2) \,d x -K(w_1,w_2)\right).    
			\end{aligned}
			$$
			We know that this functional is well defined and of class $C^2$ in $H$.
			We also define the usual Nehari manifold of $\mathcal{J}_t(w)$ as follows:
			$$
			\mathcal{M}_t=\left\{(w_1,w_2)\in H \backslash\{(0,0)\} \mid \langle \mathcal{J}_t^{\prime}(w_1,w_2) ,(w_1,w_2)\rangle=0\right\}.
			$$
			Notice that, for $(w_1,w_2) \in \mathcal{M}_t$, we have 
			$$\mathcal{J}_{t}(w_1,w_2)=\mathcal{J}_{t}(w_1,w_2)-\frac{1}{3}\langle \mathcal{J}_t^{\prime}(w_1,w_2) ,(w_1,w_2)\rangle = \frac{1}{6}  \norm{(w_1, w_2)}_H^2,$$
			whence, we can define 
			$$
			\mathfrak{m}(t)=\inf _{(w_1,w_2) \in \mathcal{M}_t} \mathcal{J}_t(w_1,w_2).
			$$
			Notice that if $(w_1,w_2)$ is a nontrivial solution of \eqref{2systemloc} with $\mathcal{J}_t(w_1,w_2)=\mathfrak{m}(t)$,
			then  $(w_1,w_2)$ is a ground state of \eqref{2systemloc}.

			On the other hand,  consider the following system:
			\begin{equation}\label{3systemloc}
				\left\{\begin{array}{l}
					-\Delta w_1+t_{N}w_1+ V(x) w_1=w_1 w_2 , \\\\
					- \Delta w_2+t_{N}w_2+ V(x) w_2=\frac{1}{2}w_1^2.
				\end{array}\right.   
			\end{equation}
			The corresponding functional of \eqref{3systemloc} is given by
			\[
			\begin{split}
				\mathcal{I}_t(w_1,w_2)&=\frac{1}{2}\left(\|\nabla w_1\|_{\lt}^2+ \|\nabla w_2\|_{\lt}^2+t_{N}\paar{\|w_1\|_{\lt}^2+\|w_2\|_{\lt}^2}\right.\\&
				\left.\qquad\quad+\int_{\rn}  V(x) (|w_1|^2 + |w_2|^2) \,d x  -K(w_1,w_2)\right) .    
			\end{split}
			\]
			This functional is well defined and of class $C^2$ in   $H$. Similarly, as mentioned above, we define the  manifold  
			$$
			\mathcal{P}_t=\left\{(w_1,w_2) \in H \backslash\{(0,0)\} : \langle \mathcal{I}_t^{\prime}(w_1,w_2),(w_1,w_2)\rangle =0\right\}.
			$$
			Let
			$$
			\mathbb{M}(t)=\inf _{(w_1,w_2) \in \mathcal{P}_t} \mathcal{I}_t(w_1,w_2).
			$$
			\begin{definition}
				If $(w_1,w_2)$ is a nontrivial solution of \eqref{3systemloc} with $\mathcal{I}_t(w_1,w_2)=\mathbb{M}(t)$, then $(w_1,w_2)$ is a ground state of \eqref{3systemloc}.
			\end{definition} 
			Next, we establish a standard result associated with our elliptic problem \eqref{2systemloc}.
			\begin{proposition}\label{prp4}
				Let $n=4,5$. Then \eqref{2systemloc} has a positive ground state $(w_1^t,w_2^t)$ for all $t>0$ satisfying $C_1 t^{\frac{6-n}{2}} \leq \left\|(w_1^t,w_2^t)\right\|_{\lt}^2 \leq C_2 t^{\frac{6-n}{2}}$ as $t \rightarrow 0$ and $(w_1^t,w_2^t) \rightarrow (w_{\infty}^1, w_{\infty}^2)$ strongly in $H^1\left(\rn\right)\times H^1\left(\rn\right)$ as $t \rightarrow$ $+\infty$, where $(w_{\infty}^1, w_{\infty}^2)$ is the unique (up to translations) positive solution of the following equation:
				\begin{equation}\label{systemsim}
					\left\{\begin{array}{l}
						-\Delta w_1+w_1=w_1w_2  \\
						- \Delta w_2+w_2=\frac{1}{2}w_1^2.
					\end{array}\right.    
				\end{equation}
				Moreover, $(w_1^t,w_2^t) $ is unique for $t>0$ sufficiently large.   
			\end{proposition}
			\begin{proof}
				Initially, we will ensure the existence of a ground state for the system \eqref{2systemloc}. Applying the strategy used before, we know  for any $t>0$ that there exists a minimizing sequence $\left\{(w_1^{n,t},w_2^{n,t})\right\}$ on the Nehari manifold $\mathcal{M}_t$ such that $(w_1^{n,t},w_2^{n,t})$ is real and positive. 
				Since $$\mathcal{J}_{t}(w_1,w_2)=\mathcal{J}_{t}(w_1,w_2)-\frac{1}{3}\langle \mathcal{J}_t^{\prime}(w_1,w_2) ,(w_1,w_2)\rangle = \frac{1}{6}  \norm{(w_1, w_2)}_H^2,$$ it follows  that $\mathfrak{m}(t)>0$. Moreover, we know that  there exists $\left\{z_m\right\} \in \mathbb{R}$ such that
				$$
				(w_1^{n,t},w_2^{n,t})\left(x^{\prime}, x_n-z_m\right) \rightharpoonup  (w_1^{t},w_2^{t}) \neq (0,0) \quad \text {  in } H \text { as } m \rightarrow \infty.
				$$
				Following the strategy used in Theorem \ref{2theorem}, we can demonstrate that indeed for any $t>0$, we have that $(w_1^{t},w_2^{t})$ is a ground state for the system \eqref{2systemloc}.
				
				We next prove  $(w_1^t,w_2^t) \rightarrow (w_{\infty}^1, w_{\infty}^2)$ for $t \rightarrow+\infty$. Let $(w_1^{t},w_2^{t})$ be a positive ground state of \eqref{2systemloc} for $t>0$. Indeed, we know that 
				$$\Delta w_1^t-w_1^t+w_1^tw_2^t=t^{-2}V(x)w_1^t\geq 0$$
				and $$\Delta w_2^t-w_2^t+\frac{1}{2}(w_1^t)^2=t^{-2}V(x)w_1^t\geq 0,$$
				then, for $t\geq 1$ (using arguments in Lemma \ref{regularity}) we have that $j=1,2$,   $w_j^t \in C^{2}(\rn)$, and \begin{equation}\label{decay}
					\sum_{k=1}^{2}\left|D^{\beta} w_{k}^t(x)\right|\stackrel{|x| \rightarrow \infty}{\longrightarrow } 0     
				\end{equation} 
    for all $|\beta| \leq 2$.
				Now, since $V(x) \geqslant 0, t>0$ and $w_1^t$, $w_2^t$ are positive, we know that 
				$$\Delta w_1^t-w_1^t+w_1^tw_2^t=t^{-2}V(x)w_1^t\geq 0$$
				and $$\Delta w_2^t-w_2^t+\frac{1}{2}(w_1^t)^2=t^{-2}V(x)w_1^t\geq 0,$$
				hence
				\begin{equation}\label{decay1}
					-\Delta w_1^t+w_1^t \leqslant w_1^tw_2^t\quad \text { in } \rn     
				\end{equation}
				and \begin{equation}\label{decay2}
					-\Delta w_2^t+w_2^t \leqslant \frac{1}{2}(w_1^t)^2\quad \text { in } \rn.   
				\end{equation}
				Moreover, since  $\mathfrak{m}(t)$ is decreasing with respect with $t>0$ (see Lemma 3.2 in \cite{wei}). Thus, $$\frac{1}{6}  \norm{(w_1^t, w_2^t)}_H^2=\mathcal{J}_{t}(w_1^t,w_2^t)-\frac{1}{3}\langle \mathcal{J}_t^{\prime}(w_1^t,w_2^t) ,(w_1,w_2)\rangle = \mathcal{J}_{t}(w_1^t,w_2^t)=m(t)\leq C,$$ so, $\left\{(w_1^t,w_2^t)\right\}$ is bounded in $H^1\left(\rn\right)\times H^1\left(\rn\right) $.  Let $\varepsilon>0$ and $\theta_{\varepsilon}(x)=\exp \left(\frac{|x|}{1+\varepsilon|x|}\right)$. Then $\theta_{\varepsilon}$ is bounded, Lipschitz continuous and satisfies $\left|\nabla \theta_{\varepsilon}\right| \leq \theta_{\varepsilon}$ w.r.t in $\rn$.
				Multiplying the equation \eqref{decay1} by $\theta_{\varepsilon} \bar{w_1^t} \in H^1(\rn)$, it follows that
				$$
				\int _{\rn}\Re\left(\nabla w_1^t\,d x \cdot \nabla\left(\theta_{\varepsilon} \bar{w_1^t}\right)\right)\,d x+\int _{\rn}\theta_{\varepsilon}|w_1^t|^2\,d x\leq  \Re\int _{\rn}\theta_{\varepsilon}|w_1^t|^{2}w_2^t\,d x.
				$$
				Since $\Re\left(\nabla w_1^t \cdot \nabla\left(\theta_{\varepsilon} \bar{w_1^t}\right)\right) \geq \theta_{\varepsilon}|\nabla w_1^t|^2-\theta_{\varepsilon}|w_1^t||\nabla w_1^t|$, we get
				$$
				\int _{\rn}\theta_{\varepsilon}|\nabla w_1^t|^2\,d x-\int _{\rn}\theta_{\varepsilon}|w_1^t||\nabla w_1^t|\,d x+\int _{\rn}\theta_{\varepsilon}|w_1^t|^2 \,d x\leq \int _{\rn}\theta_{\varepsilon}|w_1^t|^{2}|w_2^t| \,d x
				$$
				Let
				\begin{equation}\label{delta}
					0<\delta<\frac{1}{4},   
				\end{equation}
				then it follows from \eqref{decay} that, for some $R_1$,
				$ 
				|w_2^t(x)|<\delta
				$ 
				provided that $|x| \geq R_1$. So, by Cauchy-Schwartz inequality, we get
				$$
				\begin{aligned}
					\frac{1}{2} \int _{\rn}\theta_{\varepsilon}|\nabla w_1^t|^2\,d x+ & \frac{1}{2} \int _{\rn}\theta_{\varepsilon}|w_1^t|^2\,d x \leq  \int_{\left\{|x| \leq R_1\right\}} e^{|x|}|w_1^t|^{2}|w_2^t|\,d x+\delta \int_{\left\{|x| \geq R_1\right\}} \theta_{\varepsilon}|w_1^t|^2\,d x.
				\end{aligned}
				$$
				From \eqref{delta}, we obtain
				$$
				\frac{1}{2} \int _{\rn}\theta_{\varepsilon}|\nabla w_1^t|^2\,d x+\frac{1}{4} \int _{\rn}\theta_{\varepsilon}|w_1^t|^2 \,d x\leq C,
				$$
				where $C$ is a positive constant independent of $\varepsilon$.
				Letting $\varepsilon \downarrow 0$ we obtain
				$$
				\frac{1}{4} \int _{\rn}e^{|x|}\left(|\nabla w_1^t|^2+|w_1^t|^2\right) \leq C .
				$$
				There exists from \eqref{decay} a positive constant $R_2$ such that
				$$
				|w_1^t(x)|+|\nabla w_1^t(x)|<1 \quad \text { if } \quad|x| \geq R_2
				$$
				and for $|x| \leq R_2$ we get
				\begin{equation}\label{1}
					e^{\frac{2|x|}{n+2}}\{|\nabla w_1^t(x)|+|w_1^t(x)|\} \leq e^{\frac{2 R_2}{N+2}}\|w_1^t\|_{W^{1, \infty}(\rn)}.   
				\end{equation}
				Let $x \in \rn$ such that $|x| \geq R_2$. Since $w_1^t$ and $\nabla w_1^t$ are globally Lipschitz continuous, there exists $L>0$ such that, for all $y \in \rn$, we have
				$$
				\left\{\begin{array}{c}
					|\nabla w_1^t(y)| \geq|\nabla w_1^t(x)|-\frac{L}{\sqrt{2}}|x-y|, \\
					|w_1^t(y)| \geq|w_1^t(x)|-\frac{L}{\sqrt{2}}|x-y|,
				\end{array}\right.
				$$
				from which we deduce that
				$$
				|w_1^t(x)|^2+|\nabla w_1^t(x)|^2 \leq 2\left(|w_1^t(y)|^2+|\nabla w_1^t(y)|^2+L^2|x-y|^2\right).
				$$
				Taking
				\begin{equation}\label{rho}
					\rho=\frac{1}{2 L}\left(|w_1^t(x)|^2+|\nabla w_1^t(x)|^2\right)^{1 / 2}    ,
				\end{equation}
				we have
				\begin{equation}\label{est}
					|w_1^t(x)|^2+|\nabla w_1^t(x)|^2 \leq 4\left(|w_1^t(y)|^2+|\nabla w_1^t(y)|^2\right) \quad \text{for any} \, \,y \in B_\rho(x).   
				\end{equation}
				Integrating \eqref{est} on $B_\rho(x)$ we obtain
				$$
				C_n \rho^n\left(|w_1^t(x)|^2+|\nabla w_1^t(x)|^2\right) \leq 4 \int_{B_\rho(x)}\left(|w_1^t(y)|^2+|\nabla w_1^t(y)|^2\right) d y
				$$
				and from \eqref{rho} we get
				$$
				C_n\left(|w_1^t(x)|^2+|\nabla w_1^t(x)|^2\right)^{\frac{n+2}{2}} \leq 4 \int_{B_\rho(x)}\left(|w_1^t(y)|^2+|\nabla w_1^t(y)|^2\right) d y .
				$$
				Since $\rho \leq 1 / 2 L$ if $|x| \geq R_2$, it follows that
				$$
				|y|-|x|+1 / 2 L \geq 0 \quad \forall y \in B_\rho(x)
				$$
				and so, we conclude that
				\begin{equation}\label{2}
					C_n (1+|x|)^{\frac{n+2}{2}}e^{|x|}\left(|w_1^t(x)|^2+|\nabla w_1^t(x)|^2\right)^{\frac{n+2}{2}} \leq C, 
				\end{equation}
				then,     it follows from \eqref{1} and \eqref{2} that there exists $a<1$ such that 
				\begin{equation}\label{eq1}
					|w_1^t| \leq C(1+|x|)^{-1} e^{-\frac{1}{2}|x|} \quad \text { in } \rn \text { for } t \geqslant 1.    
				\end{equation}
				Thus, 
				$$
				\int_{\rn} V(x) |w_1^t|^2  \leq  C \quad \text { for all } t \geqslant 1,
				$$
				which implies that
				\begin{equation}\label{V1}
					t^{-2} \int_{\rn} V(x) |w_1^t|^2 =o_t(1) \quad \text { as } t \rightarrow+\infty.    
				\end{equation}
				Similarly, we can also obtain 
				\begin{equation}\label{V2}
					t^{-2} \int_{\rn} V(x) |w_2^t|^2 =o_t(1) \quad \text { as } t \rightarrow+\infty.    
				\end{equation}
				Now, since $$
				\begin{aligned}
					\mathcal{J}_t(w_1^t,w_2^t) &=\frac{1}{2}\Big(\|\nabla w_1^t\|_{\lt}^2+ \|\nabla w_2^t\|_{\lt}^2+\|w_1^t\|_{\lt}^2+\|w_2^t\|_{\lt}^2\\
					&\qquad+t^{-2}\int_{\rn}  V(x) (|w_1^t|^2 + |w_2^t|^2) \,d x -\Re\int_{\rn}(w_{1}^t)^{2}w_{2}^t\,d x\Big)\\
					&= \frac{1}{2}\Big(\|\nabla w_1^t\|_{\lt}^2+ \|\nabla w_2^t\|_{\lt}^2+\|w_1^t\|_{\lt}^2+\|w_2^t\|_{\lt}^2\\
					&\qquad-K(w_1^t,w_2^t)\Big)+o_t(1)\\
					&\geq m+o_t(1) 
				\end{aligned}
				$$ as $t \rightarrow+\infty$, where
				$$
				\mathfrak{m}=\inf _{(w_1,w_2) \in \mathcal{M}} \mathcal{J}(w_1,w_2)
				$$
				with
				$$
				\mathcal{J}(w_1,w_2)=\frac{1}{2}\left(\|\nabla w_1\|_{\lt}^2+ \|\nabla w_2\|_{\lt}^2+\|w_1\|_{\lt}^2+\|w_2\|_{\lt}^2-K(w_1,w_2)\right)
				$$
				and
				$$
				\mathcal{M}=\left\{(w_1,w_2) \in H^1\left(\rn\right) \backslash\{(0,0)\} : \langle\mathcal{J}^{\prime}(w_1,w_2),(w_1,w_2)\rangle=0\right\} .
				$$
				Therefore, $$m(t)\geq m+o_t(1).$$
				On the other hand, it is well known that  \eqref{systemsim}    has a unique (up to translations) positive radial solution $(w_{\infty}^1,w_{\infty}^2)$, which exponentially decays to zero at infinity (see \cite{acs,chengt,lopes,ZhaoZhaoShi}). 
				Moreover, since for any $(w_1,w_2) \in \mathcal{M}$, $$\mathcal{J}_t(w_1,w_2)=\mathcal{J}(w_1,w_2)+o_t(1),$$ then $\mathfrak{m}(t) \leqslant \mathfrak{m}+o_t(1)$ as $t \rightarrow+\infty$. 
				It follows that $\mathfrak{m}(t)=\mathfrak{m}+$ $o_t(1)$ as $t \rightarrow+\infty$. Since $\{(w_1^t,w_2^t)\}$ is bounded in $H^1(\rn)\times H^1(\rn) $, 
				we have
				$$
				(w_1^t,w_2^t) \rightharpoonup  (w_1,w_2) \quad \text{as}\, \,  t \rightarrow+\infty,$$ then $$\mathcal{J}_t(w_1^t,w_2^t)  \rightarrow  \mathcal{J}_t(w_1,w_2)\quad \text{as}\, \,  t \rightarrow+\infty.
				$$ 
				Since $m(t) \rightarrow m$ as $t  \rightarrow \infty$,  it follows by uniqueness of $(w_{\infty}^1,w_{\infty}^2)$ that $$(w_1,w_2)=(w_{\infty}^1,w_{\infty}^2).$$ Now, once again, as $m(t)= m+o_t(1)$, $$\langle\mathcal{J}_{t}^{\prime}(w_1^t,w_2^t),(w_1^t,w_2^t)\rangle=0$$ 
				and 
				$$
				\langle\mathcal{J}^{\prime}(w_{\infty}^1,w_{\infty}^2),(w_{\infty}^1,w_{\infty}^2)\rangle=0,$$ we have \begin{equation}\label{conh1}
					\left\|(w_1^t,w_2^t)\right\|_{H^1(\rn) \times H^1(\rn)}^2=\left\|(w_{\infty}^1,w_{\infty}^2)\right\|_{H^1(\rn) \times H^1(\rn)}^2+o_t(1)    
				\end{equation} and \begin{equation}\label{convh2}
					K(w_1^t,w_2^t)=K(w_{\infty}^1,w_{\infty})+o_t(1), \, \, \text{as}\,\, t \rightarrow+\infty.   
				\end{equation}
				Therefore, using the weak convergence, $(w_1^t,w_2^t) \rightharpoonup  (w_{\infty}^1,w_{\infty}^2)$, it follows that \begin{equation*}
					(w_1^t,w_2^t) \rightarrow w_{\infty} \, \, \text{strongly in}\, \, H^1(\rn) \times H^1(\rn) \, \, \text{as}\,\, t \rightarrow+\infty.   
				\end{equation*}
				We now turn to the proof of the convergent conclusion for $t \rightarrow 0$. For every $t>0$, let $(w_1^t,w_2^t)$ be a positive ground state of \eqref{2systemloc}, then, $(u_1^t,u_2^t)$ is a positive solution of \eqref{3systemloc}. Moreover, by direct calculations,
				$$
				\mathcal{J}_t(w_1^t,w_2^t)=t^{\frac{6-n}{2}} \mathcal{I}_t(u_1^t,u_2^t)\quad \text { and } \quad \langle \mathcal{J}_t^{\prime}(w_1^t,w_2^t), (w_1^t,w_2^t) \rangle =t^{\frac{6-n}{2}} \langle \mathcal{I}_t^{\prime}(u_1^t,u_2^t), (u_1^t,u_2^t)\rangle .
				$$
				Thus, $(u_1^t,u_2^t)$  is a positive ground state of \eqref{3systemloc} for all $t>0$.
				Thus, by using similar arguments as those applied previously, we have $\mathbb{M}(0)>0$ and  \begin{equation}\label{a1}
					\mathbb{M}(t)=\mathbb{M}(0)+o_t(1),
				\end{equation} as $t \rightarrow 0$. Since $(u_1^t,u_2^t)$  is a positive ground state of \eqref{3systemloc} for all $t>0$, then $$\|(u_1^t,u_2^t)\|_{H}^2\leq C$$ and $$\frac{1}{6}  \norm{(u_1^t, u_2^t)}_H^2=\mathcal{I}_{t}(u_1^t,u_2^t)-\frac{1}{3}\langle \mathcal{I}_t^{\prime}(u_1^t,u_2^t) ,(u_1^t,u_2^t)\rangle = \mathcal{I}_{t}(u_1^t,u_2^t)=M(t)=M(0)+o_t(1).$$
				Hence, \begin{equation}\label{a2}
					C_1\leq \|(u_1^t,u_2^t)\|_{H}^2\leq C_2,   \, \, \text{as} \, \, t \rightarrow 0.
				\end{equation} 
				Now, it follows from $$\langle \mathcal{I}_t^{\prime}(u_1^t,u_2^t) ,(u_1^t,u_2^t)\rangle=0,$$ that $$ C_2\|(u_1^t,u_2^t)\|_{H}^2 \leq  K(u_1^t,u_2^t) \leq C_3\|(u_1^t,u_2^t)\|_{H}^2,$$
				this is, \begin{equation}\label{a3}
					C_4 \leq K(u_1^t,u_2^t)\leq C_5.  
				\end{equation} 
				Therefore, combining \eqref{a1}-\eqref{a3}, we have $$C_6\leq \|(u_1^t,u_2^t)\|_{\lt}^2\leq C_7.$$
				Consequently, $$C_6t^{\frac{6-n}{2}}\leq \|(w_1^t,w_2^t)\|_{\lt}^2\leq C_7t^{\frac{6-n}{2}}.$$ We finish the proof by showing the uniqueness of $(w_1^t,w_2^t)$ for $t>0$ sufficiently large. Let $(w_1^t,w_2^t)$ and $(\tilde{w}_1^t,\tilde{w}_2^t)$ be two different positive ground states of \eqref{2systemloc} and we define for $j=1,2$,  $\phi_j^t=\frac{w_j^t-\tilde{w}_j^t}{\left\|w_j^t-\tilde{w}_j^t\right\|_{L^{\infty}\left(\rn\right)}}$. Then,
				\begin{equation*}
					\begin{aligned}
						-\Delta \phi_1^t+\phi_1^t+t^{-2} V(x) \phi_1^t &=\frac{\left(w_1^tw_2^t-\tilde{w}_1^t\tilde{w}_2^t\right)}{\left\|w_1^t-\tilde{w}_1^t\right\|_{L^{\infty}\left(\rn\right)}}\\&=\frac{\left(w_1^tw_2^t+w_1^t\tilde{w}_2^t-w_1^t\tilde{w}_2^t-\tilde{w}_1^t\tilde{w}_2^t\right)}{\left\|w_1^t-\tilde{w}_1^t\right\|_{L^{\infty}\left(\rn\right)}}\\
						&=w_1^t\phi_2^t+\tilde{w}_2^t \phi_1^t.
					\end{aligned}     
				\end{equation*}
				That is \begin{equation}\label{equa}
					\begin{aligned}
						-\Delta \phi_1^t+\phi_1^t+t^{-2} V(x) \phi_1^t =w_1^t\phi_2^t+\tilde{w}_2^t \phi_1^t.
					\end{aligned}     
				\end{equation}  Similarly,
				\begin{equation}\label{equa2}
					\begin{aligned}
						-\Delta \phi_2^t+\phi_2^t+t^{-2} V(x) \phi_2^t =\frac{1}{2}(w_1^t+\tilde{w}_1^t)\phi_1^t.
					\end{aligned}     
				\end{equation} 
				Since $V(x) \geqslant 0$, by \eqref{eq1},
				$$
				-\Delta\left(\phi_1^t\right)^2+\frac{3}{2}\left(\phi_1^t\right)^2 \leqslant 0, \quad \text { in } \mathbb{R}^3 .
				$$
				Thus, by the maximum principle, $\left|\phi_1^t\right| \leq Ce^{-\frac{1}{2}|x|}$ for $|x| \geqslant 1$. Now, taking the inner product in \eqref{equa} and \eqref{equa2} with $\phi_1^t$ and $\phi_2^t$ respectively, we have that for $j=1,2$, $\phi_j^t \rightarrow \phi$ strongly in any compact sets as $t \rightarrow+\infty$ and
				$$
				-\Delta \phi_1+\phi_1=w_{\infty}^{1}\phi_2+w_{\infty}^{2}\phi_1,
				\qquad
				-\Delta \phi_2+\phi_2=w_{\infty}^1\phi_1.
				$$
				Note that $(w_1^t,w_2^t)$ and $(\tilde{w}_1^t,\tilde{w}_2^t)$ are radial w.r.t. $x^{\prime}$ for all $x_n$ and even w.r.t. $x_n$ for all $x^{\prime}$. Thus, $\phi_t$ is also radial w.r.t. $x^{\prime}$ for all $x_n$ and even w.r.t. $x_n$ for all $x^{\prime}$. Now, by the well-known nondegeneracy of $w_{\infty}^1$, we have $\phi_{1} \equiv 0$ and $\phi_{2} \equiv 0$ . It, together with $\left|\phi_1^t\right| \leq Ce^{-\frac{1}{2}|x|}$ for $|x| \geqslant 1$ implies that $\|\phi_1\|_{L^{\infty
					}(\rn)} \neq 1$ and $\|\phi_2\|_{L^{\infty
					}(\rn)} \neq 1$, which is a contradiction. Therefore, $(w_1^t,w_2^t)$ is
				unique for $t > 0$ sufficiently large.
			\end{proof}
			In the following, we will proceed to demonstrate Theorem \ref{4theorem}. Initially, let $(w_1^t,w_2^t)$ be a positive ground state of \eqref{2systemloc} given by Proposition \ref{prp4} and we define
			$$
			f(N, t):=N^2-t^{\frac{4-n}{2}}\left(\frac{n+6}{4}K(w_{1}^t, w_{2}^t)-2 t^{-2} \int_{\rn} V(x)( |w_1^t|^2+|w_2^t|^2)  \,d x\right).
			$$
			
			For every $t>0$ sufficiently large, by Proposition \ref{prp4},  there exists a unique
			\begin{equation}\label{N}
				N_t:=\left(t^{\frac{4-n}{2}}\left(\frac{n+6}{4}K(w_{1}^t, w_{2}^t)-2 t^{-2} \int_{\rn} V(x)( |w_1^t|^2+|w_2^t|^2)  \,d x\right)\right)^{\frac{1}{2}}>0    
			\end{equation}
			such that $f\left(\blue{N_t}, t\right)=0$. Therefore, from \eqref{normal} there exists $\left(u_{N}^t,v_{N}^t\right)$  a positive normalized solution of \eqref{systemloc} with a positive Lagrange multiplier $t>0$. Based on this, we proceed to the proof of Theorem \ref{4theorem}.
			
			\begin{proof}[Proof of   Theorem \ref{4theorem}.]
				Let $(w_1^t,w_2^t)$  be a positive ground state of \eqref{2systemloc} given by Proposition \ref{prp4}. Since $(w_1^t,w_2^t)$ is unique for $t>0$ sufficiently large,  we consider  $t>T_{\star}$. It is standard to show that $\int_{\mathbb{R}^3} V(x) w_t^2$ is continuous for $t>T_*$. Now from \eqref{convh2} in Proposition \ref{prp4}, we have 
				$$
				\frac{n+6}{4}K(w_{1}^t, w_{2}^t)-2 t^{-2} \int_{\rn} V(x)( |w_1^t|^2+|w_2^t|^2)  \,d x=\frac{n+6}{4}K(w_{\infty}^1,w_{\infty}^2)+o_t(1) .
				$$
				That is, the term on the left-hand side in the above equality is bounded for sufficiently large $t$. Therefore, 
				for every $$N<\left(T_{\star}^{\frac{4-n}{2}}\left(\frac{n+6}{4}K(w_{1}^{T_{\star}},w_{2}^{T_{\star}})-2 {T_{\star}}^{-2} \int_{\rn} V(x)( |w_1^{T_{\star}}|^2+|w_2^{T_{\star}}|^2)  \,d x\right)\right)^{\frac{1}{2}},$$  we have $$f(N,T_{\star})<0.$$ Now, by $n\geq 4$ and taking $t$ sufficiently large, say $t_1$, we have $f(N,t_1)>0.$
				Hence, from continuity of $f$ for $t>T_{\star}$, it follows that, there exists $t_N>T_{\star}$ such that $f(r, t_N)=0$. This  implies that \eqref{systemloc} has a second positive normalized solution $(u_{N,2},v_{N,2})$ with a positive Lagrange multiplier $\lambda_{N, 2}$. On the other hand,  from \eqref{normal}, \eqref{def} and  \eqref{N}, \begin{equation}\label{convlamd}
					\lambda_{ N,2}=(1+o_N(1))\left[\frac{n+6}{4N^2}K(w_{\infty}^1,w_{\infty}^2)\right]^{\frac{2}{n-4}}\, \,  \rightarrow +\infty \, \, \, \text{as}\, \, \, N \rightarrow 0.
				\end{equation} It remains to show that $(u_{N, 2},v_{N,2})$ is a mountain-pass solution of \eqref{systemloc} for $N>0$ sufficiently small. As that in  \cite[remark 1.10]{Bell}, we introduce the mountain-pass level
				$$
				\alpha(N)=\inf _{(g_1,g_2) \in \Gamma_N} \max _{t \in[0,1]} \mathcal{Y}(g_1[t],g_2[t]),
				$$
				where \begin{equation}
					\begin{aligned}
						\mathcal{Y}(u,v)&=\frac{1}{2}\Big(\|\nabla u\|_{\lt}^2+ \|\nabla v\|_{\lt}^2+\int_{\rn}V(x)(|u|^2+|v|^2)\,d x -K(u,v)\Big)
					\end{aligned}  
				\end{equation} and
				$$
				\Gamma_N=\sett{(g_1[s],g_2[s]) \in C\left([0,1], \mathcal{S}_N\right) : g_1[0]=u_{N, 1}, \, g_2[0]=v_{N, 1}, \, \, \mathcal{Y}(g_1[1],g_2[1])<\mathcal{Y}(g_1[0],g_2[0])}
				$$
 	with $(u_{N, 1},v_{N,1})$ being a local minimizer of $\mathcal{Y}(u,v)$ in $\mathcal{S}_N$, \blue{obtained in a similar way to the minimizers of the previous sections}, with
 \blue{	\[
 	\mathcal{S}_N=\left\{(u,v) \in H :\|u\|_{\lt}^2+\blue{\kappa}\|v\|_{\lt}^2=N^2\right\}.\]}
   	 Let
				$$
				B_{\rho, H, t}=\left\{(u,v) \in H :\|(u,v)\|_{H, t}^2 \leqslant \rho^2\right\},
				$$
				where $\|(u,v)\|_{H, t}$ is defined via
				$$
				\|(u,v)\|_{H, t}=\left(\|\nabla u\|_2^2+\|\nabla v\|_2^2+t^{-2} \int_{\rn} V(x)( |u|^2+|v|^2)\,d x\right)^{\frac{1}{2}}.
				$$
				Let us consider the set \( (u,v) \in B_{\rho, H, t} \) such that \( K(u,v) > 0 \). If this condition does not hold, we work with \( B_{\rho, H, t} \cap \mathcal{P} \), where \( \mathcal{P} = \left\{(u,v) \in H : K(u,v) > 0 \right\} \). Then, by \cite[Lemma 2.1]{Bell} and the Sobolev inequality, for a sufficiently small fixed \( \rho > 0 \), it can be shown, as in \cite{rab}, by standard methods, that
\[
\mathfrak{m}(t) = \inf _{(h_1,h_2) \in \Theta} \max _{t \in [0,1]} \mathcal{J}_t(h_1[t],h_2[t]),
\]
				where
				$$
				\Theta=\left\{(h_1[t],h_2[t]) \in C([0,1], H) : h_1[0], h_2[0] \in B_{\rho, H, t} \, \text { and }\, \mathcal{J}_t(h_1[1],h_2[1])<\frac{1}{4} \rho^2\right\}.
				$$
				Now, for every $(g_1[s],g_2[s]) \in \Gamma_N$, we define $$\tilde{g}_1[s]=\lambda_{N, 2}^{-1} g_1[s]\left(\lambda_{N, 2}^{-\frac{1}{2}} x\right) \, \, \text{and}\, \, \tilde{g}_2[s]=\lambda_{N, 2}^{-1} g_2[s]\left(\lambda_{N, 2}^{-\frac{1}{2}} x\right).$$ 
				Then, we have from \eqref{propd} that
				\begin{equation}\label{2lam}
					\mathcal{J}_{\lambda_{N, 2}}\left(\tilde{g}_1[s],\tilde{g}_2[s]\right)=\lambda_{N, 2}^{\frac{n-6}{2}}\left(\mathcal{Y}(g_1[s],g_2[s])+\frac{\lambda_{N, 2} N^2}{2}\right).    
				\end{equation}
				Moreover,   Theorem \ref{theorem1} and   \eqref{convlamd} reveal that
				$$
				\left\|(\tilde{g}_1[0],\tilde{g}_2[0])\right\|_{H, \lambda_{N, 2}}^2 \leq C\lambda_{N, 2}^{\frac{n-6}{2}}\left\|(u_{N,1},v_{N,1})\right\|_{H, \lambda_{N, 2}}^2\leq C\lambda_{N, 2}^{\frac{n-6}{2}} N^2 \rightarrow 0, \, \, \text{if} \, \, \, N \rightarrow 0.
				$$
				Thus, $(\tilde{g}_1[0],\tilde{g}_2[0]) \in B_{\rho, H, \lambda_{N, 2}}$. Similar to \eqref{2lam}, we have $$\begin{aligned}\mathcal{J}_t\left(\tilde{g}_1[0],\tilde{g}_2[0]\right)&=t^{\frac{n-6}{2}}\left(\mathcal{Y}(g_1[0],g_2[0])+\frac{\lambda_{N, 2} N^2}{2}\right) \\&\leq  t^{\frac{n-6}{2}}\left(C\left\|(u_{N,1},v_{N,1})\right\|_{H}^2+\frac{\lambda_{N, 2} N^2}{2}\right)\\
					&\leq Ct^{\frac{n-6}{2}}N^2\rightarrow 0, 
				\end{aligned} $$ for $N>0$ sufficiently small. By the definition of $(g_1[t],g_2[t])$, we also have $\mathcal{J}_t\left(\tilde{g}_1[1],\tilde{g}_2[1]\right)<\frac{1}{4} \rho^2$. It follows that $(\tilde{g}_1[t],\tilde{g}_2[t]) \in \Theta$, which implies from \eqref{2lam}
				\begin{equation}\label{primeine}
					\mathfrak{m}\left(\lambda_{N, 2}\right) \leqslant \lambda_{N, 2}^{\frac{n-6}{2}}\left(\alpha(N)+\frac{\lambda_{N, 2} N^2}{2}\right) .   
				\end{equation}
				We now show that the inequality above is indeed equality. In fact, we consider the fibering map of $\mathcal{Y}$  at $(u_{N,2},v_{N,2})$ given by
				$$
				\begin{aligned}
					\mathcal{T}_{u_{N, 2}}(\tau)&=\frac{\tau^2}{2}\left(\left\|\nabla u_{N, 2}\right\|_{\lt}^2+ \left\|\nabla v_{N, 2}\right\|_{\lt}^2\right)+\frac{1}{2 \tau^2} \int_{\rn} V(x) (|u_{N, 2}|^2+|v_{N, 2}|^2)\,d x\\&\quad-\frac{\tau^{n/2}}{2}K(u_{N,2},v_{N,2})   
				\end{aligned}
				$$
				By direct calculations,
				$$
				\begin{aligned}
					\mathcal{T}_{u_{N, 2}}^{\prime}(\tau)&=\tau\left(\left\|\nabla u_{N, 2}\right\|_{\lt}^2+ \left\|\nabla v_{N, 2}\right\|_{\lt}^2\right)-\frac{1}{\tau^3} \int_{\rn} V(x) (|u_{N, 2}|^2+|v_{N, 2}|^2)\,d x\\&\quad-\frac{n\tau^{\frac{n-2}{2}}}{4}K(u_{N,2},v_{N,2})  
				\end{aligned}
				$$
				and
				$$
				\begin{aligned}
					\mathcal{T}_{u_{N, 2}}^{\prime \prime}(\tau)&=\left\|\nabla u_{N, 2}\right\|_{\lt}^2+ \left\|\nabla v_{N, 2}\right\|_{\lt}^2+\frac{3}{\tau^4} \int_{\rn} V(x) (|u_{N, 2}|^2+|v_{N, 2}|^2)\,d x\\&\quad-\frac{n(n-2)\tau^{\frac{n-4}{2}}}{8}K(u_{N,2},v_{N,2}) .  
				\end{aligned}
				$$
				From \eqref{pointcrit}, we have $\mathcal{T}_{u_{N, 2}}^{\prime}(1)=0$. Moreover, by \eqref{conh1}, \eqref{V1}, \eqref{V2} and the Pohozaev identity of $(w_{\infty}^1,w_{\infty}^2)$, we have
				$$
				\begin{aligned}
					\mathcal{T}_{u_{N, 2}}^{\prime \prime}(1)&=\lambda_{N, 2}^{\frac{6-n}{2}}\left(\left\|\nabla w_1^N\right\|_{\lt}^2+ \left\|\nabla w_2^N\right\|_{\lt}^2-\frac{n(n-2)}{8}K(w_1^N,w_2^N) +o_N(1)\right) \\
					&=\lambda_{N, 2}^{\frac{6-n}{2}}\left(\frac{n(4-n)}{8}K(w_{\infty}^1,w_{\infty})+o_N(1)\right)\\
					&<0
				\end{aligned}
				$$
				for $r>0$ sufficiently small. Now, let $$h(\tau)=\tau^4\left(\left\|\nabla u_{N, 2}\right\|_{\lt}^2+ \left\|\nabla v_{N, 2}\right\|_{\lt}^2\right)-\frac{n\tau^{\frac{n+4}{2}}}{4}K(u_{N,2},v_{N,2}) $$ then,
				$$h^{\prime}(\tau)=0 \, \, \text{if and only if }\, \, \, \tau=\frac{32\left(\left\|\nabla w_1^N\right\|_{\lt}^2+ \left\|\nabla w_2^N\right\|_{\lt}^2\right)}{n(n+4)K(w_1^N,w_2^N) },$$
				hence 
				$$\begin{aligned}
					& \max _{\tau \geqslant 0} h(\tau)\\&=\frac{n-4}{n+4}\left[\frac{32\left(\left\|\nabla w_1^N\right\|_{\lt}^2+ \left\|\nabla w_2^N\right\|_{\lt}^2\right)}{n(n+4)K(w_1^N,w_2^N)  }\right]^{\frac{2}{n-4}} \left(\left\|\nabla w_1^N\right\|_{\lt}^2+ \left\|\nabla w_2^N\right\|_{\lt}^2\right)\\&>0.   
				\end{aligned}$$
				Since \( \mathcal{T}_{u_{N, 2}}'(1) = 0 \), we have
\[
h(1) = \int_{\mathbb{R}^n} V(x) \left( |u_{N, 2}|^2 + |v_{N, 2}|^2 \right) \,d x.
\]
Thus,
\[
\max_{\tau \geqslant 0} h(\tau) > \int_{\mathbb{R}^n} V(x) \left( |u_{N, 2}|^2 + |v_{N, 2}|^2 \right) \,d x.
\]
It follows that
\[
\max_{\tau \geq 0} \mathcal{T}_{u_{N, 2}}'(\tau) > 0.
\]
Therefore, there exists \( \tau_N < 1 \) such that \( \mathcal{T}_{u_{N, 2}}'(\tau_N) = 0 \) and \( \mathcal{T}_{u_{N, 2}}''(\tau_N) > 0 \). (Note that \( \mathcal{T} \) is either increasing or decreasing for \( \tau > 1 \).)

We claim that \( \tau_N \to 0 \) as \( N \to 0 \). If this were not the case, there would exist a sequence \( N_m \to 0 \) such that \( \tau_{N_m} \gtrsim 1 \) as \( m \to \infty \). Without loss of generality, we may assume that \( \tau_N \gtrsim 1 \) for all sufficiently small \( N > 0 \).
 We know that $\mathcal{T}_{u_{N, 2}}^{\prime}\left(\tau_N\right)=0$, then \eqref{conh1}, \eqref{V1}, \eqref{V2} and the Pohozaev identity of $(w_{\infty}^1,w_{\infty}^2)$, we have
				$$
				\begin{aligned}
					\mathcal{T}_{u_{N, 2}}^{\prime \prime}(\tau_N) & =\lambda_{N, 2}^{\frac{6-n}{2}}\left(\left\|\nabla w_1^t\right\|_{\lt}^2+ \left\|\nabla w_2^t\right\|_{\lt}^2-\frac{n(n-2)}{8}\tau_{N}^{\frac{n-4}{2}}K(w_1^t,w_2^t) +o_N(1)\right) \\
					&=\lambda_{N, 2}^{\frac{6-n}{2}}\left(\frac{n}{4}-\frac{n(n-2)}{8}\tau_{N}^{\frac{n-4}{2}} +o_N(1)\right)K(w_{\infty}^1,w_{\infty}^2).
				\end{aligned}
				$$
				Since \( \mathcal{T}_{u_{N, 2}}^{\prime\prime}(\tau_N) > 0 \), this implies \( \tau_N < \left( \frac{2}{n-2} \right)^{\frac{2}{n-4}} \leq 1 \). Thus, we can assume that \( \tau_N \to \tau_0 \) as \( N \to \infty \). Then, by \( \mathcal{T}_{u_{N, 2}}^{\prime}(\tau_N) = 0 \), \eqref{conh1}, \eqref{V1}, \eqref{V2}, and the Pohozaev identity for \( (w_{\infty}^1, w_{\infty}^2) \), we must have \( \tau_0 = 0 \), which leads to a contradiction. Therefore, we conclude that \( \tau_N \to 0 \) as \( N \to \infty \).
 Now, since $\mathcal{T}_{u_{N, 2}}^{\prime}\left(\tau_N\right)=0$, then 
				\[
				\begin{split}
					\frac{1}{\tau_{N}^4}  \int_{\rn} V(x) (|u_{N, 2}|^2+|v_{N, 2}|^2)\,d x&=\left\|\nabla u_{N, 2}\right\|_{\lt}^2+ \left\|\nabla v_{N, 2}\right\|_{\lt}^2
					-\frac{n\tau_{N}^{\frac{n-4}{2}}}{4}K(u_{N,2},v_{N,2})
				\end{split}
				\]
				That is, 
				\[
				\begin{split}
					\frac{1}{\tau_{N}^4\lambda_{N,2}^{2}}\int_{\rn} V(x) (|w_{1}^t|^2+|w_2^t|^2)\,d x 
					&=\left\|\nabla w_1^t\right\|_{\lt}^2+ \left\|\nabla w_2^t\right\|_{\lt}^2    -\frac{n\tau_{N}^{\frac{n-4}{2}}}{4}K(w_1^t,w_2^t)
				\end{split}
				\]
				By \eqref{conh1} and \eqref{convh2} we have \begin{equation}\label{tau1}
					\begin{split}
						\frac{1}{\tau_{N}^4\lambda_{N,2}^{2}} &\left(\int_{\rn} V(x) (|w_{\infty}^1|^2+|w_{\infty}^2|^2)\,d x+o_N(1)\right)\\
						&\qquad=\left\|\nabla w_{\infty}^1\right\|_{\lt}^2+ \left\|\nabla w_{\infty}^2\right\|_{\lt}^2+o_N(1).  
					\end{split} 
				\end{equation}
				Applying again \eqref{conh1} and \eqref{convh2}  as $N \rightarrow \infty$,
				$$
				\begin{aligned}
					&\left\|\left(u_{N, 2},v_{N,2})\right)_{\tau_N}\right\|_H^2\\&=\tau_N^2\left(\left\|\nabla u_{N, 2}\right\|_{\lt}^2+ \left\|\nabla v_{N, 2}\right\|_{\lt}^2\right)+\tau_N^{-2} \int_{\rn} V(x) \left(|u_{N, 2}|^2 +|v_{N, 2}|^2\right) \,d x     \\
					&=\tau_N^2\lambda_{N,2}^{\frac{4-n}{2}}\Big(\lambda_{N,2}\left\|\nabla w_{1}^t\right\|_{\lt}^2+ \lambda_{N,2}\left\|\nabla w_2^t\right\|_{\lt}^2+\tau_N^{-4}\lambda_{N,2}^{-1} \int_{\rn} V(x) \left(|w_1^t|^2 +|w_2^t|^2\right)\,d x\Big)\\&=\tau_N^2\lambda_{N,2}^{\frac{4-n}{2}}\Big(\lambda_{N,2}\left(\left\|\nabla w_{\infty}^1\right\|_{\lt}^2+ \left\|\nabla w_{\infty}^2\right\|_{\lt}^2+o_N(1)\right)\\&\quad\quad\quad\quad\quad\left.+\tau_N^{-4}\lambda_{N,2}^{-1}\left( \int_{\rn} V(x) \left(|w_{\infty}^1|^2 +|w_{\infty}^2|^2\right)\,d x+o_N(1)\right) \right),
				\end{aligned}$$
				where $\left(u_{N, 2},v_{N,2}\right)_{\tau_N}=(\tau_N^{\frac{n}{2}} u_{N, 2}\left(\tau_r x\right),\tau_N^{\frac{n}{2}} v_{N, 2}\left(\tau_N x\right))$. Thus,  from \eqref{tau1}, we have $\left(u_{N, 2},v_{N,2}\right)_{\tau_N} \in B_{N \chi, H, 1}$ for a fixed and large $\chi>0$. Since $B_{r \chi, H, 1}$ is connected, we can find a continuous path $\Upsilon$ with $\Upsilon(0)=(u_{N, 1},u_{N, 1})$ and $\Upsilon(1)=\left(u_{N, 2},v_{N, 2}\right)_{\tau_N}$. Now, we define
				$$
				h^{\star \star}[s]=\left\{\begin{array}{l}
					\Upsilon[(2 s)], \quad 0 \leqslant s \leqslant \frac{1}{2}, \\
					\left(u_{N, 2},v_{N, 2}\right)_{2(1-s) \tau_N+(2 s-1) \tau_{N,\star}}, \quad \frac{1}{2} \leqslant s \leqslant 1,
				\end{array}\right.
				$$
				where we choose $\tau_{N, \star}>1$ such that 
				\[
				\mathcal{Y}(h^{\star \star}[1])=\mathcal{T}_{u_{N, 2}}\left(\tau_{N, \star}\right)<\mathcal{Y}\left(u_{N, 1},v_{N, 1}\right)=\mathcal{Y}(h^{\star \star}[0]).
				\]
				Note  for  $(u,v) \in B_{N \chi, H, 1}$ that
				
				$$\begin{aligned}
					\mathcal{Y}(u,v)&= \frac{1}{2}\Big(\|\nabla u\|_{\lt}^2+ \|\nabla v\|_{\lt}^2+\int_{\rn}V(x)(|u|^2+|v|^2)\,d x -K(u,v)\Big) \leq CN^2
				\end{aligned}$$ 
				and again from \eqref{conh1} and \eqref{convh2},   we have for $N>0$ sufficiently small that
    $$\begin{aligned}
					\mathcal{Y}\left(u_{N, 2},v_{N, 2}\right)&= \frac{1}{2}\left(\left\|\nabla u_{N, 2}\right\|_{\lt}^2+ \left\|\nabla v_{N, 2}\right\|_{\lt}^2\right)+\frac{1}{2 } \int_{\rn} V(x) (|u_{N, 2}|^2+|v_{N, 2}|^2)\,d x\\&\quad-\frac{1}{2}K(u_{N,2},v_{N,2}) \\
					&=\lambda_{N,2}^{\frac{6-n}{2}}\left(\frac{1}{2}\left(\left\|\nabla w_1^t\right\|_{\lt}^2+ \left\|\nabla w_2^t\right\|_{\lt}^2\right)-\frac{1}{2}K(w_1^t,w_2^t) +o_N(1) \right)\\&=\lambda_{N,2}^{\frac{6-n}{2}}\left(\frac{1}{2}\left(\left\|\nabla w_{\infty}^1\right\|_{\lt}^2+ \left\|\nabla w_{\infty}^2\right\|_{\lt}^2\right)-\frac{1}{2}K(w_{\infty}^1,w_{\infty}^2) +o_N(1) \right)\\
					&\geq C.
				\end{aligned} $$ 
				Thus, for $N>0$ sufficiently small, $h^{\star \star }[s] \in \Gamma_N$ and
				$$
				\alpha(N) \leqslant \max _{0 \leqslant s \leqslant 1} \mathcal{Y}( h^{\star \star}[s])\leq \mathcal{Y}\left(u_{N, 2},v_{N, 2}\right)=\mathcal{T}_{u_{N, 2}}(1)=\mathfrak{m}\left(\lambda_{N, 2}\right) \lambda_{N, 2}^{\frac{6-n}{2}}-\frac{\lambda_{N, 2} N^2}{2} .
				$$
				Hence,
\[
\alpha(N) \leqslant \lambda_{N, 2}^{\frac{6-n}{2}} \mathfrak{m}\left(\lambda_{N, 2}\right) - \frac{\lambda_{N, 2} N^2}{2}.
\]
Therefore, from \eqref{primeine},
\[
\alpha(N) = \lambda_{N, 2}^{\frac{6-n}{2}} \mathfrak{m}\left(\lambda_{N, 2}\right) - \frac{\lambda_{N, 2} N^2}{2}.
\]
From this, we conclude that \( (u_{N, 2}, v_{N, 2}) \) is a mountain-pass solution of \eqref{2elliptic} for \( N > 0 \) sufficiently small.
			\end{proof}

			\section{Long time behavior of solutions of \eqref{system1}}\label{localw}

		In this  section, our aim is to establish criteria regarding the dichotomy between the existence of blow-up and the global existence of solutions. 
		
		Let us consider the following system \eqref{systemq2}. A study conducted in \cite{acs,pas1,lopes,ZhaoZhaoShi} allows us to establish the following result.
		
		\begin{proposition}\label{Gprop}
			The functional	 
			\begin{equation}\label{GJ}
				\mathscr{J}(u_1,u_2)=\frac{\left(\int_{\rn} (|\nabla u_1|^2+\kappa|\nabla u_2|^2)\,d x\right)^{\frac{3}{2}-\frac{n}{4}}\left(\mq(u_1,u_2)\right)^{\frac{n}{4}}}{K(u_1,u_2)}
			\end{equation}
			is minimized by $\frac{n^{\frac{n}{4}}(6-n)^{1-\frac{n}{4}}}{2}  \mq(Q_1,Q_2)^{1/2}$, where $(Q_1,Q_2)$ is  a   positive radially symmetric solution of \eqref{systemq2}.
		\end{proposition}
		\begin{remark}
		    It is known that the minimizer of \eqref{GJ} is unique if $\kappa=2$.
		\end{remark}
Additionally, from \eqref{systemq2}, we have the following lemma.
		\begin{lemma}\label{Glem}
			\blue{Let $(u_1,u_2) \in H^1\left(\rn\right)\times H^1(\rr^n)$ be a solution of system} \eqref{systemq2}. Then one has
			$$
			\begin{aligned}
				\int_{\rn}\left(|\nabla u_1|^2+ \kappa|\nabla u_2|^2\right)\,d x &=\frac{n}{4}K(u_1,u_2),\\
				\int_{\rn}\left(|u_1|^2+ 2|u_2|^2\right)\,d x &=\frac{6-n}{4}K(u_1,u_2).
			\end{aligned}
			$$
		\end{lemma}
		
		\begin{proof}
			Using the Pohozaev identity, we have
			$$
			\begin{aligned}
				0= (n-2) \int_{\rn}\paar{|\nabla u|^{2}+\kappa|\nabla v|^{2}} d x -nK(u,v)+n\int_{\rn}|u|^{2}\,d x+n\int_{\rn}2| v|^{2}\,d x.
			\end{aligned}
			$$
			\blue{Moreover}, using \eqref{systemq2} with $(u_1, u_2)$, we find
			
			$$
			\begin{aligned}
				& \int_{\rn}|u_1|^{2}\,d x=K(u_1,u_2)-\int_{\rn}|\nabla u_1|^{2} \,d x, \\
				&2\int_{\rn}|u_2|^{2} \,d x=\frac{1}{2}K(u_1,u_2)-\int_{\rn}\kappa|\nabla v|^{2}\,d x.
			\end{aligned}
			$$
			Then, we have 
			$$ \int_{\rn}\paar{|\nabla u_1|^2+\kappa|\nabla u_2|^2}\,d x=\frac{n}{4}K(u_1,u_2)
			$$
			and 
			$$ \mq(u_1,u_2) =\frac{6-n}{4}K(u_1,u_2).$$
		\end{proof}
		The following lemma is a corollary of Proposition \ref{Gprop}.
		\begin{lemma}
			 \blue{ Let $1 \leq n \leq 5$}. Then, for any $(u_1,u_2) \in H^1\left(\rn\right)\times H^1\left(\rn\right)$, we have
			$$
			\begin{aligned}
				K(u_1,u_2) & \leq \frac{2}{n^{\frac{n}{4}}(6-n)^{1-\frac{n}{4}}}\left( \int_{\rn} \paar{\abs{\nabla u_1}^2+ \kappa\abs{\nabla u_2}^2} \,d x\right)^{\frac{3}{2}-\frac{n}{4}} \\
				&\qquad
				\times\paar{ 
				 \mq(u_1,u_2)} ^{\frac{n}{4}}\mq(Q_1,Q_2)^{-\frac{1}{2}},
			\end{aligned}
			$$
			where $(Q_1,Q_2)$ is the ground state solution of \eqref{systemq2}.
		\end{lemma}
		
		Finally, we recall the following classical inequality.
		
		\begin{lemma}[\cite{strauss}]\label{Gsim}
			Let $n \geq 3$ and   $f \in \dot{H}^1\left(\rn\right)$ be radially symmetric. Then, there exists a positive constant $C=C(n)$ such that
			$$
			\sup _{x \in \rn}|x|^{\frac{n-2}{2}}|f(x)| \leq C\|\nabla f\|_{\lt} .
			$$
		\end{lemma}

		Now, we are in a position to prove the blow-up result. 
		\begin{proof}[Proof of Theorem \ref{Theoremblow1}]
			For arbitrary $\lambda>0, \mu>0$, we put $Q_{\lambda, \mu}^j(x)=\mu \lambda^{n/2} Q_j(\lambda x)$ for $j=1,2$. By scaling argument, we have that
			\begin{equation}\label{Gscat}
				\begin{gathered}
					\int\left[Q_{\lambda, \mu}^j\right]^2 \,d x=\mu^2 \int Q_j^2 \,d x, \\
					\int\left[Q_{\lambda, \mu}^1\right]^2 Q_{\lambda, \mu}^2\,d x=\mu^3 \lambda^{n/2} K(Q_1,Q_2), \\
					\int\left[\nabla Q_{\lambda, \mu}^j\right]^2 \,d x=\mu^2 \lambda^{2} \int|\nabla Q_j|^2\,d x, \\
					\int V(x)\left[Q_{\lambda, \mu}^j(x)\right]^2 \,d x=\lambda^{-2} \mu^2 \int V(x) Q_j^2 \,d x .
				\end{gathered}    
			\end{equation}
			Now, we take  
			\begin{equation}\label{Gmulambda}
				\begin{gathered}
					\mu=\left[\left(\mq(Q_1,Q_2)+\epsilon\right) /\mq(Q_1,Q_2)\right]^{\frac{1}{2}}, \\
					\lambda<\left(\frac{\int_{\rn}V(x)\left( Q_1^2+Q_2^2\right)\,d x}{\left(1-\frac{n}{4}\mu\lambda^{\frac{n-4}{2}}\right)\int_{\rn}(|\nabla Q_1|^2+\kappa|\nabla Q_2|^2)\,d x}\right)^{\frac{1}{4}}, \quad\varphi_j(x)=\mu \lambda^{n/2} Q_j(\lambda x) .
				\end{gathered}   
			\end{equation}
			It is obvious that \( (u_1, u_2) \in H \). Moreover, from \eqref{Gscat} and \eqref{Gmulambda}, it follows that
\[
\mq(u_1,u_2) =\mq(Q_1,Q_2) + \epsilon.
\]
From the definition of \( E \) in \eqref{E}, and using \eqref{Glem}, \eqref{Gscat} and \eqref{Gmulambda}, one has that
			\blue{$$
			\begin{aligned}
				E\left(u_1,u_2\right)&=\|\nabla u_1\|_\lt^2+\kappa\|\nabla u_1\|_\lt^2
				+\int_\rn V(x)(|u_1|^2+|u_1|^2)\,d x-\Re\int_\rn u_1^2\overline{u}_1\,d x\\&=\mu^2\lambda^2 \left(\|\nabla Q_1\|_\lt^2+\kappa\|\nabla Q_2\|_\lt^2\right)+\mu^2\lambda^{-2}\int_\rn V(x)(|Q_{1}|^2+|Q_{2}|^2)\,d x\\
				&\quad-\mu^3 \lambda^{n/2} \int Q_1^2Q_2 \,d x\\&=\mu^2 \lambda^2 \int_{\rn}\left[|\nabla Q_1|^2+\kappa|\nabla Q_2|^2+\lambda^{-4} V(x) (Q_1^2+Q_2^2)- \frac{4\mu}{n}\lambda^{\frac{n-4}{2}} (|\nabla Q_1|^2+\kappa|\nabla Q_2|^2))\right] \,d x \\
				& =\mu^2 \lambda^2 \int_{\rn}\left[\left(1-\frac{4\mu}{n}\lambda^{\frac{n-4}{2}}\right)(|\nabla Q_1|^2+\kappa|\nabla Q_2|^2)+\lambda^{-4}V(x) (Q_1^2+Q_2^2)\right] \,d x 
				<0.
			\end{aligned}
			$$ }
						
			Now,  we consider $$\mathbb{V}(t)= \int_\rn
			|x|^2\paar{|u_1|^2+\frac1\kappa|u_2|^2
			}\,d x.$$
			Then, 
			\begin{equation}\label{virialidentity1}
				\begin{split}
					\frac{\,d^2}{\,d t^2} \int_\rn
					|x|^2&\paar{|u_1|^2+\frac1\kappa|u_2|^2
					}\,d x\\
					&
					=\frac{\,d}{\,d t}\paar{
						4\Im\int_\rn x\cdot(\bar u_1\nabla u_1+\bar u_2\nabla u_2)\,d x
						- (1-\frac1{2\kappa})\Im\int_\rn|x|^2u_2\overline{u}_1^2\,d x
					}
					\\
					&
					=
					8N(u_1,u_2)- (1-\frac1{2\kappa})\frac{\,d}{\,d t}\Im \int_\rn|x|^2u_2\overline{u}_1^2\,d x\\
					&=8E(u_1^0,u_2^0)+ 2(4-n)K(u_1,u_2)-16\int_\rn V(x)(|u_1|^2+|u_2|^2)\,d x\\
					&\quad -  (1-\frac1{2\kappa})\frac{\,d}{\,d t}\Im \int_\rn|x|^2 u_2\overline{u}_1^2\,d x,
				\end{split}
			\end{equation}
			where
			\[
			N(u_1,u_2)=\|\nabla u_1\|_\lt^2+\kappa\|\nabla u_2\|_\lt^2-\frac n4 K(u_1,u_2)-\int_\rn V(x)(|u_1|^2+|u_2|^2)\,d x.
			\]
			If we assume that \( n > 4 \) and \( \kappa = \frac{1}{2} \), then we have
\begin{equation}\label{GVE}
\mathbb{V}^{\prime \prime}(t) \leq 8E(u_1^0, u_2^0) + 2(4 - n) K(u_1, u_2) - 16 \int_{\mathbb{R}^n} V(x)(|u_1|^2 + |u_2|^2) \,d x < 0.
\end{equation}
Since \( \mathbb{V}(0) > 0 \),
\[
\begin{aligned}
\mathbb{V}(t) &= \mathbb{V}(0) + \mathbb{V}^{\prime}(0) t + \int_0^t (t - s) \mathbb{V}^{\prime \prime}(s) \, ds, \quad 0 \leq t < \infty, \\
\mathbb{V}(t) &\leq \mathbb{V}(0) + \mathbb{V}^{\prime}(0) t + 4 E(u_1, u_2) t^2,
\end{aligned}
\]
and it follows from \eqref{GVE} that
\[
\mathbb{V}(t) \leq \mathbb{V}(0) + \mathbb{V}^{\prime}(0) t + 4 E(u_1, u_2) t^2.
\]
Thus, there exists \( T > 0 \) such that
\[
\lim_{t \to T^-} \mathbb{V}(t) = 0.
\]
Lemma \ref{Glem} and Proposition \ref{Gprop} imply that
\[
\lim_{t \to T^-} \int_{\mathbb{R}^n} (|\nabla u_1|^2 + |\nabla u_2|^2) \,d x = \infty.
\]

			If $n=4$, using  \eqref{virialidentity1}, it suffices to carry out the same process, demonstrating that $E(u_1^0,u_2^0)<0$.
		\end{proof}
		\begin{remark}
			It is worth noting that the previous theorem was formulated for a general potential $V(x)$ as long as $\kappa=\frac{1}{2}$. However, when the interaction potential is partial, that is, $V_2(x)$, we can obtain a similar result. At the moment, it has not been possible to obtain a result in general when the potential is harmonic and $\kappa\neq \frac{1}{2}$.
		\end{remark}
		Let us consider the following operator 
		\[ 
		N_1(t):=N_{1}(u(t),v(t))=\|\partial_{x_n} u\|_\lt^2+\kappa\|\partial_{x_n} v\|_\lt^2-\frac 14K(u ,v) 
		\]
		\blue{and   we define  for the ground state $(\varphi_1,\varphi_2)$  of \eqref{system1},}
		\[
		M=\sett{(u,v)\in H: I(u,v)<I(\varphi_1,\varphi_2) \,\, \text{and}\, \, N_{1}(u,v)<0}.
		\]  
		\begin{theorem}\label{thm4.8}
			Let $4\leq n$. If that $\left(u_0, v_0\right) \in M$, then the solution $(u,v)$ of \eqref{system1} with the initial data $\left(u_0, v_0\right) $ blows up in both time directions. 
		\end{theorem} 
		\begin{proof}
			Assume that $\left(u_0, v_0\right)  \in M$. Now, suppose there exists $t_0>0$ such that $N_1(t_0)>0$. Then, since the mapping $t  \rightarrow N_1(t)$ is continuous, it follows that there exists $t_1 \in (0, t_0)$ such that $N_1(t_1)=0$ and $N_1(t)>0$ for all $t\in (t_1, t_0)$. 
			
			Now, note that if $(u(t_1),v(t_1))\neq (0,0)$, then $I(u(t_1),v(t_1))\leq I(\varphi_1,\varphi_2)$ (see Remark \ref{remarktheorem2}). But, since the energy and the mass are conserved we have $I(u(t_1),v(t_1))<I(\varphi_1,\varphi_2)$ which is a contradiction. Thus, $N_1(t)<0$ for ant $t$ in
			the existence interval.  
			We take a smooth function $\chi_1: [0,\infty)\rightarrow [0,\infty)$ such that
			\[
			\chi_1(r):= \begin{cases}
				{ r ^ { 2 } , } & { \text { if } 0 \leq r \leq 1 , } \\
				{ \text { positive constant, } } & { \text { if } r \geq 2 , }
			\end{cases} 
			\qquad \chi _ { 1 } ^ { \prime } ( r )  
			\begin{cases}
				=2 r, & \text { if } 0 \leq r \leq 1, \\
				\leq 2 r, & \text { if } 1 \leq r \leq 2, \\
				=0, & \text { if } r \geq 2,
			\end{cases} 
			\]
			and
			$ 
			\chi_1^{\prime \prime}(r) \leq 2.
			$ 			
			For $R>0$, we take $\chi: \rn \rightarrow [0,\infty)$ such that
			\[
			\chi(x)=R^2 \chi_1\left(\frac{|x_n|}{R}\right) 
			\]
			and consider 
			\[
			\mathbb{V}(t)= \int_\rn \chi(x)
			\paar{|u|^2+\frac1\kappa|v|^2
			}\,d x.
			\]
			Then, 
			\begin{equation}\label{G1}
				\begin{aligned}
					\mathbb{V}'(t) 
					& =2 \int \partial_{x_n} \chi(x) \cdot \Im\left[\overline{u} \partial_{x_n} u+ \overline{v} \partial_{x_n}v\right](x) d x.
				\end{aligned}   
			\end{equation}
			Notice that, $$\begin{aligned}
				\mathbb{V}''(t) &=-4\Im\int_{\rn}\partial_{x_n} \chi(x) \partial_{x_n} \overline{u}\,\partial_{t}u\,d x-2\Im\int_{\rn}\partial_{x_n}^2 \chi(x) \overline{u}\,\partial_{t}u\,d x\\
				&\quad-4\Im\int_{\rn}\partial_{x_n} \chi(x) \partial_{x_n} \overline{v}\,\partial_{t}v\,d x-2\Im\int_{\rn}\partial_{x_n}^2 \chi(x) \overline{v}\,\partial_{t}v\,d x\\&=-4\Im\int_{\rn}\partial_{x_n} \chi(x) \partial_{x_n} \overline{u}(i\Delta u-iV(x)u+i\overline{u}v)\,d x-2\Im\int_{\rn}\partial_{x_n}^2 \chi(x) \overline{u}(i\Delta u-iV(x)u+i\overline{u}v)\,d x\\
				&\quad-4\Im\int_{\rn}\partial_{x_n} \chi(x) \partial_{x_n} \overline{v}(i\Delta v-iV(x)v+\frac{i}{2}u^2)\,d x-2\Im\int_{\rn}\partial_{x_n}^2 \chi(x) \overline{v}(i\Delta v-iV(x)v+\frac{i}{2}u^2)\,d x\\
				&=2\Re\int_{\rn}\partial_{x_n}^2 \chi(x) (|\partial_{x^{\prime}} u|^2+\kappa|\partial_{x^{\prime}} v|^2)\,d x-\Re\int_{\rn}\partial_{x_n}^4 \chi(x) (|u|^2+\kappa| v|^2)\,d x\\&\quad+ 2\Re\int_{\rn}\partial_{x_n}^2 \chi(x) (|\partial_{x_n} u|^2+\kappa|\partial_{x_n} v|^2)\,d x+2\Re\int_{\rn}\partial_{x_n}^2 \chi(x)V_2(x) (| u|^2+| v|^2)\,d x\\&\quad-2\Re\int_{\rn}\partial_{x_n}^2 \chi(x)u^2\overline{v}\,d x-2\Re\int_{\rn}\partial_{x_n}^2 \chi(x) (|\partial_{x^{\prime}} u|^2+\kappa|\partial_{x^{\prime}} v|^2)\,d x\\
				&\quad +2\Re\int_{\rn}\partial_{x_n}^2 \chi(x) (|\partial_{x_n} u|^2+\kappa|\partial_{x_n} v|^2)\,d x-2\Re\int_{\rn}\partial_{x_n}^2 \chi(x)V_2(x) (| u|^2+| v|^2)\,d x\\&\quad +2\Re\int_{\rn}\partial_{x_n}^2 \chi(x)u^2\overline{v}\,d x+2\Re\int_{\rn}\partial_{x_n} \chi(x)u^2\partial_{x_n}\overline{v}\,d x-2\Re\int_{\rn}\partial_{x_n} \chi(x)u^2\partial_{x_n}\overline{v}\,d x\\&\quad-\Re\int_{\rn}\partial_{x_n}^2 \chi(x)u^2\overline{v}\,d x.
			\end{aligned}$$
			Therefore, 
			\begin{equation}\label{G2}
				\begin{aligned}
					\mathbb{V}^{\prime \prime}(t)= &  4 \Re \int_{\rn} \partial_{x_n}^2\chi(x)(|\partial_{x_n} u|^2 +\kappa|\partial_{x_n} v|^2)d x-\Re\int_{\rn}\partial_{x_n}^4\chi(x)\left(|u|^2+|v|^2\right)\,d x \\
					& -\Re\int_{\rn}\partial_{x_n}^2 \chi(x) v \overline{u}^2\,d x.
				\end{aligned}    
			\end{equation}
			From  \eqref{G2}, we get
			$$
			\mathbb{V}^{\prime \prime}(t)=8 N_1(t)+\mathcal{A}_1+\mathcal{A}_2+\mathcal{A}_3,
			$$
			where $\mathcal{A}_1, \mathcal{A}_2, \mathcal{A}_3$ are defined by
			$$
			\begin{aligned}
				\mathcal{A}_1:= & 4 \int_{\rn}\left(\chi_1^{\prime \prime}\left(\frac{|x_n|}{R}\right) -2\right)\left(|\partial_{x_n} u|^2+\kappa|\partial_{x_n} v|^2\right) d x, \\
				\mathcal{A}_2:= & -\int_{\rn} \partial_{x_n}^4 \chi(x)\left(|u|^2+\kappa|v|^2\right) d x, \\
				\mathcal{A}_3:= & -\Re\int_{\rn}\left\{\chi_1^{\prime \prime}\left(\frac{|x_n|}{R}\right)-2 \right\} v \overline{u}^2 d x.
			\end{aligned}
			$$
			The fact $\chi_1^{\prime \prime}(r) \leq 2$ implies that
			$$
			\mathcal{A}_1=4 \int_{\rn}\left(\chi_1^{\prime \prime}\left(\frac{|x|}{R}\right)-2\right)\left(|\partial_{x_n} u|^2+\kappa|\partial_{x_n} v|^2\right)\,d x \leq 0.
			$$
			On the other  hand, $\chi_1^{\prime}(r)=2 r, \chi_1^{\prime \prime}(r)=2$, and $\chi_1^{(3)}(r)=\chi_1^{(4)}(r)=0$ on $r \leq 1$, reveal that
			$$
			\begin{aligned}
				\partial_{x_n}^4 \chi(x)=  & 0 \text { on }|x_n| \leq R 
			\end{aligned}
			$$
			and $$
			\begin{aligned}
				\partial_{x_n}^4 \chi(x)=R^{-2}\chi_1^{(4)}\left(\frac{|x_n|}{R}\right)  &  \text { on }|x_n| \geq R .
			\end{aligned}
			$$
			Therefore, $\mathcal{A}_2$ is estimated as follows:
			$$
			\begin{aligned}
				\mathcal{A}_2 & = \int_{|x| \geq R}\partial_{x_n}^4 \chi(x)\left(|u|^2+\kappa|v|^2\right)\,d x \\
				& \leq C \int_{|x| \geq R} R^{-2}\left(|u|^2+\kappa|v|^2\right)\,d x \\
				& \leq R^{-2} C_\kappa M(u, v) .
			\end{aligned}
			$$
			On the other hand, notice that 
			$$
			\chi_1^{\prime \prime}\left(\frac{|x_n|}{R}\right)=2 \text { on }|x_n| \leq R,
			$$
			so			it follows from  Lemma \ref{Gsim}, that
			$$
			\begin{aligned}
				\mathcal{A}_3 & =-\Re\int_{|x_n| \geq R}\left\{\chi_1^{\prime \prime}\left(\frac{|x_n|}{R}\right)-2 \right\} v \overline{u}^2\,d x \\
				& \leq C \int_{|x| \geq R}|v||u|^2 d x\\
				& \leq CR^{-\frac{n-2}{2}} \int_{|x| \geq R}|x|^{\frac{n-2}{2}}|v \| u|^2 d x.
			\end{aligned}
			$$
			Hence, $$\begin{aligned}
				\mathcal{A}_3 &\leq    C R^{-\frac{n-2}{2}}\|v\|_{L^2}\|u\|_{L^2}\left\||x|^{\frac{n-2}{2}} u\right\|_{L^{\infty}} \\& \leq C R^{-\frac{n-2}{2}}\|v\|_{L^2}\|u\|_{L^2}\|\nabla u\|_{L^2} \\
				& \leq R^{-\frac{n-2}{2}} C M(u, v)\|\nabla u\|_{L^2}.\end{aligned}$$
			By combining these estimates, we get
			\begin{equation*}
				\mathbb{V}^{\prime \prime}(t) \leq 8 N_1(t)+R^{-2} C_\kappa M(u, v)+R^{-\frac{n-2}{2}} C M(u, v)\|\nabla u\|_{L^2}.    
			\end{equation*}
			Now, since $N_1(t)<0$, then \begin{equation}\label{GJprime}
				\mathbb{V}^{\prime \prime}(t) \leq -\delta+ R^{-2} C_\kappa M(u, v)+R^{-\frac{n-2}{2}} C M(u, v)\|\nabla u\|_{L^2}.    
			\end{equation}				
			By taking sufficiently large $R>0$, we get
			$$
			\mathbb{V}^{\prime \prime}(t) \leq-2 \delta<0.
			$$
			Therefore, by following a process similar to that used in Theorem \ref{Theoremblow1}, the desired result follows.
		\end{proof}
		\begin{remark}\label{remarktheorem2}
			Notice that  if $(u_1,u_2)\neq (0,0)$ and $N_1(t)=0$, then $I(u_1,u_2)\leq I(\varphi_1,\varphi_2)$. In fact, consider $(u_1^{s},u_2^s)=(s^{1/2}u_1(x^{\prime},sx_n),s^{1/2}u_2(x^{\prime},sx_n))$. Thus, $$P(u_1^{s},u_2^s)=s^2(\|\partial_{x_n} u_1\|_\lt^2+\kappa\|\partial_{x_n} u_2\|_\lt^2)-s^{1/2}K(u_1,u_2)+ M_{u_1,u_2},$$
			where $P(u_1,u_2)=\langle I^{\prime}(u_1,u_2),(u_1,u_2)\rangle$  and  
			\[
			M_{u_1,u_2}=\|\partial_{x^{\prime}} u_1\|_\lt^2+\|\partial_{x^{\prime}} u_2\|_\lt^2+\int_{\rn}V_2(x)(|u_1|^2+|u_1|^2)\,d x.
			\]
			Then, since $N_1(t)=0$, it follows that $$P(u_1^{s},u_2^s)=\left(\frac{s^2}{4}-s^{1/2}\right)K(u_1,u_2)+ M_{u_1,u_2}.$$
			Therefore, there exists $s_0\in(0,\infty)$ such that $P(u_1^{s_0},u_2^{s_0})=0$. This implies that \[
			I(u_1^{s_0},u_2^{s_0})\geq I(\varphi_1,\varphi_2).
			\] Moreover, since $\partial_{s}I(u_1^{s},u_2^s)\big|_{s=1}=2P(u_1,u_2)=0$, the desired result follows.
		\end{remark}
		Now we give a threshold for the global existence of the Cauchy problem \eqref{system1}.
		
		\begin{proof}[Proof of Theorem \ref{thm410}]
			Let \blue{$(u_1(x,t), u_2(x,t))\in C((0, T) : H)$} be the solution of the Cauchy problem \eqref{system1} with initial data $(u_1^0,u_2^0) \in H$. From  Lemma \ref{Glem} and \eqref{Gscat} we get 
			$$
			\begin{aligned}
				E(u_1,u_2) &\geq \left(1-\frac{4  \mq(u_1,u_2) }{n\mq(Q_1,Q_2)}\right)\left[\int_{\rn}(|\nabla u_1|^2+|\nabla u_2|^2)\,d x\right] \\
				&\qquad+\int_{\rn}V(x)(|u_1|^2+|u_2|^2)\,d x .
			\end{aligned}
			$$
			Since $\mq(u_1,u_2)=\mq(u_1^0,u_2^0)$ and $E(u_1,u_2)=E(u_1^0,u_2^0),$ then  we obtain from \eqref{GINQ} that  
			\[
			\int_{\rn}(|\nabla u_1|^2+|\nabla u_2|^2)\,d x  
			\quad\text{and}\quad
			\int_{\rn}V(x)(|u_1|^2+|u_2|^2)\,d x 
			\]
			are bounded for $t \in[0, T)$ and any $T<\infty$. Hence,   it yields from Theorem  \ref{theoremsub}
			that \blue{$(u_1(x,t), u_2(x,t))$}  globally exists in time.
		\end{proof}
		\blue{\begin{remark}
				Although Theorem \ref{thm410} guarantees that the local solution extends globally in time, the solution may still increase as time progresses.
		\end{remark}}
		%
		%
		%
		%
		%
		%
		%
		%
		%
		\subsection*{Acknowledgment}
		The authors would like to thank the unknown referees for their valuable suggestions and corrections
		which helped to improve the paper

		The authors are supported by Nazarbayev University under the Faculty Development Competitive Research Grants Program for 2023-2025 (grant number 20122022FD4121). 
		
		\subsection*{Conflict of interest} The authors declare that they have no conflict of interest. 
		
		\subsection*{Data Availability}
		Data sharing is not applicable to this article as no new data were created or analyzed in this study.

	\end{document}